\documentclass{bmcart}

\usepackage[utf8]{inputenc} 

\usepackage{amsmath,amsfonts,amssymb,amsthm,stmaryrd,mathrsfs,mathdots,amsbsy,float,cmll,xcolor,mathpartir,colonequals}
\usepackage{hyperref}
\usepackage{bm}
\usepackage{tikz}
\usetikzlibrary{tikzmark}
\usepackage[small]{diagrams}
\usepackage{float}
\usetikzlibrary{positioning,calc}
\usepackage{tikz-cd}
\usepackage{bussproof}

\theoremstyle{definition}
\newtheorem{definition}{Definition}[section]
\newtheorem{example}[definition]{Example}
\theoremstyle{theorem}
\newtheorem{theorem}[definition]{Theorem}
\newtheorem*{theorem*}{Theorem}
\newtheorem*{corollary*}{Corollary}
\newtheorem{proposition}[definition]{Proposition}
\newtheorem{lemma}[definition]{Lemma}
\newtheorem{corollary}[definition]{Corollary}
\theoremstyle{remark}
\newtheorem*{remark}{Remark}
\newtheorem*{notation}{Notation}
\newtheorem*{convention}{Convention}

\newarrow{Corresponds}<--->
\newarrow{TeXto}----{->}

\def\includegraphics{}

\startlocaldefs
\endlocaldefs

\begin{document}

\begin{frontmatter}

\begin{fmbox}
\dochead{Research}


\title{Game semantics of Martin-L\"of type theory}


\author[
   addressref={aff1},                   
   corref={aff1},                       
   email={yamad041@umn.edu}   
]{\inits{NY}\fnm{Norihiro} \snm{Yamada}} \\
This is a preprint submitted to \emph{Research in the Mathematical Sciences (RMS).}


\address[id=aff1]{
  \orgname{University of Minnesota}, 
  \city{Minneapolis},                              
  \postcode{MN 55455},                                
  \cny{USA}                                    
}

\end{fmbox}


\begin{abstractbox}

\begin{abstract} 
We present new \emph{game semantics} of \emph{Martin-L\"{o}f type theory (MLTT)} equipped with One-, Zero-, N-, Pi-, Sigma- and Id-types. 
Our game semantics interprets MLTT \emph{more accurately} than existing ones. 
Another advantage of our game semantics over existing ones is its interpretation of Sigma-types that is \emph{direct} and \emph{compatible with the game semantics of product types}.
Besides, its mathematical structure is novel and useful; e.g., the category of our games has all \emph{finite limits}, which is a key step to an extension of the present work to \emph{homotopy type theory}, and our games interpret \emph{subtyping on dependent types} for the first time as game semantics.
Finally, we provide a new, game-semantic proof of the \emph{independence of Markov's principle} from MLTT, which demonstrates an advantage of our game semantics over extensional models of MLTT such as the effective topos.
\end{abstract}


\begin{keyword}
\kwd{game semantics}
\kwd{Martin-L\"{o}f type theory}
\kwd{constructive mathematics}

\end{keyword}

\end{abstractbox}
%

\end{frontmatter}





\tableofcontents

\section{Introduction}
\subsection{Martin-L\"{o}f type theory and the meaning explanation}
\label{MLTT}
\emph{Martin-L\"{o}f type theory (MLTT)} \cite{martin1975intuitionistic,martin1998intuitionistic} is one of the best-known formal systems for \emph{constructive mathematics} \cite{troelstra1988constructivism}, which is comparable to set theory \cite{zermelo1908untersuchungen,fraenkel1922grundlagen} for classical mathematics.
MLTT is also a programming language \cite{martin1982constructive} that is a generalisation of the \emph{simply-typed lambda-calculus (STLC)} \cite{church1940formulation} along the generalisation of (intuitionistic) propositional logic to predicate logic under the \emph{Curry-Howard isomorphisms} \cite{sorensen2006lectures}. 
By this computational nature, MLTT and similar formal systems enable computer formalisations of mathematics and its applications to programming \cite{constable1986implementing,hottbook}. 

Like set theory is explained informally by \emph{sets}, the conceptual foundation of MLTT is \emph{computations} in an informal sense. 
That is, the fundamental idea of MLTT is to regard objects and proofs in constructive mathematics uniformly as computations, and MLTT is a syntactic formalisation of this foundational idea \cite{martin1982constructive}.
Hence, objects and proofs in MLTT are unified into \emph{terms}, where formulas are called \emph{types}.
This standard, informal semantics of MLTT is called the \emph{meaning explanation} \cite[\S 5]{dybjer2016intuitionistic}.

However, MLTT is not always the best formalisation of this conceptual foundation of constructive mathematics since it is an intricate formal system that inevitably contains superficial syntactic details.\footnote{A syntactic formalisation is also unsatisfactory from the \emph{syntax-first-view}, i.e., the view that semantic concepts come first, and syntax merely provides notations.}
In other words, the intuition behind MLTT is often blurred by the complexity and the syntactic nature of MLTT. 
In addition, the syntactic complexity makes it difficult to study the meta-theory of MLTT.

Accordingly, \emph{mathematical semantics} \cite{amadio1998domains} of MLTT that faithfully formalises the meaning explanation is strongly desired since such semantics would  accurately and directly (or non-inductively) describe the intuition behind MLTT, abstracting the syntactic details.
It would not only deepen our understanding of MLTT in this way but also suggest improvements and extensions of MLTT like \emph{coherence spaces} by Girard \cite{girard1987linear} led to \emph{linear logic} \cite{girard1987linear}, and the \emph{groupoid model} by Hofmann and Streicher \cite{hofmann1998groupoid} to \emph{homotopy type theory (HoTT)} \cite{hottbook}.
Besides, mathematical semantics has been highly effective for the meta-theoretic study of MLTT; e.g., see \cite{streicher2012semantics}.

\subsection{Game semantics}
\label{IntroGameSemantics}
\emph{Game semantics} \cite{abramsky1997semantics,hyland1997game} is a particular class of mathematical semantics of logic and computation that models types and terms by \emph{games} and \emph{strategies}, respectively. 

Its strong point is its \emph{conceptual naturality}: `Logic is the study of reasoning' \cite[p.~1]{shoenfield1967mathematical}, where one can regard `reasoning' as \emph{dialogical arguments} between \emph{Player} (or a mathematician) and \emph{Opponent} (or an oracle), and game semantics formalises this intuition.
This game-semantic view on logic is also in harmony with the meaning explanation since dialogical arguments are a certain kind of \emph{computations}.

Another strong advantage of game semantics is its \emph{precision} in modelling syntax as various \emph{full completeness/abstraction} results \cite{curien2007definability} in the literature demonstrate. 
This precision is due to its \emph{intensionality}: Game semantics captures the \emph{processes themselves} underlying terms rather than their extensions such as functions.

Finally, the concrete nature of game semantics enables its \emph{algorithmic applications} to program analysis and verification (i.e., meta-theories of programs) \cite{abramsky2002algorithmic}.

\subsection{Main results}
\label{MainResults}
To summarise the points so far, mathematical semantics of MLTT that advances our understanding of MLTT, promotes its improvements and extensions, and/or clarifies its meta-theory is strongly desired, and game semantics seems perfect for this role by its conceptual naturality, harmony with the meaning explanation, precision in modelling syntax and algorithmic applications. 
Also, game semantics models \emph{effects} \cite{plotkin2004computational} and \emph{linear logic} \cite{girard1987linear} in a highly systematic way (see \cite{abramsky1999game} for the details); thus, game semantics of MLTT, if any, may lead to MLTT with effects and linear typing.

However, although game semantics of various logics and computations has been given, it is quite difficult to establish game semantics of MLTT. 
The main challenge in game semantics of MLTT is to model the \emph{extensional} type dependency in Sigma-types by \emph{intensional} processes in games; see the beginning of \S\ref{PredicateGames}.
In fact, this problem had been open for more than twenty years, and even today its definitive solution is yet to emerge though a few candidates have arisen recently \cite{abramsky2015games,blot2018extensional}; see \S\ref{RelatedWorkAndOurContributions}.

Hence, we aim to provide another candidate for game semantics of MLTT with the hope that it sheds new light on this problem. 
Motivated in this way, we prove:
\begin{theorem*}[Game semantics of MLTT]
\label{MainTheorem}
There is new game semantics of MLTT with One-, Zero-, N-, Pi-, Sigma- and Id-types (\S\ref{GameSemanticCwF}--\ref{GameSemanticTypeFormers}). 
\end{theorem*}

Our key idea for this theorem is to generalise games so that they can model Sigma-types, while we keep strategies \emph{unchanged} so that we retain the advantages of game semantics such as \emph{intensionality} (\S\ref{IntroGameSemantics}).
See the beginning of \S\ref{PredicateGames} for the outline of this idea.
Also, see \S\ref{RelatedWorkAndOurContributions} for the advantages of this method over existing ones \cite{abramsky2015games,blot2018extensional}. 

We also illustrate the \emph{utility} of our game semantics by giving a new proof of:
\begin{corollary*}[Independence of Markov's principle \cite{mannaa2017independence}]
\label{MainCorollary}
Markov's principle \cite{markov1962constructive} is independent from MLTT equipped with the types listed in Theorem~\ref{MainTheorem} (\S\ref{Independence}).
\end{corollary*}

The method used by Mannaa and Coquand \cite{mannaa2017independence} is \emph{syntactic}, while our \emph{semantic} approach provides a new, intuitive argument on why the independence holds.
This corollary also illustrates an advantage of game semantics over other computational models since, e.g., the \emph{effective topos} \cite{hyland1982effective} cannot show the independence.
Moreover, by the \emph{non-inductive} nature of game semantics, the present method would be easily applied to the independence of Markov's principle from various \emph{extensions} of MLTT.

Finally, the novel mathematical structure of our game semantics enables:
\begin{corollary*}[Game semantics of subtyping on dependent types]
The game semantics given by Theorem~\ref{MainTheorem} models subtyping \cite[\S 15]{pierce2002types} on dependent types (\S\ref{GameSemanticsOfSubtyping}). 
\end{corollary*}

\subsection{Related work and our contributions}
\label{RelatedWorkAndOurContributions}
Abramsky et al. have established the first game semantics of MLTT equipped with One-, Pi-, Sigma-, Id- and finite inductive types \cite{abramsky2015games,vakar2018game}.
Its significance is that it is the first \emph{intensional} model of MLTT, and thus it stands in sharp contrast with other computational models such as realisability and domain models \cite{streicher2012semantics,palmgren1990domain}, which are extensional.
Their main result is a certain kind of \emph{full completeness}. 
However, they only interpret Sigma-types \emph{indirectly} and \emph{inductively} by a \emph{list} construction and the \emph{adjunction} between Pi- and Sigma-types.
Specifically, they interpret a Sigma-type $\mathsf{\vdash \Sigma_{x:C}D(x) \ type}$\footnote{For simplicity, here we only consider the \emph{empty} context.} by the list $(C, D)$ of the game $C$ that models the simple type $\mathsf{\vdash C \ type}$ and the family $D=(D(x))_{x:C}$ of games $D(x)$ indexed by strategies $x$ on $C$ that models the dependent type $\mathsf{x:C \vdash D(x) \ type}$, and interpret a term of the form $\mathsf{\vdash \langle c, d \rangle : \Sigma_{x:C}D(x)}$ by the list $(c, d)$ of the strategies $c$ and $d$ that respectively model the terms $\mathsf{\vdash c : C}$ and $\mathsf{\vdash d : D(c)}$, and a term of the form $\mathsf{z : \Sigma_{x:C}D(x) \vdash e(z) : E(z)}$ by identifying it with the one $\mathsf{x : C, y : D(x)  \vdash e(\langle x, y \rangle) : E(\langle x, y \rangle)}$.
As a result, they interpret types and terms by the \emph{lists} of (families of) games and strategies, respectively; this method merely simulates the \emph{term model} \cite[\S 2.4]{hofmann1997syntax} and drops the non-inductive nature of game semantics.
Another undesirable feature of this approach is that it identifies the types $\mathsf{\Pi_{z:\Sigma_{x : A} B(x)}C(z)}$ and $\mathsf{\Pi_{x : A} \Pi_{y : B} C(\langle x, y \rangle)}$.
Finally, for \emph{composing winning strategies}, they use the \textrm{O-sat} operation \cite[Remark~4.5]{vakar2018game}; however, it generates a significant \emph{gap} between MLTT and their model. 
In fact, this method only works for a very specific class of finite inductive types \cite[Figure~7]{vakar2018game}; see \S\ref{DependentPairSpace} on this point.

Another related work is the denotational model of MLTT \cite{blot2018extensional} by Blot and Laird based on \emph{concrete date structures} and \emph{sequential algorithms} \cite{berry1982sequential}. 
They interpret Boolean-, Pi- and Sigma-types and a universe though their interpretation of the universe is not by game semantics but by domain theory.
Their main results are certain \emph{full completeness/abstraction}. 
Notably, they \emph{directly} interpret Sigma-types without the list construction, overcoming the problem of the preceding work. 
However, it is possible to play on \emph{both} sides of their interpretation of a Sigma-type within a single play, which is far from the game semantics \cite{abramsky2000full,hyland2000full,mccusker1998games} (and even the \emph{graph game semantics} \cite{hyland2002games}) of product types. 
Thus, it is arguable if their model properly captures the generalisation of product types to Sigma-types.
As a more pragmatic disadvantage, their model admits control operators or \emph{classical reasonings}, but the logical part of MLTT is \emph{intuitionistic}; i.e., there is a \emph{gap} between MLTT and their model.
Besides, their model does not achieve the \emph{linear decomposition} of function types \cite{girard1987linear} or the \emph{characterisation of effects} by constraints on strategies, which are both strong advantages of game semantics \cite{abramsky1999game}.
Last but not least, their interpretation of Id-types by finite tuples of Boolean-type sketched in \cite[\S 9]{blot2018extensional} does not work in the presence of N-type since the set $\mathbb{N}$ of all natural numbers is unbounded.

Thus, each of the existing approaches to game semantics of MLTT has pros and cons, and we have not reached a consensus on which option should be a definitive solution.
In this context, we offer the third method with the novel features listed below, hoping that it would eventually lead to a definitive solution in the future.

First, we achieve the first game semantics of both N- and Id-types.
The interpretation of these types enables us to show the independence of Markov's principle.

Second, our games are a modest generalisation of a standard variant, \emph{McCusker's games} \cite{mccusker1998games}, and we model types and terms by such games and (ordinary) strategies, not lists of them, respectively.
As a result, we retain the \emph{syntax-independence} and the \emph{non-inductive} nature of game semantics.
Besides, we \emph{directly} model Sigma-types by our games, where a play occurs only in \emph{either} side.
Hence, we overcome the main shortcomings of the preceding methods. 
For the basic idea, see the beginning of \S\ref{PredicateGames}.

Third, our method inherits the linear decomposition of function types and the characterisation of effects by constraints on strategies in McCusker's one \cite{abramsky1999game}.

Fourth, the novel mathematical structure of our games enables us to dispense with the \textrm{O-sat} operation unlike Abramsky et al. so that we can interpret the type dependency of a standard class of Pi- and Sigma-types \emph{more accurately}; see \S\ref{DependentPairSpace}.

Finally, the mathematical structure of our games is novel and useful. 
For instance, the category of our games has all \emph{finite limits} (Corollary~\ref{CorGameSemanticLimits}), while that of existing games does not.
This novel structure enables us to internalise a certain notion of $\infty$-groupoids in the category of our games, which is a key step to extend the present work to HoTT \cite{yamada2021game}.
Moreover, we accomplish the first game semantics of \emph{subtyping on dependent types} (\S\ref{GameSemanticsOfSubtyping}).
On the other hand, in order to focus on the main idea of the present work, we leave it to another article to interpret universes.

\subsection{Concluding remarks}
\label{ConcludingRemarks}
Due to the novel mathematical structure of our games, one might misunderstand that our semantics is close to extensional models such as realisability and domain models.
However, since our strategies are just the \emph{ordinary} ones (\S\ref{RelatedWorkAndOurContributions}), our method inherits the advantages of standard game semantics such as \emph{intensionality} (\S\ref{IntroGameSemantics}).
In fact, the intensional features of the preceding game semantics \cite[\S 1]{abramsky2015games} are mostly valid in ours (\S\ref{Intensionality}); e.g., both \emph{refute function extensionality}. 
Also, our model refutes Markov's principle by its intensionality (\S\ref{Independence}), while the effective topos does not.

Last but not least, we do not prove full completeness for the following reasons.
First, the standard syntax of MLTT, specifically N-type, is not very suited to fully complete game semantics.
For instance, the full completeness results \cite{abramsky2015games,vakar2018game,blot2018extensional} are on modifications of MLTT, and they exclude N-type.
However, our main topic is MLTT itself (\S\ref{MLTT}), and thus we leave full completeness on a modification of MLTT as future work.
Second, our priority is more on an interpretation of N-type than full completeness without N-type since our motivation comes from foundations of mathematics (\S\ref{MLTT}), for which N-type is crucial.
Finally, one of our aims is to provide tools for the study of MLTT (\S\ref{MLTT}), but full completeness is not necessarily the most important result for this aim.
For example, the fully complete model \cite{blot2018extensional} cannot show the independence of Markov's principle since it admits classical reasonings.
Hence, we instead show the \emph{utility} of our model by proving the independence (\S\ref{Independence}).

\subsection{The structure of the present article}
The rest of this article proceeds as follows.
We first recall McCusker's games and strategies in \S\ref{GamesAndStrategies} and generalise the games in \S\ref{PredicateGames}. 
We then interpret MLTT by the generalised games and strategies in \S\ref{GameSemanticsOfMLTT}, where we analyse the intensionality of our game semantics in \S\ref{Intensionality}, give a game-semantic proof of the independence of Markov's principle from MLTT in \S\ref{Independence}, and interpret subtyping on dependent types in \S\ref{GameSemanticsOfSubtyping}.

\section{Games and strategies for simple type theories}
\label{GamesAndStrategies}
We first recall McCusker's games and strategies for simple type theories \cite{mccusker1998games}, which the present work is based on.
We select this variant for the following reasons.
First, it combines the strong points of the two best-known variants: the \emph{linear decomposition} of function types \cite{girard1987linear} achieved by \emph{AJM-games} \cite{abramsky2000full} and the characterisation of \emph{effects} by constraints on strategies \cite{abramsky1999game} that utilises \emph{pointers} in \emph{HO-games} \cite{hyland2000full} (originally introduced in \cite{coquand1995semantics}).
Our games inherit these advantages so that they would shed new light on the problem of combining MLTT and linear logic and/or effects. 
Second, pointers enable us to refine game semantics into a \emph{model of computation} \cite{yamada2019game}, which is highly desirable as a mathematical foundation of \emph{constructive} mathematics (\S\ref{MLTT}).

We assume that the reader is familiar with McCusker's games and strategies, and leave more expositions and examples to the gentle introduction \cite{abramsky1999game}.
We henceforth call McCusker's games and strategies respectively \emph{games} and \emph{strategies}. 

We first recall two preliminary concepts in \S\ref{ArenasAndLegalPositions}, and then games and strategies in \S\ref{SubsectionGamesAndStrategies}. 
We finally recall standard constructions on games and strategies in \S\ref{ConstructionsOnGamesAndStrategies}.

\begin{notation}
We use the following notations throughout the present article:
\begin{itemize}

\item We use bold small letters $\boldsymbol{s}, \boldsymbol{t}, \boldsymbol{u}, \boldsymbol{v}$, etc. for sequences, in particular $\boldsymbol{\epsilon}$ for the \emph{empty sequence}, and small letters $a, b, m, n, x, y$, etc. for elements of sequences;

\item We define $\overline{n} \colonequals \{ \, 1, 2, \dots, n \, \}$ for each $n \in \mathbb{N}^+ \colonequals \mathbb{N} \setminus \{ 0 \}$, and $\overline{0} \colonequals \emptyset$;

\item We write $x_1 x_2 \dots x_{|\boldsymbol{s}|}$ for $\boldsymbol{s} = (x_1, x_2, \dots, x_{|\boldsymbol{s}|})$, where $|\boldsymbol{s}|$ is the \emph{length} of $\boldsymbol{s}$, define $\boldsymbol{s}(i) \colonequals x_i$ ($i \in \overline{|\boldsymbol{s}|}$) and write $a \in \boldsymbol{s}$ if $a = \boldsymbol{s}(j)$ for some $j \in \overline{|\boldsymbol{s}|}$;


\item A \emph{concatenation} of sequences $\boldsymbol{s}$ and $\boldsymbol{t}$ is represented by their juxtaposition $\boldsymbol{s}\boldsymbol{t}$ (or $\boldsymbol{s} . \boldsymbol{t}$), but we often write $a \boldsymbol{s}$, $\boldsymbol{t} b$, $\boldsymbol{u} c \boldsymbol{v}$ for $(a) \boldsymbol{s}$, $\boldsymbol{t} (b)$, $\boldsymbol{u} (c) \boldsymbol{v}$, and so on;


\item We write $\mathrm{Even}(\boldsymbol{s})$ (resp. $\mathrm{Odd}(\boldsymbol{s})$) if $\boldsymbol{s}$ is of even- (resp. odd-) length, and given a set $S$ of sequences and $\mathsf{P} \in \{ \mathrm{Even}, \mathrm{Odd} \}$, we define $S^\mathsf{P} \colonequals \{ \, \boldsymbol{s} \in S \mid \mathsf{P}(\boldsymbol{s}) \, \}$;

\item We write $\boldsymbol{s} \preceq \boldsymbol{t}$ if $\boldsymbol{s}$ is a \emph{prefix} of $\boldsymbol{t}$, and given a set $S$ of sequences, $\mathrm{Pref}(S)$ for the set of all prefixes of sequences in $S$, i.e., $\mathrm{Pref}(S) \colonequals \{ \, \boldsymbol{s} \mid \exists \boldsymbol{t} \in S . \, \boldsymbol{s} \preceq \boldsymbol{t} \, \}$;




\item We use informal `tags' $(\_)_{[i]}$ ($i \in \mathbb{N}$) for clarity (e.g., see Definition~\ref{DefConstructionsOnStrategies}).



\end{itemize}
\end{notation}

\subsection{Arenas and legal positions}
\label{ArenasAndLegalPositions}
A \emph{game} is a certain kind of a \emph{rooted directed acyclic graph}, whose paths from a root represent possible developments or \emph{positions} in a `game in the ordinary sense' (e.g., chess). 
These positions are finite sequences of vertices or \emph{moves}, and a \emph{play} in the game proceeds as its participants alternately perform moves along a position. 
It is conventional to identify each game with the set of all positions in the game.
We focus on standard \emph{two-person} games between \emph{\bfseries Player (P)} (or a \emph{mathematician}) and \emph{\bfseries Opponent (O)} (or an \emph{oracle}), in which \emph{O always starts a play}. 

Technically, games are based on two preliminary concepts: \emph{arenas} and \emph{legal positions}. An arena defines the basic components of a game, which in turn induces legal positions of the arena that specify the basic rules of the game in the sense that each position of the game must be legal. 
Let us first recall these two concepts.

\begin{definition}[Moves]
\label{DefMoves}
Let us fix, throughout the present work, arbitrary pairwise distinct symbols $\mathsf{O}$, $\mathsf{P}$, $\mathsf{Q}$ and $\mathsf{A}$, and call them \emph{\bfseries labels}.
A \emph{\bfseries move} is any triple $m^{xy} \colonequals (m, x, y)$ such that $x \in \{ \mathsf{O}, \mathsf{P} \}$ and $y \in \{ \mathsf{Q}, \mathsf{A} \}$. 
We usually abbreviate moves $m^{xy}$ as $m$, and instead define $\lambda(m) \colonequals xy$, $\lambda^\mathsf{OP}(m) \colonequals x$ and $\lambda^\mathsf{QA}(m) \colonequals y$.

We call a move $m$ an \emph{\bfseries O-move} if $\lambda^\mathsf{OP}(m) = \mathsf{O}$, a \emph{\bfseries P-move} if $\lambda^\mathsf{OP}(m) = \mathsf{P}$, a \emph{\bfseries question} if $\lambda^\mathsf{QA}(m) = \mathsf{Q}$, and an \emph{\bfseries answer} if $\lambda^\mathsf{QA}(m) = \mathsf{A}$.
\end{definition}

\begin{definition}[Arenas \cite{hyland2000full,mccusker1998games}]
\label{DefArenas}
An \emph{\bfseries arena} is a pair $G = (M_G, \vdash_G)$ such that
\begin{itemize}
\item $M_G$ is a set of moves;

\item $\vdash_G$ is a subset of the cartesian product $(\{ \star \} \cup M_G) \times M_G$, where $\star$ (or represented more precisely by $\star_G$) is an arbitrarily fixed element such that $\star \not \in M_G$, called the \emph{\bfseries enabling relation}, that satisfies
\begin{itemize}

\item \textsc{(E1)} If $\star \vdash_G m$, then $\lambda (m) = \mathsf{OQ}$;

\item \textsc{(E2)} If $m \vdash_G n$ and $\lambda^\mathsf{QA} (n) = \mathsf{A}$, then $\lambda^\mathsf{QA} (m) = \mathsf{Q}$;

\item \textsc{(E3)} If $m \vdash_G n$ and $m \neq \star$, then $\lambda^\mathsf{OP} (m) \neq \lambda^\mathsf{OP} (n)$.

\end{itemize}

\end{itemize}
We call moves $m \in M_G$ \emph{\bfseries initial} if $\star \vdash_G m$, and set $M_G^{\mathrm{Init}} \colonequals \{ \, m \in M_G \mid \star \vdash_G m \, \}$.

An arena $G$ is \emph{\bfseries well-founded (w.f.)} if $\vdash_G$ is well-founded, i.e., there is no sequence $(m_i)_{i \in \mathbb{N}}$ of moves $m_i \in M_G$ such that $\star \vdash_G m_0$ and $m_i \vdash_G m_{i+1}$ for all $i \in \mathbb{N}$.
\end{definition}

\begin{remark}
In the original article \cite{mccusker1998games}, an arena is a triple $G = (M_G, \lambda_G, \vdash_G)$, where labels are \emph{assigned} to moves by the \emph{labelling function} $\lambda_G : M_G \rightarrow \{ \mathsf{O}, \mathsf{P} \} \times \{ \mathsf{Q}, \mathsf{A} \}$.
Instead, we \emph{embed} labels into moves (Definition~\ref{DefMoves}); this modification is convenient, e.g., when we take \emph{unions} of games, since it makes labels on moves unambiguous \emph{without underlying arenas}.
Besides, the axiom E1 in \cite{mccusker1998games} further requires $n \vdash_G m \Leftrightarrow n = \star$ whenever $\star \vdash_G m$.
We discard this condition too again for unions of games.
\end{remark}

An arena $G$ specifies moves in a game, each of which is O's/P's question/answer, and which move $n$ can be performed for each move $m$ during a play in the game by the relation $m \vdash_G n$ (cf.~Definition~\ref{DefJSequences}), where $\star \vdash_G m$ means that O can initiate a play by $m$ in the game.
The axioms E1, E2 and E3 are then to be read as follows:
\begin{itemize}

\item E1 sets the convention that an initial move must be O's question;

\item E2 states that an answer must be performed for a question;

\item E3 says that an O-move must be performed for a P-move, and vice versa.

\end{itemize}

We next review \emph{legal positions}, a certain class of  finite sequences of moves equipped with \emph{pointers} from later to earlier occurrences in the sequences. 
The idea is that each non-initial occurrence in a legal position must be made for a specific previous occurrence, and pointers specify such pairs of occurrences. 
Technically, pointers enable us to distinguish similar yet different plays \cite[\S 2.4]{abramsky1999game} and define \emph{views} (Definition~\ref{DefViews}). 
Views play crucial roles when we define constraints on strategies (Definition~\ref{DefConstraintsOnStrategies}). 

We call a finite sequence of moves together with a pointer a \emph{justified (j-) sequence}.
A legal position is then a particular kind of a j-sequence.

\begin{definition}[Justified sequences \cite{coquand1995semantics,mccusker1998games}]
\label{DefJSequences}
An \emph{\bfseries occurrence} in a finite sequence $\boldsymbol{s}$ is a pair $(\boldsymbol{s}(i), i)$ such that $i \in \overline{|\boldsymbol{s}|}$. 
A \emph{\bfseries justified (j-) sequence} is a pair $\boldsymbol{s} = (\boldsymbol{s}, \mathcal{J}_{\boldsymbol{s}})$ of a finite sequence $\boldsymbol{s}$ of moves and a map $\mathcal{J}_{\boldsymbol{s}} : \overline{|\boldsymbol{s}|} \rightarrow \{ 0 \} \cup \overline{|\boldsymbol{s}|-1}$ such that $0 \leqslant \mathcal{J}_{\boldsymbol{s}}(i) < i$ for all $i \in \overline{|\boldsymbol{s}|}$, called the \emph{\bfseries pointer} of the j-sequence. 
An occurrence $(\boldsymbol{s}(i), i)$ is \emph{\bfseries initial} in $\boldsymbol{s}$ if $\mathcal{J}_{\boldsymbol{s}}(i) = 0$.
We say that the occurrence $(\boldsymbol{s}({\mathcal{J}_{\boldsymbol{s}}(i)}), \mathcal{J}_{\boldsymbol{s}}(i))$ is the \emph{\bfseries justifier} of a non-initial one $(\boldsymbol{s}(i), i)$ in $\boldsymbol{s}$, and $(\boldsymbol{s}(i), i)$ is \emph{\bfseries justified} by $(\boldsymbol{s}({\mathcal{J}_{\boldsymbol{s}}(i)}), \mathcal{J}_{\boldsymbol{s}}(i))$ in $\boldsymbol{s}$. 

A j-sequence $\boldsymbol{s}$ is \emph{\bfseries in an arena} $G$ if its elements are moves in $G$, and its pointer respects the enabling relation $\vdash_G$ in $G$, i.e., $\forall i \in \overline{|\boldsymbol{s}|} . \, \Big(\mathcal{J}_{\boldsymbol{s}}(i) = 0 \Rightarrow \star \vdash_G \boldsymbol{s}(i)\Big) \wedge \Big(\mathcal{J}_{\boldsymbol{s}}(i) \neq 0 \Rightarrow \boldsymbol{s}({\mathcal{J}_{\boldsymbol{s}}(i)}) \vdash_G \boldsymbol{s}(i)\Big)$.
We write $\mathscr{J}_G$ for the set of all j-sequences in $G$.

A \emph{\bfseries justified (j-) subsequence} of a j-sequence $\boldsymbol{s}$ is a j-sequence $\boldsymbol{t}$, written $\boldsymbol{t} \sqsubseteq \boldsymbol{s}$, such that $\boldsymbol{t}$ is a subsequence of $\boldsymbol{s}$, and $\mathcal{J}_{\boldsymbol{t}}(i) = j$ if and only if $\mathcal{J}_{\boldsymbol{s}}^n(i) = j$ for some $n \in \mathbb{N}^+$ with the occurrences $(\boldsymbol{s}(\mathcal{J}_{\boldsymbol{s}}^k(i)), \mathcal{J}_{\boldsymbol{s}}^k(i))$ for $k = 1, 2, \dots, n-1$ deleted in $\boldsymbol{t}$.
\end{definition}

\begin{remark}
Unlike the original formulation \cite{mccusker1998games}, we define j-sequences in such a way that they make sense \emph{without underlying arenas}. 
This reformulation is convenient as it enables us to manipulate j-sequences without calculating underlying arenas. 
\end{remark}

\begin{convention}
Henceforth, we are casual about the distinction between moves and occurrences, and by abuse of notation, we frequently keep the pointer $\mathcal{J}_{\boldsymbol{s}}$ of each j-sequence $\boldsymbol{s} = (\boldsymbol{s}, \mathcal{J}_{\boldsymbol{s}})$ implicit since it is mostly obvious, and abbreviate occurrences $(\boldsymbol{s}(i), i)$ in $\boldsymbol{s}$ as $\boldsymbol{s}(i)$.
Besides, we sometimes write $\mathcal{J}_{\boldsymbol{s}}(\boldsymbol{s}(i)) = \boldsymbol{s}(j)$ if $\mathcal{J}_{\boldsymbol{s}}(i) = j > 0$.
\end{convention}

\begin{definition}[Views \cite{coquand1995semantics,hyland2000full,mccusker1998games}] 
\label{DefViews}
The \emph{\bfseries P-view} $\lceil \boldsymbol{s} \rceil$ and the \emph{\bfseries O-view} $\lfloor \boldsymbol{s} \rfloor$ of a j-sequence $\boldsymbol{s}$ are the j-subsequences of $\boldsymbol{s}$ defined by the following induction: 
\begin{itemize}

\item $\lceil \boldsymbol{\epsilon} \rceil \colonequals \boldsymbol{\epsilon}$; 

\item $\lceil \boldsymbol{s} m \rceil \colonequals \lceil \boldsymbol{s} \rceil . m$ if $m$ is a P-move; 

\item $\lceil \boldsymbol{s} m \rceil \colonequals m$ if $m$ is initial;

\item $\lceil \boldsymbol{s} m \boldsymbol{t} n \rceil \colonequals \lceil \boldsymbol{s} \rceil . m n$ if $n$ is an O-move such that $m$ justifies $n$; 

\item $\lfloor \boldsymbol{\epsilon} \rfloor \colonequals \boldsymbol{\epsilon}$;

\item $\lfloor \boldsymbol{s} m \rfloor \colonequals \lfloor \boldsymbol{s} \rfloor . m$ if $m$ is an O-move; 

\item $\lfloor \boldsymbol{s} m \boldsymbol{t} n \rfloor \colonequals \lfloor \boldsymbol{s} \rfloor . m n$ if $n$ is a P-move such that $m$ justifies $n$. 

\end{itemize}
\end{definition}

The idea on views is as follows.
Given a nonempty j-sequence $\boldsymbol{s} m$ such that $m$ is a P- (resp. O-) move, the P-view $\lceil \boldsymbol{s} \rceil$ (resp. O-view $\lfloor \boldsymbol{s} \rfloor$) is the currently `relevant part' of the previous occurrences in $\boldsymbol{s}$ for P (resp. O). 
I.e., P (resp. O) is concerned only with the last occurrence of an O- (resp. P-) move, its justifier and that justifier's P- (resp. O-) view, which then recursively proceeds.
See \cite{coquand1995semantics} for an explanation of views in terms of their counterparts in logical calculi, and \cite{curien1998abstract} in lambda-calculi.

\begin{definition}[Legal positions \cite{abramsky1999game}]
\label{DefLegalPositions}
A \emph{\bfseries legal position} is a j-sequence $\boldsymbol{s}$ such~that
\begin{itemize}

\item \textsc{(Alternation)} If $\boldsymbol{s} = \boldsymbol{s_1} m n \boldsymbol{s_2}$, then $\lambda^\mathsf{OP} (m) \neq \lambda^\mathsf{OP} (n)$;


\item \textsc{(Visibility)} If $\boldsymbol{s} = \boldsymbol{t} m \boldsymbol{u}$ with $m$ non-initial, then $\mathcal{J}_{\boldsymbol{s}}(m)$ occurs in the P-view $\lceil \boldsymbol{t} \rceil$ if $m$ is a P-move, and in the O-view $\lfloor \boldsymbol{t} \rfloor$ otherwise.

\end{itemize}

A legal position is \emph{\bfseries in an arena} $G$ if it is a j-sequence in $G$ (Definition~\ref{DefJSequences}).
We write $\mathscr{L}_G$ for the set of all legal positions in $G$.
\end{definition}

As already noted, legal positions \emph{in an arena} are to specify the basic rules of a game in the sense that positions in the game must be legal (Definition~\ref{DefGames}) so that
\begin{itemize}

\item During a play in the game, O makes the first move by a question (by E1),\footnote{Since the initial element $\boldsymbol{s}(1)$ of a legal position $\boldsymbol{s}$ \emph{in an arena $G$} is subject to the equation $\mathcal{J}_{\boldsymbol{s}}(1) = 0$, we have $\star \vdash_G \boldsymbol{s}(1)$. Hence, the axiom E1 on $G$ implies $\lambda(\boldsymbol{s}(1)) = \mathsf{OQ}$.\label{FirstFootnoteOnArenas}} and then P and O alternately make moves (by alternation), where each non-initial move is made for a specific previous occurrence, viz., its justifier;\footnote{Again, since we focus on a legal position $\boldsymbol{s}$ \emph{in an arena $G$}, the justifier of each P-move occurring in $\boldsymbol{s}$ is an O-move, and vice versa, by the axiom E3 on $G$. In addition, the justifier of each answer occurring in $\boldsymbol{s}$ is a question by the axiom E2 on $G$.\label{SecondFootnoteOnArenas}}


\item The justifiers of non-initial occurrences are in the `relevant part' (by visibility). 
\end{itemize}


\subsection{Games and strategies}
\label{SubsectionGamesAndStrategies}
We are now ready to recall the central concepts of \emph{games} and \emph{strategies}.
For technical convenience, we slightly modify the original definition of games in \cite{mccusker1998games,abramsky1999game}:

\begin{definition}[Games \cite{mccusker1998games,abramsky1999game}]
\label{DefGames}
A \emph{\bfseries game} is a set $G$ of legal positions such that
\begin{enumerate}

\item $G$ is nonempty and \emph{prefix-closed} (i.e., $\boldsymbol{s}m \in G \Rightarrow \boldsymbol{s} \in G$);

\item $\mathrm{Arn}(G) \colonequals (M_G, \vdash_G)$ is an arena, where $M_G \colonequals \{ \, \boldsymbol{s}(i) \mid \boldsymbol{s} \in G, i \in \overline{|\boldsymbol{s}|} \, \}$ and $\vdash_G \, \colonequals \{ \, (\star, \boldsymbol{s}(j)) \mid \boldsymbol{s} \in G, \mathcal{J}_{\boldsymbol{s}}(j) = 0 \, \}\cup \{ \, (\boldsymbol{s}(i), \boldsymbol{s}(j)) \mid \boldsymbol{s} \in S, \mathcal{J}_{\boldsymbol{s}}(j)= i > 0 \, \}$.

\end{enumerate}

It is \emph{\bfseries well-founded (w.f.)} if so is $\mathrm{Arn}(G)$, and \emph{\bfseries well-opened (w.o.)} if each of its elements has at most one initial occurrence (i.e., the conjunction of $\boldsymbol{s}m \in G$ and $m \in M_G^{\mathrm{Init}}$ implies $\boldsymbol{s} = \boldsymbol{\epsilon}$).
We call elements of $G$ \emph{\bfseries (valid) positions} in $G$.

A \emph{\bfseries subgame} of $G$ is a game $H \subseteq G$, and $\mathrm{sub}(G) \colonequals \{ \, H \mid \text{$H$ is a subgame of $G$} \, \}$.
\end{definition}

\begin{remark}
The original article \cite[p.~27]{mccusker1998games} also imposes \emph{thread-closure} on each game $G$: The \emph{thread} $\boldsymbol{s} \upharpoonright \mathcal{I} \sqsubseteq \boldsymbol{s}$ of a position $\boldsymbol{s} \in G$ with respect to a given set $\mathcal{I}$ of initial occurrences in $\boldsymbol{s}$, which consists of occurrences hereditarily justified\footnote{An occurrence $n$ in a j-sequence $\boldsymbol{s}$ is \emph{hereditarily justified} by another occurrence $m$ in $\boldsymbol{s}$ if $\mathcal{J}_{\boldsymbol{s}}^i(n) = m$ for some $i \in \mathbb{N}^+$ \cite[p.~22]{mccusker1998games}.\label{FootnoteOnHereditarilyJustifiedOccurrences}} by elements in $\mathcal{I}$, \emph{must be in $G$}.
This axiom is to ensure that positions in $G$ are in the \emph{exponential} $\oc G$ (Definition~\ref{DefConstructionsOnGames}), i.e., $G \subseteq \oc G$, which matches the intuition on exponential $\oc$ \cite{girard1987linear}.

However, $\oc G$ is well-defined even if $G$ is not thread-closed.
Also, we later focus on \emph{w.o.} games (for the identities in the categories of games to be well-defined), which are trivially thread-closed.
For these reasons, we omit thread-closure in Definition~\ref{DefGames}.
\end{remark}

Each game $G$ is nonempty and prefix-closed because conceptually each nonempty position or `moment' in $G$ must have the previous `moment.'
Note that positions in $G$ are automatically \emph{legal in the arena $\mathrm{Arn}(G)$}; i.e., as noted before, the legality is the basic or minimal requirement on positions in games.
We later focus on \emph{w.o., w.f.} games because the identities in the categories of such games behave well (\S\ref{ConstructionsOnGamesAndStrategies}).

The tuple $\mathcal{M}(G) \colonequals (M_G, \lambda_G, \vdash_G, G)$, where $\lambda_G : m \mapsto \lambda(m)$, forms a game in the sense of \cite{mccusker1998games}\footnote{Except that the axiom E1 is slightly weakened, and the thread-closure is omitted; see the remarks after Definitions~\ref{DefArenas} and \ref{DefGames}, respectively.} whose labels are embedded into moves, which we call an \emph{\bfseries MC-game}.
The MC-game $\mathcal{M}(G)$ satisfies: Each move $m \in M_G$ occurs in some position in $G$, and each pair $\star \vdash_G m$ or $m_1 \vdash_G m_2$ is used in some position in $G$.
Conversely, given an MC-game $H$ that satisfies these two conditions, for which we call $H$ \emph{\bfseries economical}, the set of all positions in $H$ forms a game.
Besides, these constructions are inverses to each other. 
Hence, $\mathcal{M}$ is a \emph{bijection} between games and economical MC-games.

Since the economical axioms only exclude unused structures, our simplification of MC-games into games is \emph{harmless}. 
We may further dispense with arenas by directly axiomatising the effects of arenas on legal positions (Footnotes~\ref{FirstFootnoteOnArenas}--\ref{SecondFootnoteOnArenas}), but we do not since arenas are convenient for defining constructions on games (Definition~\ref{DefConstructionsOnGames}).

\begin{definition}[Strategies \cite{mccusker1998games,abramsky1999game}]
\label{DefStrategies}
A \emph{\bfseries strategy} on a game $G$ is a subset $\sigma \subseteq G^{\mathrm{Even}}$, written $\sigma : G$, that is nonempty, \emph{even-prefix-closed} (i.e., $\boldsymbol{s}mn \in \sigma \Rightarrow \boldsymbol{s} \in \sigma$) and \emph{deterministic} (i.e., $\boldsymbol{s}mn, \boldsymbol{s}mn' \in \sigma \Rightarrow \boldsymbol{s}mn = \boldsymbol{s}mn'$).
Let $\mathrm{st}(G) \colonequals \{ \, \sigma \mid \sigma : G \, \}$.

We define the \emph{\bfseries closure} of a strategy $\sigma : G$ with respect to another game $H$ to be the subgame $\overline{\sigma}_H \colonequals \{ \boldsymbol{\epsilon} \} \cup \{ \, \boldsymbol{s}m \in H^{\mathrm{Odd}} \mid \boldsymbol{s} \in \overline{\sigma}_H \, \} \cup \{ \, \boldsymbol{t}lr \in \sigma \mid \boldsymbol{t}l \in \overline{\sigma}_H \, \} \subseteq \sigma \cup H$.
\end{definition}

The idea is that a strategy $\sigma : G$ describes for P \emph{how to play} on the game $G$ by the computation $\boldsymbol{s}m \in G^{\mathrm{Odd}} \mapsto \boldsymbol{s}mn \in \sigma$, which is \emph{deterministic} by the determinacy of $\sigma$, and in general \emph{partial} since there can be no $\boldsymbol{s}mn \in \sigma$ for some $\boldsymbol{s}m \in G^{\mathrm{Odd}}$.

We use the closure operation in \S\ref{PredicateGames}.
We leave it to the reader to verify by induction that the equation $\overline{\sigma}_G = \sigma \cup \{ \, \boldsymbol{s}m \in G \mid \boldsymbol{s} \in \sigma \, \}$ holds for all strategies $\sigma : G$.


\begin{example}
\label{ExamplesOfGames}
The \emph{\bfseries terminal game} $T \colonequals \{ \boldsymbol{\epsilon} \}$ only has the strategy $\top \colonequals \{ \boldsymbol{\epsilon} \}$.


The \emph{\bfseries flat game} on a set $S$ is the game $\mathrm{flat}(S) \colonequals \mathrm{Pref}(\{ \, q^{\mathsf{OQ}} . m^{\mathsf{PA}} \mid m \in S \, \})$, where $q$ is an arbitrarily fixed element such that $q \not\in S$, and $q^{\mathsf{OQ}}$ justifies $m^{\mathsf{PA}}$. 
It has strategies $\bot \colonequals \{ \boldsymbol{\epsilon} \}$ and $\underline{m} \colonequals \{ \boldsymbol{\epsilon}, q m \}$ for each $m \in S$. 
Consider, for instance, the \emph{\bfseries empty game} $\boldsymbol{0} \colonequals \mathrm{flat}(\emptyset)$ and the \emph{\bfseries natural number game} $N \colonequals \mathrm{flat}(\mathbb{N})$.
\end{example}

Next, recall that not every strategy corresponds to a \emph{proof}.
For instance, the empty game $\boldsymbol{0}$ models \emph{falsity}, and thus the strategy $\bot : \boldsymbol{0}$ should not be an interpretation of a proof.
We therefore carve out strategies for proofs as \emph{winning} ones:
\begin{definition}[Constraints on strategies \cite{coquand1995semantics,laird1997full,mccusker1998games,abramsky1999game}]
\label{DefConstraintsOnStrategies}
A strategy $\sigma : G$ is
\begin{itemize}

\item \emph{\bfseries Total} if it always responds: $\forall \boldsymbol{s} \in \sigma, \boldsymbol{s} m \in G . \, \exists \boldsymbol{s} m n \in \sigma$;

\item \emph{\bfseries Innocent} if it only depends on P-views: $\forall \boldsymbol{s} m n \in \sigma, \boldsymbol{t} l \in G . \, \lceil \boldsymbol{s} m \rceil = \lceil \boldsymbol{t} l \rceil \Rightarrow \exists \boldsymbol{t}lr \in \sigma . \, \lceil \boldsymbol{s} m n \rceil = \lceil \boldsymbol{t}lr \rceil$;

\item \emph{\bfseries Noetherian} if there is no strictly increasing (with respect to the prefix relation $\preceq$) infinite sequence of elements in the set $\lceil \sigma \rceil \colonequals \{ \, \lceil \boldsymbol{s} \rceil \mid \boldsymbol{s} \in \sigma \, \}$;

\item \emph{\bfseries Winning} if it is total, innocent and noetherian;

\item \emph{\bfseries Well-bracketed (w.b.)} if its `question-answering' in P-views is in the `last-question-first-answered' fashion: If $\boldsymbol{s} q \boldsymbol{t} a \in \sigma$, where $\lambda^{\mathsf{QA}}(q) = \mathsf{Q}$, $\lambda^{\mathsf{QA}}(a) = \mathsf{A}$ and $\mathcal{J}_{\boldsymbol{s}q\boldsymbol{t}a}(a) = q$, then each question occurring in $\boldsymbol{t'}$, where the P-view $\lceil \boldsymbol{s} q \boldsymbol{t} \rceil$ is of the form $\lceil \boldsymbol{s} q \boldsymbol{t} \rceil = \lceil \boldsymbol{s} q \rceil . \boldsymbol{t'}$ by visibility, justifies an answer occurring in $\boldsymbol{t'}$.

\end{itemize}
\end{definition}

\begin{example}
The strategies $\top : T$ and $\underline{n} : N$ for all $n \in \mathbb{N}$ are winning and w.b., while the strategies $\bot : \boldsymbol{0}$ and $\bot : N$ are not even total, let alone winning. 
\end{example}

We think of winning strategies as \emph{proofs in classical logic} as follows.
First, proofs should not get `stuck,' and so strategies for proofs must be \emph{total}.
Next, imposing \emph{innocence} on strategies corresponds to excluding \emph{stateful} terms \cite[\S 2.9]{abramsky1999game}.
Since logic is concerned with \emph{truths}, which are independent of  `passage of time,' proofs should not depended on `states of arguments.' 
Hence, we impose innocence on strategies for proofs.
In addition, we need \emph{noetherianity} to handle infinite plays: If a play by an innocent, noetherian strategy keeps growing infinitely, then it cannot be P's `intention,' and so the play must be \emph{win} for P.
Technically, noetherianity is crucial for the closure of winning strategies under \emph{composition} (Definition~\ref{DefConstructionsOnStrategies}) \cite{coquand1995semantics}.

Further, \emph{well-bracketing} bans classical reasoning or \emph{control operators} \cite[\S 2.10]{abramsky1999game}.
Hence, we regard winning, w.b. strategies as \emph{proofs in intuitionistic logic}.

\subsection{Constructions on games and strategies}
\label{ConstructionsOnGamesAndStrategies}
In this section, we briefly recall constructions on games and strategies.
Since they are standard in the literature, we leave expositions and examples to \cite[\S 3.2]{abramsky1999game}.

\begin{convention}
We omit `tags' for disjoint union $\uplus$. 
For instance, we write $x \in A \uplus B$ if $x \in A$ or $x \in B$; given relations $R_A \subseteq A \times A$ and $R_B \subseteq B \times B$, we write $R_A \uplus R_B$ for the relation on $A \uplus B$ such that $(x, y) \in R_A \uplus R_B \stackrel{\mathrm{df. }}{\Leftrightarrow} (x, y) \in R_A \vee (x, y) \in R_B$. 
\end{convention}

\begin{definition}[Constructions on arenas \cite{mccusker1998games}]
Given arenas $A$ and $B$, we define
\begin{itemize}

\item $A \uplus B \colonequals (M_A \uplus M_B, \vdash_A \uplus \vdash_B)$; 

\item $A \multimap B \colonequals (\{ \, a^{(x^\bot)y} \mid a^{xy} \in M_A \, \} \uplus M_B, \vdash_{A \multimap B})$, $\mathsf{O}^\bot \colonequals \mathsf{P}$, $\mathsf{P}^\bot \colonequals \mathsf{O}$, $\star \vdash_{A \multimap B} m :\Leftrightarrow \star \vdash_B m$ and $m \vdash_{A \multimap B} n :\Leftrightarrow m \vdash_A n \vee m \vdash_B n \vee (\star \vdash_B m \wedge \star \vdash_A n)$.


\end{itemize}
\end{definition}

\begin{definition}[Constructions on games \cite{mccusker1998games}]
\label{DefConstructionsOnGames}
Given games $G$ and $H$, we define
\begin{itemize}

\item $G \otimes H \colonequals \{ \, \boldsymbol{s} \in \mathscr{L}_{\mathrm{Arn}(G) \uplus \mathrm{Arn}(H)} \mid \forall X \in \{ G, H \} . \, \boldsymbol{s} \upharpoonright X \in X \, \}$, called the \emph{\bfseries tensor} of $G$ and $H$, where $\boldsymbol{s} \upharpoonright X \sqsubseteq \boldsymbol{s}$ consists of occurrences of moves in $X$;

\item $\oc G \colonequals \{ \, \boldsymbol{s} \in \mathscr{L}_{\mathrm{Arn}(G)} \mid \forall i \in |\boldsymbol{s}| . \, \mathcal{J}_{\boldsymbol{s}}(i) = 0 \Rightarrow \boldsymbol{s} \upharpoonright \{ (\boldsymbol{s}(i), i) \} \in G \, \}$, called the \emph{\bfseries exponential} of $G$, where $\boldsymbol{s} \upharpoonright \{ (\boldsymbol{s}(i), i) \} \sqsubseteq \boldsymbol{s}$ consists of occurrence in $\boldsymbol{s}$ hereditarily justified (Footnote~\ref{FootnoteOnHereditarilyJustifiedOccurrences}) by the initial occurrence $(\boldsymbol{s}(i), i)$ in $\boldsymbol{s}$;

\item $G \mathbin{\&} H \colonequals \{ \, \boldsymbol{s} \in \mathscr{L}_{\mathrm{Arn}(G) \mathbin{\uplus} \mathrm{Arn}(H)} \mid (\boldsymbol{s} \upharpoonright G \in G \wedge \boldsymbol{s} \upharpoonright H = \boldsymbol{\epsilon}) \vee (\boldsymbol{s} \upharpoonright G = \boldsymbol{\epsilon} \wedge \boldsymbol{s} \upharpoonright H \in H) \, \}$, called the \emph{\bfseries product} of $G$ and $H$;

\item $G \multimap H \colonequals \{ \, \boldsymbol{s} \in \mathscr{L}_{\mathrm{Arn}(G) \multimap \mathrm{Arn}(H)} \mid \boldsymbol{s} \upharpoonright G^\bot \in G, \boldsymbol{s} \upharpoonright H \in H \, \}$, also written $H^G$, where $\boldsymbol{s} \upharpoonright G^\bot$ is obtained from $\boldsymbol{s} \upharpoonright G$ by modifying all the moves $m^{(x^\bot)y}$ occurring in $\boldsymbol{s} \upharpoonright G$ into $m^{xy}$, called the \emph{\bfseries linear implication} from $G$ to $H$;

\item $G \Rightarrow H \colonequals \oc G \multimap H$, called the \emph{\bfseries implication} from $G$ to $H$.

\end{itemize}

Notationally, exponential $\oc$ precedes other constructions on games, while tensor $\otimes$ and product $\mathbin{\&}$ do linear implication $\multimap$ and implication $\Rightarrow$. 
\end{definition}

\begin{lemma}[Well-defined constructions on games]
\label{LemWellDefinedConstructionsOnGames}
Games and w.f. games are closed under $\otimes$, $\mathbin{\&}$, $\oc$ and $\multimap$, and w.o. ones under $\&$, $\multimap$ and $\Rightarrow$.
\end{lemma}
\begin{proof}
See \cite{mccusker1998games} for the closure of MC-games under these constructions.
The proof is essentially the same for games.
The preservation of w.f. (resp. w.o.) ones is clear.
\end{proof}

We leave it to the reader to verify that these constructions on games correspond to those on economical MC-games \cite{mccusker1998games} under the bijection $\mathcal{M}$ (except that we have to exclude unused moves and enabled pairs from an MC-game $G \multimap H$ if $H = T$).

\begin{definition}[Constructions on strategies \cite{mccusker1998games}]
\label{DefConstructionsOnStrategies}
Given strategies $\phi : A \multimap B$, $\sigma : C \multimap D$, $\tau : A \multimap C$, $\psi : B \multimap C$ and $\theta : \oc A \multimap B$, we define
\begin{itemize}

\item $\mathrm{cp}_A \colonequals \{ \, \boldsymbol{s} \in (A_{[0]} \multimap A_{[1]})^{\mathrm{Even}} \mid \forall \boldsymbol{t} \preceq \boldsymbol{s} . \, \mathrm{Even}(\boldsymbol{t}) \Rightarrow \boldsymbol{t} \upharpoonright A_{[0]}^\bot = \boldsymbol{t} \upharpoonright A_{[1]} \, \}$, called the \emph{\bfseries copy-cat} on $A$; 

\item $\mathrm{der}_A \colonequals \{ \, \boldsymbol{s} \in (\oc A \multimap A)^{\mathrm{Even}} \mid \forall \boldsymbol{t} \preceq \boldsymbol{s} . \, \mathrm{Even}(\boldsymbol{t}) \Rightarrow \boldsymbol{t} \upharpoonright \oc A^\bot = \boldsymbol{t} \upharpoonright A \, \}$, called the \emph{\bfseries dereliction} on $A$;

\item $\phi \otimes \sigma \colonequals \{ \, \boldsymbol{s} \in A \otimes C \multimap B \otimes D \mid \boldsymbol{s} \upharpoonright A, B \in \phi, \boldsymbol{s} \upharpoonright C, D \in \sigma \, \}$, called the \emph{\bfseries tensor} of $\phi$ and $\sigma$, where $\boldsymbol{s} \upharpoonright A, B \sqsubseteq \boldsymbol{s}$ (resp. $\boldsymbol{s} \upharpoonright C, D \sqsubseteq \boldsymbol{s}$) consists of occurrences of moves in $A$ or $B$ (resp. $C$ or $D$);

\item $\langle \phi, \tau \rangle \colonequals \{ \, \boldsymbol{s} \in A \multimap B \mathbin{\&} C \mid (\boldsymbol{s} \upharpoonright A, B \in \phi \wedge \boldsymbol{s} \upharpoonright C = \boldsymbol{\epsilon}) \vee (\boldsymbol{s} \upharpoonright A, C \in \tau \wedge \boldsymbol{s} \upharpoonright B = \boldsymbol{\epsilon}) \, \}$, called the \emph{\bfseries pairing} of $\phi$ and $\tau$;

\item $\phi ; \psi \colonequals \{ \, \boldsymbol{s} \upharpoonright A, C \mid \boldsymbol{s} \in \phi \parallel \psi \, \}$, called the \emph{\bfseries composition} of $\phi$ and $\psi$, where $\phi \parallel \psi \colonequals \{ \, \boldsymbol{s} \in \mathscr{J} \mid \boldsymbol{s} \upharpoonright A, B_{[0]} \in \phi, \boldsymbol{s} \upharpoonright B_{[1]}, C \in \psi, \boldsymbol{s} \upharpoonright B_{[0]}^\bot, B_{[1]}^\bot \in \mathrm{cp}_B \, \}$, $\mathscr{J} \colonequals \mathscr{J}_{\mathrm{Arn}(((A \multimap B_{[0]}) \multimap B_{[1]}) \multimap C)}$, $\boldsymbol{s} \upharpoonright B_{[0]}^\bot, B_{[1]}^\bot$ is obtained from $\boldsymbol{s} \upharpoonright B_{[0]}, B_{[1]}$ by applying the operation $(\_)^\bot$ on all the moves, and $\phi; \psi$ is also written $\psi \circ \phi$;

\item $\theta^\dagger \colonequals \{ \, \boldsymbol{s} \in (\oc A \multimap \oc B)^{\mathrm{Even}} \mid \forall i \in |\boldsymbol{s}| . \, \mathcal{J}_{\boldsymbol{s}}(i) = 0 \Rightarrow \boldsymbol{s} \upharpoonright \{ (\boldsymbol{s}(i), i) \} \in \theta \, \}$, called the \emph{\bfseries promotion} of $\theta$. 
\end{itemize}
\end{definition}

For the dereliction $\mathrm{der}_A$ to be well-defined, we have to focus on \emph{w.o.} games $A$; see \cite[pp.~42--43]{mccusker1998games}.
Although w.o. games are not closed under exponential $\oc$, it does not matter for us, like \cite[p.~43]{mccusker1998games}, since we only need \emph{cartesian closure}, not exponential $\oc$ itself, and w.o. games are closed under implication $\Rightarrow$ and product $\&$ (Lemma~\ref{LemWellDefinedConstructionsOnGames}). 

\begin{lemma}[Well-defined constructions on strategies]
\label{LemWellDefinedConstructionsOnStrategies}
If $\phi : A \multimap B$, $\sigma : C \multimap D$, $\tau : A \multimap C$, $\psi : B \multimap C$ and $\theta : \oc A \multimap B$, then $\mathrm{cp}_A : A \multimap A$, $\phi \otimes \sigma : A \otimes C \multimap B \otimes D$, $\langle \phi, \tau \rangle : A \multimap B \mathbin{\&} C$, $\phi ; \psi : A \multimap C$ and $\theta^\dagger : \oc A \multimap \oc B$; also, $\mathrm{der}_B : B \Rightarrow B$ if $B$ is w.o. 
Moreover, $\mathrm{cp}_A$ (resp. $\mathrm{der}_B$) is winning and w.b. if $A$ is w.f. (resp. if $B$ is w.o. and w.f.), and $\otimes$, $\langle \_, \_ \rangle$, $\circ$ and $(\_)^\dagger$ preserve winning and well-bracketing. 
\end{lemma}
\begin{proof}
The only nontrivial point not shown in \cite{mccusker1998games} is that $\mathrm{cp}_A$ is noetherian if $A$ is w.f. (the case for $\mathrm{der}_B$ is the same).  
Note that $\mathrm{cp}_A$ is total, innocent and w.b. even if $A$ is not w.f.
Given $\boldsymbol{s}mm \in \mathrm{cp}_A$, we see by induction on $|\boldsymbol{s}|$ that  the P-view $\lceil \boldsymbol{s} m \rceil$ is of the form $m_1 m_1 m_2 m_2 \dots m_k m_k m$, and therefore there is a sequence $\star \vdash_{A} m_1 \vdash_{A} m_2 \dots  \vdash_{A} m_{k-1} \vdash_{A} m_k \vdash_{A} m$. 
Hence, $\mathrm{cp}_A$ is noetherian if $A$ is w.f.
\end{proof}

\begin{definition}[Categories of games \cite{mccusker1998games,abramsky1999game}]
\label{DefCategoriesOfGamesAndStrategies}
The category $\mathbb{G}$ consists of 
\begin{itemize}

\item W.o. games as objects;

\item Strategies on the implication $A \Rightarrow B$ as morphisms $A \rightarrow B$;

\item The composition $\psi \bullet \phi \colonequals \psi \circ \phi^\dagger : A \Rightarrow C$ of strategies as the composition of morphisms $\phi : A \rightarrow B$ and $\psi : B \rightarrow C$;

\item The dereliction $\mathrm{der}_A$ as the identity on each object $A$.

\end{itemize}

The subcategory $\mathbb{LG}$ (resp. $\mathbb{WG}$) of $\mathbb{G}$ consists of w.f., w.o. games as objects, and winning (resp. winning, w.b.) strategies as morphisms.
\end{definition}

Games in $\mathbb{G}$ (resp. $\mathbb{LG}$ and $\mathbb{WG}$) are \emph{w.o.} (resp. \emph{w.o.} and \emph{w.f.}) for the identities to be well-defined (Lemma~\ref{LemWellDefinedConstructionsOnStrategies}).
Strategies in $\mathbb{G}$ embody unconstrained, general \emph{computations}.
In contrast, strategies in $\mathbb{LG}$ (resp. $\mathbb{WG}$) are winning (resp. winning and w.b.), embodying \emph{proofs in classical logic} (resp. \emph{proofs in intuitionistic logic}). 

These categories are \emph{cartesian closed}, where a terminal object, product and exponential objects are the terminal game $T$, product $\mathbin{\&}$ and implication $\Rightarrow$, respectively.
Since our morphisms are the same as those in the cartesian closed categories of MC-games \cite{abramsky1999game}, they satisfy the equational axioms on cartesian closure in the same way.

Therefore, by Lemmata~\ref{LemWellDefinedConstructionsOnGames} and \ref{LemWellDefinedConstructionsOnStrategies}, we conclude:
\begin{theorem}[Well-defined cartesian closed categories of games]
\label{ThmWellDefinedGameSemanticCCCs}
The structures $\mathbb{G}$, $\mathbb{LG}$ and $\mathbb{WG}$ form cartesian closed categories.
\end{theorem}

\begin{notation}
We employ the following notations:
\begin{itemize}

\item Given a strategy $\sigma : G$, we write $\sigma^{T} : T \multimap G$ and $\sigma^{\oc T} : T \Rightarrow G$ for the evident strategies that coincide with $\sigma$ up to `tags';

\item Given strategies $\phi : T \multimap G$ and $\phi' : T \Rightarrow G$, we write $\phi_{T}, \phi'_{\oc T} : G$ for the evident strategies that coincide with $\phi$ and $\phi'$ up to `tags' respectively;

\item Given strategies $\psi : A \multimap B$ and $\alpha : A$, we define $\psi \circ \alpha  \colonequals (\psi \circ \alpha^{T})_{T} : B$;

\item Given strategies $\alpha : A$ and $\beta : B$, we define $\alpha \otimes \beta \colonequals ((\alpha^{T} \otimes \beta^{T}) \circ \iota)_{T} : A \otimes B$, where $\iota$ is the unique strategy on $T \multimap T \otimes T$, and $\langle \alpha, \beta \rangle \colonequals \langle \alpha^{T}, \beta^{T} \rangle_{T} : A \mathbin{\&} B$;

\item Given a strategy $\alpha : A$, we define $\alpha^\dagger \colonequals ((\alpha^{\oc T})^\dagger)_{\oc T} : \oc A$.


\end{itemize}
\end{notation}

\section{Predicate games}
\label{PredicateGames}
Having reviewed games and strategies in \S\ref{GamesAndStrategies}, let us now initiate our contributions.
Before going into details, we sketch our idea in the following paragraphs.
In short, our main challenge is to interpret Sigma-types without destroying the \emph{non-inductive} nature of game semantics or the \emph{additive} nature of product $\&$ (\S\ref{RelatedWorkAndOurContributions}), and we achieve it by generalising games. 
On the other hand, we keep strategies \emph{unchanged} so that our method retains the advantages of game semantics such as \emph{intensionality} (\S\ref{IntroGameSemantics}).

Naively, we can interpret each dependent type $\mathsf{x : C \vdash D(x) \ type}$ by a family $D = (D(\sigma))_{\sigma : C}$ of games $D(\sigma)$ indexed by strategies $\sigma$ on the game $C$ that models the simple type $\mathsf{C}$. 
In the presence of Sigma-types, dependent types with only a single variable cover those with more than one variable, and so we focus on the former. 

In light of product $\&$ (Definition~\ref{DefConstructionsOnGames}), which models a particular kind of Sigma-types, viz., product types, it seems a natural idea to model the Sigma-type $\mathsf{\Sigma_{x : C} D(x)}$ by a subgame $\Sigma(C, D) \subseteq C \mathbin{\&} \bigcup_{\sigma : C}D(\sigma)$ such that strategies on $\Sigma(C, D)$ are the pairings $\langle \sigma, \tau \rangle$ of $\sigma : C$ and $\tau : D(\sigma)$.
However, this idea does not work since
\begin{enumerate}

\item Each game $G$, by definition, determines the set $\mathrm{st}(G)$ of all strategies on $G$;

\item It is impossible for P, when playing on such a game $\Sigma(C, D)$, if any, to \emph{fix} a strategy $\sigma : C$, let alone a game $D(\sigma)$, at the beginning of a play. 

\end{enumerate}

As an example of the first problem, consider a dependent type $\mathsf{x : N \vdash N_b(x) \ type}$ such that the canonical terms of the simple type $\mathsf{N_b(\underline{k})}$ ($k \in \mathbb{N}$) are numerals $\mathsf{\underline{n}}$ such that $n \leqslant k$, and assume that we model $\mathsf{N_b}$ by the family $N_b = (N_b(\sigma))_{\sigma : N}$ of games $N_b(\sigma)$ defined by $N_b(\underline{k}) \colonequals \mathrm{Pref}(\{ \, qn \mid n \leqslant k \, \})$ ($k \in \mathbb{N}$) and $N_b(\bot) \colonequals N$.
However, there is no subgame $G \subseteq N \mathbin{\&} \bigcup_{\sigma : N}N_b(\sigma) = N \mathbin{\&} N$ such that $\langle \underline{k}, \underline{n} \rangle : G$ if and only if $\underline{n} : N_b(\underline{k})$ for all $k, n \in \mathbb{N}$ since if such a game $G$ existed, then $\langle \underline{0}, \underline{0} \rangle, \langle \underline{1}, \underline{1} \rangle : G$, which implies $\langle \underline{0}, \underline{1} \rangle : G$ by the definition of strategies on a game (Definition~\ref{DefStrategies}), a contradiction. 
Hence, no game can properly model the Sigma-type $\mathsf{\Sigma_{x:N}N_b(x)}$.

Let us next give an example of the second problem. 
Let $\mathsf{x : N \vdash List_N(x) \ type}$ be a dependent type such that the canonical terms of the simple type $\mathsf{List_N(\underline{k})}$ ($k \in \mathbb{N}$) are the $k$-lists of numerals, and assume that we model $\mathsf{List_N}$ by the family $\mathrm{List}_N = (\mathrm{List}_N(\sigma))_{\sigma : N}$ of games $\mathrm{List}_N(\sigma)$ such that $\mathrm{List}_N(\underline{k})$ ($k \in \mathbb{N}$) is the $k$-ary tensor $\otimes$ on $N$, where $\mathrm{List}_N(\underline{0}) \colonequals T$, and $\mathrm{List}_N(\bot) \colonequals \bigcup_{k \in \mathbb{N}} \mathrm{List}_N(\underline{k})$.
If there were a subgame $H \subseteq N \mathbin{\&} \bigcup_{\sigma : N} \mathrm{List}_N(\sigma)$ that models the Sigma-type $\mathsf{\Sigma_{x:N} List_N(x)}$, then the pairings $\langle \underline{k}, \underline{n_1} \otimes \underline{n_2} \otimes \cdots \otimes \underline{n_k} \rangle$ for all $k, n_1, n_2, \dots, n_k \in \mathbb{N}$ would be \emph{total} on $H$ since strategies for proofs must be winning (\S\ref{SubsectionGamesAndStrategies}). 
However, there is no such $H$ since O may select, by his first move, e.g., the $(k+1)$st component of $\mathrm{List}_N(\underline{k+1})$.

We have seen the two fundamental limitations of games in modelling Sigma-types.
We solve this problem by generalising games to pairs $\Gamma = (|\Gamma|, \| \Gamma \|)$ of a game $|\Gamma|$ and a family $\| \Gamma \| = (\Gamma(\gamma))_{\gamma : |\Gamma|}$ of subgames $\Gamma(\gamma) \subseteq | \Gamma |$, called \emph{predicate (p-) games}, and define strategies $\gamma$ \emph{on} $\Gamma$, written $\gamma : \Gamma$, to be those $\gamma : |\Gamma|$ satisfying $\overline{\gamma}_{\Gamma(\gamma)}^{\mathrm{Even}} : \Gamma(\gamma)$.
That is, a p-game $\Gamma$ is a game $|\Gamma|$ equipped with the \emph{specification} $\| \Gamma \|$ for strategies $\gamma : |\Gamma|$ to be on $\Gamma$: The \emph{restriction} of $\gamma$ to $\Gamma(\gamma)$, i.e., $\overline{\gamma}_{\Gamma(\gamma)}^{\mathrm{Even}}$, is a strategy on $\Gamma(\gamma)$.

A play in a p-game $\Gamma$ then proceeds as follows.
First, \emph{Judge (J)} asks P a question $q_\Gamma$ (`What is your strategy?') and P answers it by a strategy $\gamma : \Gamma$ (`It is $\gamma$!').
After this \emph{initial protocol} between J and P, an ordinary play on the game $\Gamma(\gamma)$ between P and O follows, in which P must employ the declared strategy $\gamma$ restricted to $\Gamma(\gamma)$, i.e., $\overline{\gamma}_{\Gamma(\gamma)}^{\mathrm{Even}} : \Gamma(\gamma)$.
Accordingly, $\gamma : \Gamma$ is \emph{winning} (resp. \emph{w.b.}) if so is $\overline{\gamma}_{\Gamma(\gamma)}^{\mathrm{Even}} : \Gamma(\gamma)$.

This generalisation of games to p-games solves the first problem as follows.
Let us define a p-game $\Sigma(N, N_b)$ by $|\Sigma(N, N_b)| \colonequals N \mathbin{\&} N$ and $\Sigma(N, N_b)(\langle \sigma, \tau \rangle) \colonequals \begin{cases} N \mathbin{\&} N_b(\underline{k}) &\text{if $\sigma = \underline{k}$ for some $k \in \mathbb{N}$} \\ N \mathbin{\&} N &\text{otherwise} \end{cases}$ for all $\langle \sigma, \tau \rangle : |\Sigma(N, N_b)|$.
Then, observe that strategies on the p-game $\Sigma(N, N_b)$ are the pairings $\langle \sigma, \tau \rangle : N \mathbin{\&} N$ such that $\tau : N_b(\underline{k})$ if $\sigma = \underline{k}$.
For instance, typical plays by the strategy $\langle \underline{7}, \underline{3} \rangle : \Sigma(N, N_b)$ are
\begin{small}
\begin{mathpar}
\begin{tabular}{ccc}
$\Sigma(N,$ & & $N_b)$ \\ \cline{1-3}
& $q_{\Sigma(N, N_b)}$ & \\
& $\langle \underline{7}, \underline{3} \rangle$ & \\
\tikzmark{csigma1} $q$&& \\
\tikzmark{dsigma1} $7$&&
\end{tabular}
\begin{tikzpicture}[overlay, remember picture, yshift=.25\baselineskip]
\draw [->] ({pic cs:dsigma1}) [bend left] to ({pic cs:csigma1});
\end{tikzpicture}
\and
\begin{tabular}{ccc}
$\Sigma(N,$ & & $N_b)$ \\ \cline{1-3}
&$q_{\Sigma(N, N_b)}$& \\
&$\langle \underline{7}, \underline{3} \rangle$& \\
&& $q$ \tikzmark{csigma2} \\
&& $3$ \tikzmark{dsigma2}
\end{tabular}
\begin{tikzpicture}[overlay, remember picture, yshift=.25\baselineskip]
\draw [->] ({pic cs:dsigma2}) [bend right] to ({pic cs:csigma2});
\end{tikzpicture}
\end{mathpar}
\end{small}where J first asks P the question $q_{\Sigma(N, N_b)}$ (`What is your strategy?'), and P answers it by $\langle \underline{7}, \underline{3} \rangle : \Sigma(N, N_b)$ (`It is $\langle \underline{7}, \underline{3} \rangle$!'); then, an ordinary play between P and O on the game $\Sigma(N, N_b)(\langle \underline{7}, \underline{3} \rangle) = N \mathbin{\&} N_b(\underline{7})$ follows, where P must play by $\overline{\langle \underline{7}, \underline{3} \rangle}_{N \mathbin{\&} N_b(\underline{7})}^{\mathrm{Even}} = \langle \underline{7}, \underline{3} \rangle$.
The arrows in the diagram represent pointers in j-sequences (Definition~\ref{DefJSequences}). 

Although the declaration of a strategy is not necessary in this example, it is clear why P cannot play by, e.g., $\langle \underline{0}, \underline{1} \rangle$ on $\Sigma(N, N_b)$: $\overline{\langle \underline{0}, \underline{1} \rangle}_{N \mathbin{\&} N_b(\underline{0})}^{\mathrm{Even}} \not\subseteq \Sigma(N, N_b)(\langle \underline{0}, \underline{1} \rangle)$.
In this way, the specification $\| \Sigma(N, N_b) \|$ solves the first problem by \emph{filtering} strategies.

Moreover, the declaration of a strategy solves the second problem: The p-game $\Sigma(N, \mathrm{List}_N)$ defined by $\Sigma(N, \mathrm{List}_N)(\underline{k}) \colonequals N \mathbin{\&} (\underbrace{N \otimes N \otimes \dots \otimes N}_k)$ for all $k \in \mathbb{N}$, $\Sigma(N, \mathrm{List}_N)(\bot) \colonequals \bigcup_{k \in \mathbb{N}} \Sigma(N, \mathrm{List}_N)(\underline{k})$ and $|\Sigma(N, \mathrm{List}_N)| \colonequals \Sigma(N, \mathrm{List}_N)(\bot)$ models the Sigma-type $\mathsf{\Sigma_{x:N} List_N(x)}$.
Typical plays in $\Sigma(N, \mathrm{List}_N)$ look like
\begin{small}
\begin{mathpar}
\begin{tabular}{ccc}
$\Sigma(N,$ & & $\mathrm{List}_N)$  \\ \cline{1-3}
& $q_{\Sigma(N, \mathrm{List}_N)}$ & \\
& $\langle \underline{2}, \underline{1} \otimes \underline{3} \rangle$ & \\
\tikzmark{csigma71} $q$&& \\
\tikzmark{dsigma71} $2$&& \\
&& \\
&& 
\end{tabular}
\begin{tikzpicture}[overlay, remember picture, yshift=.25\baselineskip]
\draw [->] ({pic cs:dsigma71}) [bend left] to ({pic cs:csigma71});
\end{tikzpicture}
\and
\begin{tabular}{cccccc}
$\Sigma(N,$ & & & &$\mathrm{List}_N)$  \\ \cline{1-6} 
&$q_{\Sigma(N, \mathrm{List}_N)}$&& \\
&$\langle \underline{2}, \underline{1} \otimes \underline{3} \rangle$&&&& \\
&&&&& $q$ \tikzmark{csigma72} \\
&&&&& $3$ \tikzmark{dsigma72} \\
&&& \tikzmark{csigma73} $q$ \\
&&& \tikzmark{dsigma73} $1$
\end{tabular}
\begin{tikzpicture}[overlay, remember picture, yshift=.25\baselineskip]
\draw [->] ({pic cs:dsigma72}) [bend right] to ({pic cs:csigma72});
\draw [->] ({pic cs:dsigma73}) [bend left] to ({pic cs:csigma73});
\end{tikzpicture}
\end{mathpar}
\end{small}where the declaration of the strategy $\langle \underline{2}, \underline{1} \otimes \underline{3} \rangle : \Sigma(N, \mathrm{List}_N)$ \emph{fixes} the underlying game $N \mathbin{\&} (N \otimes N)$.  
Thus, $\overline{\langle \underline{2}, \underline{1} \otimes \underline{3} \rangle}_{N \mathbin{\&} (N \otimes N)}^{\mathrm{Even}} = \langle \underline{2}, \underline{1} \otimes \underline{3} \rangle$ is \emph{total} on $\Sigma(N, \mathrm{List}_N)$.

In this way, we achieve game semantics of Sigma-types \emph{directly without the list construction} of Abramsky et al. \cite{abramsky2015games,vakar2018game} (\S\ref{RelatedWorkAndOurContributions}).
Let us add another important point that their method cannot properly interpret the Sigma-type $\mathsf{\Sigma(N, List_N)}$ since total strategies on their interpretation of this Sigma-type are the lists $(\underline{k}, \underline{n_1} \otimes \underline{n_2} \otimes \dots)$ of strategies, where the second component is an \emph{infinite} iteration of tensor $\otimes$, due to the \textrm{O-sat} operation \cite[Remark~4.5]{vakar2018game}; we come back to this point in \S\ref{DependentPairSpace}.

Besides, our interpretation of Sigma-types retains the \emph{additive} nature of product $\&$ because a play in the interpretation occurs in either side, not both, which is in contrast with the interpretation of Sigma-types by Blot and Laird \cite{blot2018extensional} (\S\ref{RelatedWorkAndOurContributions}).

Finally, we justify p-games, in particular the use of J, as a generalisation of games as follows. 
First, J is also assumed at least implicitly in conventional games as well since there must be someone other than P or O to check if j-sequences played by P and O are \emph{valid} positions in the underlying game. 
Hence, the use of J is not a big departure from games.
However, this argument is only \emph{conceptual}, and the following \emph{mathematical} arguments matter much more.
Second, even in the literature of game semantics, P always plays by a \emph{fixed} strategy as well. 
Besides, this predetermination of a strategy does not lose generality since each position $\boldsymbol{s}$ in a game $G$ is the result of a play by some $\sigma : G$, viz., $\sigma \colonequals \mathrm{Pref}(\{ \boldsymbol{s} \})^{\mathrm{Even}}$.
Third, the initial two elements played in p-games are by J and P, not O and P, so that p-games inherit the \emph{intensionality} of games: Strategies for P are revealed to O \emph{only gradually} along the development of a play. 
I.e., the declaration of a strategy by P is \emph{hidden} from O, so that O can see the strategy \emph{only gradually via his play against the strategy} (as in the case of games). 
Technically, we implement this idea by \emph{excluding the initial two elements from ordinary plays (especially from O-views)}. 
Dually, O's strategy on the domain $\Gamma$ of a linear implication $\Gamma \multimap \Delta$ between p-games is revealed \emph{only gradually} to P via plays.
Consequently, the extensions of strategies on $\Gamma \multimap \Delta$ are \emph{continuous} maps as in the case of games.
More generally, since our strategies are just the ordinary ones, our method inherits the \emph{intensionality} of standard game semantics; see \S\ref{Intensionality}--\ref{Independence}.

In summary, since the declarations of strategies by P and O are \emph{invisible} to each other, plays in p-games proceed just like those in games, and our approach inherits the intensionality of game semantics. 
We emphasise this intensionality since due to the declarations the reader may mistake our model to be close to extensional ones.

The rest of this section proceeds as follows.
We first define p-games in \S\ref{SubsectionPredicateGames}, and then generalise constructions on games (\S\ref{ConstructionsOnGamesAndStrategies}) to p-games in \S\ref{CCCsOfPredicateGamesAndStrategies}.
In these sections, we slightly modify and formalise the aforementioned interpretation of dependent types (Definition~\ref{DefPredicateGames}) as well as the examples $\mathrm{List}_N$ and $N_b$ (Example~\ref{ExpDependentPredicateGames}).

\subsection{Predicate games}
\label{SubsectionPredicateGames}
We reformulate the idea of p-games $\Gamma = (|\Gamma|, \| \Gamma \|)$ sketched above in a handier way as follows.
First, recall that the closure $\overline{\gamma}_{|\Gamma|}$ of a strategy $\gamma : |\Gamma|$ with respect to the ambient game $|\Gamma|$ is characterised by a simple form $\overline{\gamma}_{|\Gamma|} = \gamma \cup \{ \, \boldsymbol{s}m \in |\Gamma| \mid \boldsymbol{s} \in \gamma \, \}$.

Second, the central condition $\overline{\gamma}_{\Gamma(\gamma)}^{\mathrm{Even}} : \Gamma(\gamma)$ is equivalent to another relation $\overline{\gamma}_{|\Gamma|} \preccurlyeq \Gamma(\gamma)$, where the symbol $\preccurlyeq$ denotes Chroboczek's beautiful \emph{liveness ordering}:

\begin{definition}[Liveness ordering \cite{chroboczek2000game}]
\label{DefLivenessOrdering}
The \emph{\bfseries liveness ordering} is a partial order $\preccurlyeq$ between games \cite[Definition~8 and Theorem~9]{chroboczek2000game}, which defines $G \preccurlyeq H$ to mean that O (resp. P) is less (resp. more) restricted in $G$ than in $H$, i.e., they satisfy
\begin{enumerate}

\item If $\boldsymbol{s} \in (G \cap H)^{\mathrm{Even}}$ and $\boldsymbol{s}m \in H^{\mathrm{Odd}}$, then $\boldsymbol{s}m \in G^{\mathrm{Odd}}$;

\item If $\boldsymbol{t}l \in (G \cap H)^{\mathrm{Odd}}$ and $\boldsymbol{t}lr \in G^{\mathrm{Even}}$, then $\boldsymbol{t}lr \in H^{\mathrm{Even}}$.

\end{enumerate}
\end{definition}

\begin{proposition}[Liveness characterisation]
\label{PropLivenessCharacterisation}
Assume $\sigma : G$ and $H \in \mathrm{sub}(G)$.
\begin{enumerate}

\item $\overline{\sigma}_H^{\mathrm{Even}} : H$ if and only if $\overline{\sigma}_{G} \preccurlyeq H$;

\item If $\overline{\sigma}_{G} \preccurlyeq H$, then $\overline{\sigma}_{H}^{\mathrm{Even}} = \sigma \cap H$.

\end{enumerate}
\end{proposition}
\begin{proof}
We focus on the first clause since it is just a routine  to verify the second one.

First, it is straightforward to see that the relation $\overline{\sigma}_H^{\mathrm{Even}} : H$ is equivalent to the subset relation $\overline{\sigma}_H^{\mathrm{Even}} \subseteq H^{\mathrm{Even}}$. 
Hence, it suffices to show $\overline{\sigma}_H^{\mathrm{Even}} \subseteq H^{\mathrm{Even}}$ if and only if $\overline{\sigma}_{G} \preccurlyeq H$.
The sufficiency ($\Leftarrow$) is shown straightforwardly by the induction on the lengths of positions, which we leave to the reader; we focus on the necessity ($\Rightarrow$).

Assume $\overline{\sigma}_{G} \not\preccurlyeq H$; it remains to show $\overline{\sigma}_H^{\mathrm{Even}} \not\subseteq H^{\mathrm{Even}}$.
By $\overline{\sigma}_{G} \not\preccurlyeq H$, there is some $\boldsymbol{s}mn \in \overline{\sigma}_G^{\mathrm{Even}}$ such that $\boldsymbol{s}m \in \overline{\sigma}_G \cap H$ and $\boldsymbol{s}mn \not\in H$.
Note that $\boldsymbol{s}mn \in \overline{\sigma}_G^{\mathrm{Even}}$ implies $\boldsymbol{s}mn \in \sigma$.
We also see by induction on the lengths of positions that $\overline{\sigma}_G \cap H \subseteq \overline{\sigma}_H$ holds, whence $\boldsymbol{s}m \in \overline{\sigma}_H$.
Then, $\boldsymbol{s}mn \in \overline{\sigma}_H$ follows from $\boldsymbol{s}m \in \overline{\sigma}_H$ and $\boldsymbol{s}mn \in \sigma$.
We have shown $\overline{\sigma}_H^{\mathrm{Even}} \not\subseteq H^{\mathrm{Even}}$ since $\boldsymbol{s}mn \in \overline{\sigma}_H^{\mathrm{Even}}$ and $\boldsymbol{s}mn \not\in H^{\mathrm{Even}}$.
\end{proof}

Proposition~\ref{PropLivenessCharacterisation} reduces the relation $\overline{\gamma}_{\Gamma(\gamma)}^{\mathrm{Even}} : \Gamma(\gamma)$ to another $\overline{\gamma}_{|\Gamma|} \preccurlyeq \Gamma(\gamma)$, where $\overline{\gamma}_{|\Gamma|}$ is often easier to handle than $\overline{\gamma}_{\Gamma(\gamma)}^{\mathrm{Even}}$ by its simple form $ \overline{\gamma}_{|\Gamma|} = \gamma \cup \{ \, \boldsymbol{s}m \in |\Gamma| \mid \boldsymbol{s} \in \gamma \, \}$. 

Finally, if $\overline{\gamma}_{|\Gamma|} \preccurlyeq \Gamma(\gamma)$, then Proposition~\ref{PropLivenessCharacterisation} simplifies $\overline{\gamma}_{\Gamma(\gamma)}^{\mathrm{Even}}$ by $\overline{\gamma}_{\Gamma(\gamma)}^{\mathrm{Even}} = \gamma \cap \Gamma(\gamma)$.

We are now ready to introduce the central concept of the present work:
\begin{definition}[Predicate games]
\label{DefPredicateGames}
A \emph{\bfseries predicate (p-) game} is a pair $\Gamma = (|\Gamma|, \| \Gamma \|)$ of a game $|\Gamma|$ and a family $\| \Gamma \| = (\Gamma(\gamma))_{\gamma : |\Gamma|}$ of subgames $\Gamma(\gamma) \subseteq |\Gamma|$.
It is \emph{\bfseries well-founded (w.f.)} (resp. \emph{\bfseries well-opened (w.o.)}) if so is the game $|\Gamma|$.

A \emph{\bfseries strategy} on $\Gamma$, written $\gamma : \Gamma$, is a strategy $\gamma : |\Gamma|$ that satisfies $\overline{\gamma}_{|\Gamma|} \preccurlyeq \Gamma(\gamma)$.
It is \emph{\bfseries total} (resp. \emph{\bfseries innocent}, \emph{\bfseries noetherian}, \emph{\bfseries w.b.}) if so is the strategy $\gamma \cap \Gamma(\gamma) : \Gamma(\gamma)$.

Let $\mathrm{st}(\Gamma) \colonequals \{ \, \gamma \mid \gamma : \Gamma \, \}$ and $\overline{\gamma}_\Gamma \colonequals \overline{\gamma}_{\Gamma(\gamma)}$ ($\gamma : \Gamma$).
A \emph{\bfseries position} in $\Gamma$ is a prefix of a sequence $q_\Gamma \gamma \boldsymbol{s}$ such that $\gamma : \Gamma$ and $\boldsymbol{s} \in \overline{\gamma}_\Gamma$, where $q_\Gamma$ is an arbitrarily fixed element such that $q_\Gamma \not\in M_{|\Gamma|}$, $q_\Gamma \gamma$ is an \emph{\bfseries initial protocol}, and $\boldsymbol{s}$ is an \emph{\bfseries actual position}.
\end{definition}

A play in a p-game $\Gamma$ proceeds as follows.
First, \emph{\bfseries Judge (J)} asks P a question $q_\Gamma$ (`What is your strategy?'), and P answers it by a strategy $\gamma : \Gamma$ (`It is $\gamma$!').
After this initial protocol $q_\Gamma \gamma$, an ordinary play on the game $\Gamma(\gamma)$ between O and P follows, in which P must play by $\gamma$ restricted to $\Gamma(\gamma)$, i.e., $\gamma \cap \Gamma(\gamma) : \Gamma(\gamma)$.
The crucial points of p-games are that P can \emph{only} select a strategy $\gamma : |\Gamma|$ that satisfies $\overline{\gamma}_{|\Gamma|} \preccurlyeq \Gamma(\gamma)$ in the initial protocol, and the declaration of $\gamma : \Gamma$ \emph{fixes} the underlying game $\Gamma(\gamma)$.

In essence, the game-semantic counterpart of the path from STLC to MLTT (\S\ref{MLTT}) is the family $\| \Gamma \|$ added to the game $|\Gamma|$, which brings these \emph{strategy filtering} and \emph{game fixing} abilities to $|\Gamma|$.
As we have seen, the first and the second abilities address the first and the second problems listed at the beginning of \S\ref{PredicateGames}, respectively. 

If initial protocols were part of ordinary plays or visible to O, then by duality those on the domain of each linear implication would be visible to P, so that the extensions of strategies on linear implication may not be continuous. 
This \emph{extensionality} would be undesirable for the \emph{intensional} nature of game semantics.
We address this point by excluding initial protocols from actual positions. 
In particular, initial protocols are out of the scope of P- and O-views since pointers are only on \emph{actual} positions.

\begin{example}
\label{ExPGames}
Given a game $G$, we define the p-game $\mathscr{P}(G) \colonequals (G, \kappa_G)$, where $\kappa_G$ is the constant family at $G$.
Clearly, $\mathrm{st}(\mathscr{P}(G)) = \mathrm{st}(G)$.
We abbreviate $\mathscr{P}(T)$, $\mathscr{P}(\boldsymbol{0})$ and $\mathscr{P}(N)$ as $T$, $\boldsymbol{0}$ and $N$, and call them the \emph{\bfseries terminal p-game}, the \emph{\bfseries empty p-game} and the \emph{\bfseries natural number p-game}, respectively (cf. Example~\ref{ExamplesOfGames}).
\end{example}

\subsection{Cartesian closed categories of predicate games}
\label{CCCsOfPredicateGamesAndStrategies}
Next, we lift constructions on games (\S\ref{ConstructionsOnGamesAndStrategies}) to p-games. 
The cases of product $\mathbin{\&}$, tensor $\otimes$ and exponential $\oc$ are simple since we can construct them \emph{pointwisely}: 

\begin{notation}
Let $G$ be a game, $\boldsymbol{s} \in \oc G$ and $i \in \mathbb{N}$. 
We write $\boldsymbol{s} \upharpoonright i$ for the j-subsequence of $\boldsymbol{s}$ that consists of occurrences hereditarily justified by the $(i+1)$st initial occurrence in $\boldsymbol{s}$.
For instance, if $\boldsymbol{s} = q2q1q0 \in \oc N$, then $\boldsymbol{s} \upharpoonright 0 = q2$, $\boldsymbol{s} \upharpoonright 1 = q1$ and $\boldsymbol{s} \upharpoonright 2 = q0$.
\end{notation}

\begin{convention}
Given a strategy $\sigma$ on the tensor $G_0 \otimes G_1$ of games $G_i$ ($i = 0, 1$), let
\begin{equation}
\sigma \upharpoonright G_i \colonequals \begin{cases} \sigma_i &\text{if $\sigma = \sigma_0 \otimes \sigma_1$ for (necessarily unique) $\sigma_0 : G_0$ and $\sigma_1 : G_1$;} \\ \uparrow &\text{otherwise, where $\uparrow$ means being \emph{undefined}.} \end{cases}
\end{equation}

Similarly, given a strategy $\tau$ on the exponential $\oc G$ of a game $G$ and $j \in \mathbb{N}$, let
\begin{equation}
\tau \upharpoonright j \colonequals \begin{cases} \{ \, \boldsymbol{s} \upharpoonright j \mid \boldsymbol{s} \in \tau \, \} &\text{if $\{ \, \boldsymbol{s} \upharpoonright k \mid \boldsymbol{s} \in \tau \, \} : G$ for all $k \in \mathbb{N}$;} \\ \uparrow &\text{otherwise.} \end{cases}
\end{equation}

Given a p-game $\Gamma$, we define the value $\Gamma(\uparrow)$ to be \emph{undefined}, and the constructions $\otimes$, $\multimap$, $\&$ and $\oc$ on undefined games to be \emph{undefined}.
Finally, we extend the relation $\overline{\gamma}_{|\Gamma|} \preccurlyeq \Gamma(\gamma)$ by defining that \emph{it does not hold if the game $\Gamma(\gamma)$ is undefined}.
\end{convention}

\begin{definition}[Product and tensor on predicate games]
The \emph{\bfseries product} of p-games $\Gamma$ and $\Delta$ is the p-game $\Gamma \mathbin{\&} \Delta$ defined by $|\Gamma \mathbin{\&} \Delta| \colonequals |\Gamma| \mathbin{\&} |\Delta|$ and $(\Gamma \mathbin{\&} \Delta)(\langle \gamma, \delta \rangle) \colonequals \Gamma(\gamma) \mathbin{\&} \Delta(\delta)$ for all $\langle \gamma, \delta \rangle : |\Gamma \mathbin{\&} \Delta|$, and their \emph{\bfseries tensor} is the p-game $\Gamma \otimes \Delta$ defined by $|\Gamma \otimes \Delta| \colonequals |\Gamma| \otimes |\Delta|$ and $(\Gamma \otimes \Delta)(\sigma) \colonequals \Gamma(\sigma \upharpoonright |\Gamma|) \otimes \Delta(\sigma \upharpoonright |\Delta|)$ for all $\sigma : |\Gamma \otimes \Delta|$.
\end{definition}

\begin{definition}[Countable tensor]
\label{DefCountableTensor}
The \emph{\bfseries countable tensor} of a family $(G_i)_{i \in \mathbb{N}}$ of subgames $G_i \subseteq H$ is the subgame $\otimes_{i \in \mathbb{N}}G_i \colonequals \{ \, \boldsymbol{s} \in \oc H \mid \forall j \in \mathbb{N} . \, \boldsymbol{s} \upharpoonright j \in G_j \, \} \subseteq \oc H$.
\end{definition} 

\begin{definition}[Exponential of predicate games]
The \emph{\bfseries exponential} of a p-game $\Gamma$ is the p-game $\oc \Gamma$ defined by $|\oc \Gamma| \colonequals \oc |\Gamma|$ and $(\oc \Gamma)(\sigma) \colonequals \otimes_{i \in \mathbb{N}} \Gamma(\sigma \upharpoonright i)$ for all $\sigma : | \oc \Gamma |$.
\end{definition}

Hence, strategies on $\Gamma \mathbin{\&} \Delta$ are the pairings $\langle \gamma, \delta \rangle$ of $\gamma : \Gamma$ and $\delta : \Delta$.
Besides, by the above convention, strategies on $\Gamma \otimes \Delta$ are the tensors $\gamma \otimes \delta$ of $\gamma : \Gamma$ and $\delta : \Delta$, and strategies on $\oc \Gamma$ are those $\sigma : \oc |\Gamma|$ such that $\{ \, \boldsymbol{s} \upharpoonright i \mid \boldsymbol{s} \in \sigma \, \} : \Gamma$ for all $i \in \mathbb{N}$.

In contrast, we cannot apply the pointwise method to linear implication $\multimap$: If we define $|\Delta^\Gamma| \colonequals |\Delta|^{|\Gamma|}$ for the linear implication $\Delta^\Gamma$, then it is unclear how to decompose strategies $\phi : |\Delta|^\Gamma$ into those on $|\Gamma|$ and $|\Delta|$.
We solve this problem based on the elegant game semantics of universal quantification \cite[p.~17]{abramsky2005game}:
\begin{definition}[Linear implication and implication between predicate games]
\label{DefLinearImplicationBetweenPredicateGames}
The \emph{\bfseries linear implication} between p-games $\Gamma$ and $\Delta$ is the p-game $\Gamma \multimap \Delta$ (also denoted by $\Delta^\Gamma$) defined by $|\Delta^\Gamma| \colonequals |\Delta|^{|\Gamma|}$ and for all $\phi : |\Delta^\Gamma|$
\begin{small}
\begin{align*}
\label{LinearImplicationBetweenPredicateGames}
(\Delta^\Gamma)(\phi) &\colonequals \{ \boldsymbol{\epsilon} \} \cup \{ \, \boldsymbol{s}m \in |\Delta^\Gamma|^{\mathrm{Odd}} \mid \boldsymbol{s} \in (\Delta^\Gamma)(\phi), \exists \gamma : \Gamma . \, \boldsymbol{s}m \in \Delta(\phi \circ \gamma)^{\overline{\gamma}_{\Gamma}} \, \} \\
&\cup \{ \, \boldsymbol{t}lr \in |\Delta^\Gamma|^{\mathrm{Even}} \mid \boldsymbol{t}l \in (\Delta^\Gamma)(\phi), \forall \gamma : \Gamma . \, \boldsymbol{t}l \in \Delta(\phi \circ \gamma)^{\overline{\gamma}_{\Gamma}} \Rightarrow \boldsymbol{t}lr \in \Delta(\phi \circ \gamma)^{\overline{\gamma}_{\Gamma}} \, \},
\end{align*}
\end{small}and the \emph{\bfseries implication} between $\Gamma$ and $\Delta$ is the linear implication $\Gamma \Rightarrow \Delta \colonequals \oc \Gamma \multimap \Delta$.
\end{definition}

The first clause of the inductive definition of the subgame $(\Delta^\Gamma)(\phi) \subseteq |\Delta^\Gamma|$ is the base case.
Then, the second one specifies one of the two inductive steps: At an even-length position $\boldsymbol{s} \in (\Delta^\Gamma)(\phi)^{\mathrm{Even}}$, O \emph{can} make a move $m$ as in $\Delta(\phi \circ \gamma)^{\overline{\gamma}_{\Gamma}} \subseteq |\Delta^\Gamma|$ for \emph{any $\gamma : \Gamma$ not yet excluded}, i.e., $\boldsymbol{s} \in \Delta(\phi \circ \gamma)^{\overline{\gamma}_{\Gamma}}$. 
Finally, the third one stipulates the other inductive step: At an odd-length position $\boldsymbol{t}l \in (\Delta^\Gamma)(\phi)^{\mathrm{Odd}}$, the next move $r$ by $\phi$ \emph{must} be as in $\Delta(\phi \circ \gamma)^{\overline{\gamma}_{\Gamma}} \subseteq |\Delta^\Gamma|$ for \emph{any $\gamma : \Gamma$ not yet excluded}, i.e., $\boldsymbol{t}l \in \Delta(\phi \circ \gamma)^{\overline{\gamma}_{\Gamma}}$.


The basic idea is that in the subgame $\Delta^\Gamma(\phi) \subseteq |\Delta^\Gamma|$ O can play as in \emph{any} subgame $\Delta(\phi \circ \gamma)^{\overline{\gamma}_{\Gamma}} \subseteq |\Delta^\Gamma|$ not yet excluded; then, since P or $\phi$ should see what $\gamma : \Gamma$ is \emph{only via plays}, $\Delta^\Gamma(\phi)$ only allows P to play as in $\Delta(\phi \circ \gamma)^{\overline{\gamma}_{\Gamma}}$ for \emph{all} $\gamma$ not yet excluded.

In particular, O's strategy $\gamma : \Gamma$ on the domain $|\Gamma|$ at each moment determines its underlying game $\overline{\gamma}_\Gamma$, and P's play by $\phi : \Delta^\Gamma$ on the domain must be as in $\overline{\gamma}_\Gamma$ for \emph{all} possible $\gamma$ not yet excluded. 
This is crucial for composition of strategies between p-games to be well-defined (Lemma~\ref{LemWellDefinedConstructionsOnStrategiesBetweenPredicateGames}); see Footnote~\ref{FoofnoteComposition}. 
This additional subtlety is the main difference from the game semantics of universal quantification \cite{abramsky2005game}.

For instance, a strategy $\phi : |\Sigma(N_{[0]}, (\mathrm{List}_N)_{[1]})| \multimap |N_{[2]}|$ containing $q_{[2]} m_{[1]}$ is not on $\Sigma(N_{[0]}, (\mathrm{List}_N)_{[1]}) \multimap N_{[2]}$ whatever $m$ is since at the odd-length position $q_{[2]}$ the strategy $\langle \underline{0}, \top \rangle : \Sigma(N, (\mathrm{List}_N))$ is not excluded, and $\Sigma(N, \mathrm{List}_N)(\langle \underline{0}, \top \rangle) = N \mathbin{\&} T$.

As emphasised before, strategies $\phi : \Delta^\Gamma$ are indeed \emph{ordinary} ones (Definition~\ref{DefStrategies}), which just satisfy some additional axiom.
In particular, $\phi$ can see O's strategy $\gamma : \Gamma$ on the domain $\Gamma$ \emph{only gradually via plays}; i.e., the declaration of $\gamma$ by O is \emph{invisible} to $\phi$.
In this way, our approach retains the \emph{intensionality} of game semantics. 

\begin{lemma}[Well-defined constructions on predicate games]
\label{LemWellDefinedConstructionsOnPredicateGames}
P-games and w.f. p-games are closed under $\&$, $\otimes$, $\oc$ and $\multimap$, and w.o. ones under $\&$, $\multimap$ and $\Rightarrow$.  
\end{lemma}
\begin{proof}
Straightforward and left to the reader.
\end{proof}

\begin{lemma}[Well-defined copy-cats and derelictions between predicate games]
\label{LemWellDefinedCopyCatsAndDerelictions}
Suppose that $\Gamma$ is a p-game, and $\Delta$ is a w.o. p-game.
\begin{enumerate}

\item The copy-cat $\mathrm{cp}_{|\Gamma|}$ is a w.b. strategy on $\Gamma \multimap \Gamma$, and winning if $\Gamma$ is w.f.;

\item The dereliction $\mathrm{der}_{|\Delta|}$ is a w.b. strategy on $\Delta \Rightarrow \Delta$, and winning if $\Delta$ is w.f.

\end{enumerate}
\end{lemma}
\begin{proof}
We focus on the first clause since the second one is similar, where we require $\Delta$ to be w.o. for the same reason as the case of derelictions between games (\S\ref{ConstructionsOnGamesAndStrategies}).

We only show $(\overline{\mathrm{cp}_{|\Gamma|}})_{|\Gamma|^{|\Gamma|}} \preccurlyeq \Gamma^\Gamma(\mathrm{cp}_{|\Gamma|})$ as it is the only nontrivial point.
We do it by induction on the lengths of positions.
The base case and the inductive step on odd-length positions are trivial.
For the other inductive step, let $\boldsymbol{t}lr \in \mathrm{cp}_{|\Gamma|}$ and $\boldsymbol{t}l \in (\overline{\mathrm{cp}_{|\Gamma|}})_{|\Gamma|^{|\Gamma|}} \cap \Gamma^\Gamma(\mathrm{cp}_{|\Gamma|})$; we have to prove $\boldsymbol{t}lr \in \Gamma^\Gamma(\mathrm{cp}_{|\Gamma|})$.
By $\boldsymbol{t}l \in \Gamma^\Gamma(\mathrm{cp}_{|\Gamma|})$ and $\boldsymbol{t}lr \in \mathrm{cp}_{|\Gamma|}$, it follows that $\boldsymbol{t}lr$ satisfies the inductive condition for $\boldsymbol{t}lr \in \Gamma^\Gamma(\mathrm{cp}_{|\Gamma|})$.
\end{proof}

\begin{lemma}[Well-defined constructions on strategies between predicate games]
\label{LemWellDefinedConstructionsOnStrategiesBetweenPredicateGames}
Given strategies $\phi : \Gamma \multimap \Delta$, $\psi : \Delta \multimap \Theta$, $\sigma : \Theta \multimap \Xi$, $\tau : \Gamma \multimap \Theta$ and $\theta : \oc \Gamma \multimap \Delta$ between p-games, by applying Definition~\ref{DefConstructionsOnStrategies} to them,\footnote{Definition~\ref{DefConstructionsOnStrategies} is valid here since $\phi : |\Delta|^{|\Gamma|}$, $\psi : |\Theta|^{|\Delta|}$, $\sigma : |\Theta|^{|\Xi|}$, $\tau : |\Theta|^{|\Gamma|}$ and $\theta : |\Delta|^{\oc |\Gamma|}$.} we obtain strategies $\psi \circ \phi : \Gamma \multimap \Theta$, $\phi \otimes \sigma : \Gamma \otimes \Theta \multimap \Delta \otimes \Xi$, $\langle \phi, \tau \rangle : \Gamma \multimap \Delta \mathbin{\&} \Theta$ and $\theta^\dagger : \oc \Gamma \multimap \oc \Delta$ between p-games, and these constructions preserve winning and well-bracketing.\footnote{We go into details in the proof since 
this lemma is foundational for the rest of the present work. After proving the lemma, we are more sketchy on similar arguments.}
\end{lemma}
\begin{proof}
We focus on the tensor $\phi \otimes \sigma$ and the composition $\psi \circ \phi$ because the other constructions are simpler to verify.
First, let us show $\phi \otimes \sigma : \Gamma \otimes \Theta \multimap \Delta \otimes \Xi$, for which it suffices to verify $(\overline{\phi \otimes \sigma})_{|\Delta \otimes \Xi|^{|\Gamma \otimes \Theta|}} \preccurlyeq (\Delta \otimes \Xi)^{\Gamma \otimes \Theta}(\phi \otimes \sigma)$.
We prove it by induction on the lengths of positions. 
The base case and the inductive step on odd-length positions are trivial.
For the other inductive step, let $\boldsymbol{s}mn \in \phi \otimes \sigma$ and $\boldsymbol{s}m \in (\overline{\phi \otimes \sigma})_{|\Delta \otimes \Xi|^{|\Gamma \otimes \Theta|}} \cap (\Delta \otimes \Xi)^{\Gamma \otimes \Theta}(\phi \otimes \sigma)$; we have to show $\boldsymbol{s}mn \in (\Delta \otimes \Xi)^{\Gamma \otimes \Theta}(\phi \otimes \sigma)$.

By $\boldsymbol{s}m \in (\Delta \otimes \Xi)^{\Gamma \otimes \Theta}(\phi \otimes \sigma)$, it suffices to show $\boldsymbol{s}mn \in (\Delta \otimes \Xi)((\phi \otimes \sigma) \circ \varphi)^{\overline{\varphi}_{\Gamma \otimes \Theta}}$ for all $\varphi : \Gamma \otimes \Theta$ such that $\boldsymbol{s}m \in (\Delta \otimes \Xi)((\phi \otimes \sigma) \circ \varphi)^{\overline{\varphi}_{\Gamma \otimes \Theta}}$.
Fix such $\varphi$.
Note that $\varphi = \gamma \otimes \vartheta$ for unique $\gamma : \Gamma$ and $\vartheta : \Theta$, and $(\phi \otimes \sigma) \circ \varphi = (\phi \circ \gamma) \otimes (\sigma \circ \vartheta)$.
Hence, it suffices to show $\boldsymbol{s}mn \upharpoonright |\Delta^\Gamma| \in \Delta(\phi \circ \gamma)^{\overline{\gamma}_\Gamma}$ and $\boldsymbol{s}mn \upharpoonright |\Xi^\Theta| \in \Xi(\phi \circ \vartheta)^{\overline{\vartheta}_\Theta}$.
Assume $n$ in $\phi$; the other case is similar.
Then, we have $\boldsymbol{s}m \upharpoonright |\Delta^\Gamma| \in \Delta(\phi \circ \gamma)^{\overline{\gamma}_\Gamma}$ and $\boldsymbol{s}mn \upharpoonright |\Xi^\Theta| = \boldsymbol{s}m \upharpoonright |\Xi^\Theta| \in \Xi(\phi \circ \vartheta)^{\overline{\vartheta}_\Theta}$ by $\boldsymbol{s}m \in (\Delta \otimes \Xi)((\phi \otimes \sigma) \circ \varphi)^{\overline{\varphi}_{\Gamma \otimes \Theta}}$.
It remains to show $\boldsymbol{s}mn \upharpoonright |\Delta^\Gamma| \in \Delta(\phi \circ \gamma)^{\overline{\gamma}_\Gamma}$, but it follows from $\phi : \Delta^\Gamma$, $(\boldsymbol{s}m \upharpoonright |\Delta^\Gamma|).n \in \phi$ and $\boldsymbol{s}m \upharpoonright |\Delta^\Gamma| \in \overline{\phi}_{|\Delta^\Gamma|} \cap \Delta^\Gamma(\phi)$, where $(\boldsymbol{s}m \upharpoonright |\Delta^\Gamma|).n \in \phi$ follows from $\boldsymbol{s}mn \in \phi \otimes \sigma$, and $\boldsymbol{s}m \upharpoonright |\Delta^\Gamma| \in \overline{\phi}_{|\Delta^\Gamma|} \cap \Delta^\Gamma(\phi)$ from $(\boldsymbol{s}m \upharpoonright |\Delta^\Gamma|).n \in \phi$, $\boldsymbol{s} \upharpoonright |\Delta^\Gamma| \in \phi \cap \Delta^\Gamma(\phi) \subseteq \Delta^\Gamma(\phi)$ and $\boldsymbol{s}m \upharpoonright |\Delta^\Gamma| \in \Delta(\phi \circ \gamma)^{\overline{\gamma}_\Gamma}$.
We have shown $\phi \otimes \sigma : \Gamma \otimes \Theta \multimap \Delta \otimes \Xi$.
The same induction shows that $\phi \otimes \sigma$ is total if so are $\phi$ and $\sigma$, and it trivially follows from Lemma~\ref{LemWellDefinedConstructionsOnStrategies} that $\phi \otimes \sigma$ is innocent (resp. noetherian and w.b.) if so are $\phi$ and $\sigma$.

Let us proceed to show $\psi \circ \phi : \Gamma \multimap \Theta$, for which it suffices to show $(\overline{\psi \circ \phi})_{|\Theta|^{|\Gamma|}} \preccurlyeq \Theta^\Gamma(\psi \circ \phi)$.
Again, we prove it by induction on the lengths of positions, and the base case and the inductive step on odd-length positions are trivial.
For the remaining inductive step, assume $\boldsymbol{t}lr \in \psi \circ \phi$ and $\boldsymbol{t}l \in (\overline{\psi \circ \phi})_{|\Theta|^{|\Gamma|}} \cap (\Theta)^{\Gamma }(\psi \circ \phi)$; we have to show $\boldsymbol{t}lr \in (\Theta)^{\Gamma }(\psi \circ \phi)$.
Similarly to the case of tensor $\otimes$, fix $\gamma : \Gamma$ such that $\boldsymbol{t}l \in \Theta(\psi \circ \phi \circ \gamma)^{\overline{\gamma}_\Gamma}$; then it suffices to show $\boldsymbol{t}lr \in \Theta(\psi \circ \phi \circ \gamma)^{\overline{\gamma}_\Gamma}$.
Assume $l$ in $\Gamma$; the other case is analogous. 
Note that $\boldsymbol{t}l = (\boldsymbol{w} \upharpoonright |\Gamma|, |\Theta|) . l$ for some unique $\boldsymbol{w} \in \phi \parallel \psi$ (Definition~\ref{DefConstructionsOnStrategies}) by the \emph{covering lemma} \cite[p.~12]{abramsky1997semantics}. 
The computation of $\phi \parallel \psi$ on $\boldsymbol{w}l$ produces $r$ after playing a finite fragment $\boldsymbol{u}$ of a position in the intermediate game $\mathrm{Pref}(\mathrm{cp}_{|\Delta|})$; see \cite[p.~11]{abramsky1997semantics}.
Crucially, $\phi : \Delta^\Gamma$ and $\psi : \Theta^\Delta$ imply by induction that $\boldsymbol{u}$ is a suffix of a position in $(\overline{\mathrm{cp}_{|\Delta|}})_{\Delta(\phi \circ \gamma)^{\Delta(\phi \circ \gamma)}}$.\footnote{\label{FoofnoteComposition}In other words, the contact point $\subseteq |\Delta|$ between $\phi$ and $\psi$ \emph{matches} at each moment.}
Hence, $\boldsymbol{t}lr \in \Theta(\psi \circ \phi \circ \gamma)^{\overline{\gamma}_\Gamma}$.

Finally, the same inductive argument verifies that $\psi \circ \phi$ is winning if so are $\phi$ and $\psi$, and it trivially follows from Lemma~\ref{LemWellDefinedConstructionsOnStrategies} that $\psi \circ \phi$ is w.b. if so are $\phi$ and $\sigma$.
\end{proof}

We are now ready to summarise the present section by:
\begin{definition}[Categories of predicate games]
\label{DefGameSemanticCategories}
The category $\mathbb{PG}$ consists of

\begin{itemize}

\item W.o. p-games as objects;

\item Strategies on the implication $\Gamma \Rightarrow \Delta$ as morphisms $\Gamma \rightarrow \Delta$;

\item The composition $\psi \bullet \phi \colonequals \psi \circ \phi^\dagger : \Gamma \Rightarrow \Theta$ of strategies as the composition of morphisms $\phi : \Gamma \rightarrow \Delta$ and $\psi : \Delta \rightarrow \Theta$;

\item The dereliction $\mathrm{der}_{|\Gamma|} : \Gamma \Rightarrow \Gamma$ as the identity $\mathrm{id}_\Gamma$ on each object $\Gamma$.

\end{itemize}

The subcategory $\mathbb{LPG}$ (resp. $\mathbb{WPG}$) of $\mathbb{PG}$ consists of w.f., w.o. p-games as objects, and winning (resp. winning, w.b.) strategies as morphisms.
\end{definition}

\begin{remark}
As in the case of games \cite[\S3.3]{mccusker1998games}, there are the \emph{linear} counterparts of the categories $\mathbb{PG}$, $\mathbb{LPG}$ and $\mathbb{WPG}$, in which morphisms $\Gamma \rightarrow \Delta$ are strategies on the \emph{linear implication} $\Gamma \multimap \Delta$.
We skip them since they are not central in this article.

Nevertheless, our game semantics of MLTT (\S\ref{GameSemanticsOfMLTT})  together with these six categories implies that some combinations of dependent types and linearity/effects are \emph{already there} in game semantics.
We leave it as future work to study such combinations.
\end{remark}

Just like the categories of games (Definition~\ref{DefCategoriesOfGamesAndStrategies}), p-games in $\mathbb{PG}$ (resp. $\mathbb{LPG}$ and $\mathbb{WPG}$) are \emph{w.o.} (resp. \emph{w.f.} and \emph{w.o.}) for the identities to be well-defined.


\begin{theorem}[Well-defined cartesian closed categories of predicate games]
\label{ThmCCCsOfPredicateGames}
The structures $\mathbb{PG}$, $\mathbb{LPG}$ and $\mathbb{WPG}$ form cartesian closed categories. 
\end{theorem}
\begin{proof}
By Lemma~\ref{LemWellDefinedConstructionsOnStrategiesBetweenPredicateGames}, it suffices to focus on $\mathbb{PG}$.
The composition is well-defined by Lemma~\ref{LemWellDefinedConstructionsOnStrategiesBetweenPredicateGames}, and so are the identities by Lemma~\ref{LemWellDefinedCopyCatsAndDerelictions}. 
Since morphisms in $\mathbb{PG}$ are a certain class of those in $\mathbb{G}$, and the composition and the identities in $\mathbb{PG}$ are those in $\mathbb{G}$, the associativity and the unit law on $\mathbb{PG}$ follow from those on $\mathbb{G}$ (Theorem~\ref{ThmWellDefinedGameSemanticCCCs}).
The cartesian closure of $\mathbb{PG}$ is by Example~\ref{ExPGames} on $T$ and Lemmata~\ref{LemWellDefinedConstructionsOnPredicateGames} and \ref{LemWellDefinedConstructionsOnStrategiesBetweenPredicateGames}, where the required equations on morphisms again follow from those on $\mathbb{G}$.
\end{proof}

\begin{convention}
We write $\mathbb{PG}(\Gamma)$, $\mathbb{LPG}(\Gamma)$ and $\mathbb{WPG}(\Gamma)$ for the hom-sets $\mathbb{PG}(T, \Gamma)$, $\mathbb{LPG}(T, \Gamma)$ and $\mathbb{WPG}(T, \Gamma)$, respectively, for each object $\Gamma$, and do not distinguish $\Gamma$ and $\mathbb{PG}(\Gamma)$; e.g., we say that $\gamma : \Gamma$ is winning and w.b. if and only if $\gamma \in \mathbb{WPG}(\Gamma)$.

\end{convention}

Finally, let us show a categorically pleasing feature of p-games (Corollary~\ref{CorGameSemanticLimits}).
We focus on $\mathbb{WPG}$ for brevity, but the following argument is valid in $\mathbb{LPG}$ and $\mathbb{PG}$ as well.
First, we define an equivalence relation $\simeq$ between morphisms $\phi, \phi' : \Gamma \rightarrow \Delta$ by $\phi \simeq \phi' \ratio \Leftrightarrow \phi \bullet \gamma = \phi' \bullet \gamma : \Delta$ for all $\gamma : \Gamma$; i.e., $\phi \simeq \phi'$ means that $\phi$ and $\phi'$ are \emph{extensionally equal}.
Next, recall that a morphism $\phi : \Gamma \rightarrow \Delta$ is \emph{strict} if $m^{xy}n^{zw} \in \phi$ implies $n^{(z^{\bot})w} \in \Gamma$ \cite{laurent2002polarized}.
We then define a category $\mathbb{WPG}^\sharp_\simeq$ out of $\mathbb{WPG}$, in which objects are those of $\mathbb{WPG}$, and morphisms $\Gamma \rightarrow \Delta$ are the equivalence classes $[\phi]$ of strict morphisms $\phi : \Gamma \rightarrow \Delta$ in $\mathbb{WPG}$ modulo $\simeq$.
The composition of morphisms $[\phi] : \Gamma \rightarrow \Delta$ and $[\psi] : \Delta \rightarrow \Theta$ in $\mathbb{WPG}^\sharp_\simeq$ is given by $[\psi] \bullet [\phi] \colonequals [\psi \bullet \phi]$, and the identities in $\mathbb{WPG}^\sharp_\simeq$ by $\mathrm{id}_\Gamma \colonequals [\mathrm{der}_{|\Gamma|}]$ for all $\Gamma \in \mathbb{WPG}^\sharp_\simeq$.
We leave it to the reader as an easy exercise to prove that $\mathbb{WPG}^\sharp_\simeq$ forms a cartesian closed category.

\begin{corollary}[Game-semantic finite limits]
\label{CorGameSemanticLimits}
$\mathbb{WPG}^\sharp_\simeq$ has all finite limits.
\end{corollary}
\begin{proof}
It suffices to establish the equaliser of given morphisms $[\phi_1], [\phi_2] : \Gamma \rightrightarrows \Delta$.
Let us then show that the required equaliser is given by the pair $(\Theta, [\mathrm{der}_{|\Theta|}])$, where $\Theta$ is the p-game defined by $| \Theta | \colonequals | \Gamma |$ and $\Theta(\theta) \colonequals \begin{cases} \Gamma(\theta) &\text{if $\phi_1 \bullet \theta = \phi_2 \bullet \theta$;} \\ T &\text{otherwise,} \end{cases}$ for all $\theta : |\Theta|$.
The p-game $\Theta$ is well-defined since it does not depend on the choice of the representatives $\phi_i$ of the morphisms $[\phi_i]$ ($i = 1, 2$).
Note that $\gamma : \Theta$ if and only if $\phi_1 \bullet \gamma = \phi_2 \bullet \gamma$ or $\gamma = \{ \boldsymbol{\epsilon} \}$ for all $\gamma : \Gamma$.
Besides, we have $\phi_1 \bullet \{ \boldsymbol{\epsilon} \} = \phi_2 \bullet \{ \boldsymbol{\epsilon} \}$ thanks to the strictness of $\phi_i$ ($i = 1, 2$).
Hence, it follows that $[\phi_1] \bullet [\mathrm{der}_{|\Theta|}] = [\phi_2] \bullet [\mathrm{der}_{|\Theta|}]$.

Next, given an object $\Theta'$ and a morphism $[\varphi] : \Theta' \rightarrow \Gamma$ with $[\phi_1] \bullet [\varphi] = [\phi_2] \bullet [\varphi]$, we have $[\varphi] : \Theta' \rightarrow \Theta$ and $[\mathrm{der}_{|\Theta|}] \bullet [\varphi] = [\varphi]$, where we leave the details to the reader.
Finally, if another morphism $[\varphi'] : \Theta' \rightarrow \Theta$ satisfies $[\mathrm{der}_{|\Theta|}] \bullet [\varphi'] = [\varphi]$, then $[\varphi'] = [\mathrm{der}_{|\Theta|}] \bullet [\varphi'] = [\varphi]$. 
We have shown the universal property of $(\Theta, [\mathrm{der}_{|\Theta|}])$.
\end{proof}

We focus on \emph{strict} strategies in $\mathbb{WPG}^\sharp_\simeq$ for Corollary~\ref{CorGameSemanticLimits} since, e.g., there is no equaliser of the non-strict strategies $\underline{0}, \underline{1} : N \rightrightarrows N$.
We further take the extensional quotient $[\phi]$ of strict strategies $\phi$ in $\mathbb{WPG}^\sharp_\simeq$ since it seems that otherwise the category would not be finitely complete (though we have not shown this negative result).

This categorical structure is quite \emph{novel}.
For instance, if we apply the proof of Corollary~\ref{CorGameSemanticLimits} to $[\kappa_{\underline{0}}], [\mathbin{\geqslant}] : N \mathbin{\&} N \rightrightarrows N$ such that $\kappa_{\underline{0}}$ is any strict strategy such that $\kappa_{\underline{0}} \bullet \sigma = \underline{0}$ for all $\sigma : N \mathbin{\&} N$, and $\geqslant$ is any strict strategy such that $\mathbin{\geqslant} \bullet \mathbin{\langle \underline{k}, \underline{n} \rangle} = \underline{0}$ if and only if $k \geqslant n$ for all $k, n \in \mathbb{N}$, then their equaliser is the p-game that models the Sigma-type $\mathsf{\Sigma_{x:N}N_b(x)}$ (see the beginning of \S\ref{PredicateGames}).
Therefore, this construction is impossible for games (even if we focus on strict strategies and take their quotient).

Corollary~\ref{CorGameSemanticLimits} is also \emph{useful}: It enables us to internalise a certain notion of $\infty$-groupoids in $\mathbb{WPG}^\sharp_\simeq$, which is a key step to extend this work to HoTT \cite{yamada2021game}.
Note that we do \emph{not} focus on strict strategies or take their quotient when we interpret HoTT, and so it is a proper extension of our game semantics of MLTT to HoTT.

\section{Game semantics of Martin-L\"of type theory}
\label{GameSemanticsOfMLTT}
We are now ready to present our game semantics of MLTT. 
Concretely, we show that the CCC $\mathbb{WPG}$ forms abstract semantics of MLTT: a \emph{category with families (CwF)} \cite{dybjer1996internal}.
CwFs are closer to the syntax than other abstract semantics, so that we can directly see the semantic counterparts of MLTT. 
In fact, we even regard a CwF as another presentation of MLTT as Clairambault and Dybjer \cite{clairambault2014biequivalence} do, and so we only show that $\mathbb{WPG}$ forms a CwF, leaving how a CwF models MLTT to \cite{hofmann1997syntax}.

Specifically, we prove that the CCC $\mathbb{WPG}$ gives rise to a CwF equipped with \emph{semantic type formers} \cite{hofmann1997syntax} for One-, Zero-, N-, Pi-, Sigma- and Id-types, establishing game semantics of MLTT equipped with these types.

The rest of this section proceeds as follows.
We model dependent types in \S\ref{DLPGames}, Pi-types in \S\ref{DependentFunctionSpace}, and Sigma-types in \S\ref{DependentPairSpace}. 
We then show that $\mathbb{WPG}$ forms a CwF in \S\ref{GameSemanticCwF} and equip it with all the semantic type formers in \S\ref{GameSemanticTypeFormers}.
Finally, we analyse the intensionality of our game semantics in \S\ref{Intensionality}, prove the independence of Markov's principle from MLTT in \S\ref{Independence}, for which we use $\mathbb{WPG}$, not $\mathbb{PG}$ or $\mathbb{LPG}$, and extend the game semantics of MLTT to subtyping on dependent types in \S\ref{GameSemanticsOfSubtyping}.

\subsection{Dependent predicate games}
\label{DLPGames}
First, we interpret \emph{dependent types} by \emph{w.o., w.f. dependent p-games}:
\begin{definition}[Dependent predicate games]
\label{DefDependentPredicateGames}
A \emph{\bfseries linearly dependent predicate (p-) game} over a p-game $\Gamma$ is a pair $L = (|L|, \| L \|)$ of a game $|L|$ and a family $\| L \| = (L(\gamma_0))_{\gamma_0 \in \mathbb{WPG}(\Gamma)}$ of p-games $L(\gamma_0)$ such that $|L(\gamma_0)| = |L|$.
It is \emph{\bfseries well-opened (w.o.)} (resp. \emph{\bfseries well-founded (w.f.)}) if so is the game $|L|$.

The \emph{\bfseries extension} of the family $\| L \|$ is the family $L^\star = (L^\star(\gamma))_{\gamma : \Gamma}$ of p-games $L^\star(\gamma)$ defined by $L^\star(\gamma) \colonequals \begin{cases} L(\gamma) &\text{if $\gamma \in \mathbb{WPG}$;} \\ \mathscr{P}(|L|) &\text{otherwise (cf.~Example~\ref{ExPGames}).} \end{cases}$

A \emph{\bfseries dependent predicate (p-) game} over $\Gamma$ is a linearly dependent one over $\oc \Gamma$.
\end{definition}

\begin{notation}
We write $\mathscr{D}_\ell(\Gamma)$ (resp. $\mathscr{D}_\ell^{\mathrm{w}}(\Gamma)$) for the set of all linearly dependent p-games (resp. w.o., w.f. ones) over $\Gamma$, and $\{ \Gamma' \}_\Gamma$ or $\{ \Gamma' \}$ for the \emph{constant} one at $\Gamma'$, i.e., $\{ \Gamma' \}_\Gamma \colonequals (\Gamma', \gamma : \Gamma \mapsto \Gamma')$.
We define $\mathscr{D}(\Gamma) \colonequals \mathscr{D}_\ell(\oc \Gamma)$ and $\mathscr{D}^{\mathrm{w}}(\Gamma) \colonequals \mathscr{D}^{\mathrm{w}}_\ell(\oc \Gamma)$.

We often write $\gamma_0^\dagger$ for an arbitrary element of $\mathbb{WPG}(\oc \Gamma)$, where $\gamma_0 \in \mathbb{WPG}(\Gamma)$, since elements of $\mathbb{WPG}(\oc \Gamma)$ are all \emph{innocent} and so promotions of elements of $\mathbb{WPG}(\Gamma)$.
\end{notation}

We explain Definition~\ref{DLPGames} as follows.
First, we define p-games $\mathcal{U}$ to model universes and encode w.o., w.f. dependent p-games $A$ over a p-game $\Gamma$ by morphisms $\phi_A : \Gamma \rightarrow \mathcal{U}$ in $\mathbb{WPG}$ in a forthcoming article.
Thus, naively $\|A\|$ would be the map $\gamma : \oc \Gamma \mapsto \mathrm{El}(\phi_A \circ \gamma)$, where $\mathrm{El}(\mu) \in \mathbb{WPG}$ is the p-game encoded by each $\mu \in \mathbb{WPG}(\mathcal{U})$.
However, $\phi_A \circ \gamma$ can be \emph{invalid as an encoding}, i.e., $\phi_A \circ \gamma \not \in \mathbb{WPG}(\mathcal{U})$, for instance when $\phi_A \circ \gamma$ is partial, where the value $A(\gamma) = \mathrm{El}(\phi_A \circ \gamma)$ is \emph{undefined}.
This suggests us to restrict $\|A\|$ to $\mathbb{WPG}(\oc \Gamma)$, so that $\mathrm{El}(\phi_A \bullet \gamma_0) \in \mathbb{WPG}$ for all $\gamma_0^\dagger \in \mathbb{WPG}(\oc \Gamma)$.
Accordingly, $\| A \|$ is \emph{trivially continuous} since elements of $\mathbb{WPG}(\oc \Gamma)$ are all \emph{total}.\footnote{This point matches the remark by Abramsky et al. \cite[Footnote~5]{vakar2018game} that the continuity of their interpretation of dependent types does not play any roles.}

Then, however, we also need the ambient game $|A|$.
For instance, if $A = \{ N \}_{\oc \boldsymbol{0}}$, then the family $\|A\|$ is \emph{empty}, and so $\| A \|$ cannot retain the natural number p-game $N$.
We address this problem by adding the game $|A|$ and extending $\|A\|$ to $A^\star$. 

\begin{example}
\label{ExpDependentPredicateGames}
We slightly modify and formalise the examples $\mathrm{List}_N, N_b \in \mathscr{D}^{\mathrm{w}}(N)$ by $\mathrm{List}_N(\underline{k}^\dagger) \colonequals \underbrace{N \otimes N \otimes \dots \otimes N}_k$ for all $k \in \mathbb{N}$ and $|\mathrm{List}_N| \colonequals \bigcup_{k \in \mathbb{N}} \mathrm{List}_N(\underline{k}^\dagger)$, as well as $|N_b| \colonequals N$ and $N_b(\underline{k}^\dagger) \colonequals \mathrm{Pref}(\{ \, qn \mid n \leqslant k \, \})$ for all $k \in \mathbb{N}$. 
\end{example}

\subsection{Pi on dependent predicate games}
\label{DependentFunctionSpace}
We next interpret \emph{Pi-types}.
Our idea is best explained by the set-theoretic analogy as follows. 
Given a dependent type $\mathsf{x : C \vdash D(x)}$, the Pi-type $\mathsf{\Pi_{x:C}D(x)}$ is something like the set of all functions $f$ from $C$ to $\bigcup_{x : C} D(x)$ such that $f(x) \in D(x)$ for all $x \in C$, called \emph{dependent maps} from $C$ to $D$, where recall that the set-theoretic semantics interprets simple types $\mathsf{C}$ and terms $\mathsf{x : C}$ as sets $C$ and elements $x \in C$, respectively, and dependent types $\mathsf{D}$ over $\mathsf{C}$ as families $D = (D(x))_{x \in C}$ of sets $D(x)$.

Hence, in light of implication $\Rightarrow$ between p-games (Definition~\ref{DefLinearImplicationBetweenPredicateGames}), it is now clear how to model Pi-types.
We interpret Pi-types by the following \emph{pi} $\Pi$:
\begin{definition}[Linear-pi and pi]
\label{DefLinearPiSpaces}
Let $L$ be a linearly dependent p-game over a p-game $\Gamma$, and $A$ a dependent p-game over $\Gamma$.
The \emph{\bfseries linear-pi} from $\Gamma$ to $L$ is the p-game $\Pi_\ell (\Gamma, L)$ defined by $|\Pi_\ell (\Gamma, L)| \colonequals |L|^{|\Gamma|}$ and for all $\phi : |\Pi_\ell (\Gamma, L)|$
\begin{footnotesize}
\begin{align*}
\Pi_\ell (\Gamma, L)(\phi) &\colonequals \{ \boldsymbol{\epsilon} \} \cup \{ \, \boldsymbol{s}m \in |\Pi_\ell (\Gamma, L)|^{\mathrm{Odd}} \mid \boldsymbol{s} \in \Pi_\ell (\Gamma, L)(\phi), \exists \gamma : \Gamma . \, \boldsymbol{s}m \in L^\star(\gamma)(\phi \circ \gamma)^{\overline{\gamma}_\Gamma} \, \} \\
&\cup \{ \, \boldsymbol{t}lr \in |\Pi_\ell (\Gamma, L)|^{\mathrm{Even}} \mid \boldsymbol{t}l \in \Pi_\ell (\Gamma, L)(\phi), \forall \gamma : \Gamma . \, \boldsymbol{t}l \in L^\star(\gamma)(\phi \circ \gamma)^{\overline{\gamma}_\Gamma} \Rightarrow \boldsymbol{t}lr \in L^\star(\gamma)(\phi \circ \gamma)^{\overline{\gamma}_\Gamma} \, \},
\end{align*}
\end{footnotesize}and the \emph{\bfseries pi} from $\Gamma$ to $A$ is the linear-pi $\Pi (\Gamma, A) \colonequals \Pi_\ell(\oc \Gamma, A)$. 
\end{definition}

The idea of linear-pi is that it is linear implication between p-games except that it also satisfies \emph{type dependency}.
Specifically, the codomain of a linear-pi $\Pi_\ell (\Gamma, L)$ is the p-game $L(\gamma)$ if the strategy $\gamma : \Gamma$ on the domain satisfies $\gamma \in \mathbb{WPG}(\Gamma)$, and the constant one $\mathscr{P}(|L|)$ otherwise.
We then define pi out of linear-pi and exponential in the same way as we define implication out of linear implication and exponential.

Accordingly, linear-pi (resp. pi) generalises linear implication (resp. implication): Given p-games $\Gamma$ and $\Gamma'$, $\Pi_\ell (\Gamma, \{ \Gamma' \}_\Gamma) = \Gamma \multimap \Gamma'$ (resp. $\Pi (\Gamma, \{ \Gamma' \}_{\oc \Gamma}) = \Gamma \Rightarrow \Gamma'$).

Essentially the same proof as the one of Lemma~\ref{LemWellDefinedConstructionsOnPredicateGames} on linear implication shows:
\begin{theorem}[Well-defined linear-pi]
Given a (w.o., w.f.) linearly dependent p-game $L$ over a (w.o., w.f.) p-game $\Gamma$, the linear-pi $\Pi_\ell(\Gamma, L)$ is a (w.o., w.f.) p-game.
\end{theorem}

This theorem also implies that, given a (w.o., w.f.) dependent p-game $A$ over a (w.o., w.f.) p-game $\Gamma$, the pi $\Pi(\Gamma, A)$ is a (w.o., w.f.) p-game.
However, we have to handle the case where $\Gamma$ is a dependent p-game. 
We address this point in \S\ref{Prod}.

\begin{example}
\label{ExDependentFunctionSpace1}
A strategy $\zeta : \Pi(N, \mathrm{List}_N)$ plays as the dependent map $n \in \mathbb{N} \mapsto (0, 0, \dots, 0) \in \mathbb{N}^n$ as follows. 
If O makes the first move $q_{[k]}$ ($k \in \mathbb{N}^+$) on the codomain $|\mathrm{List}_N|$, where $(\_)_{[k]}$ is the `tag' for tensor $N_{[1]} \otimes N_{[2]} \otimes \dots$, then $\zeta$ asks a question $q_{[0]}$ on the domain $\oc N_{[0]}$, where $(\_)_{[0]}$ is another `tag' for clarity. 
Finally, if O plays by $q_{[k]}q_{[0]} \mapsto n_{[0]}$ for any $n \in \mathbb{N}^+$, then $\zeta$ plays by $q_{[k]}q_{[0]}n_{[0]} \mapsto 0_{[k]}$.
If $k \leqslant n$, then $\underline{n}^\dagger \in \mathbb{WPG}(\oc N)$ on the domain is not yet excluded; $\zeta$ is compatible with this possibility since its computation so far is within the subgame $N \Rightarrow \mathrm{List}_N(\underline{n}^\dagger) \subseteq |\Pi(N, \mathrm{List}_N)|$.

At this point, O can further play on the pi $\Pi(N, \mathrm{List}_N)$.
For instance, O can make another move $q_{[k']}$ on the codomain such that $k \neq k'$; then, $\zeta$ again makes the move $q_{[0]}$ on the domain.
If O answers it by $n'_{[0]}$ for any $n' \in \mathbb{N}^+$, then $\zeta$ plays $0_{[k']}$ on the codomain.
If $n \neq n'$, $k > n$ or $k' > n$, then $\zeta$ no longer has to play within the subgame $N \Rightarrow \mathrm{List}_N(\underline{n}^\dagger)$.
In contrast, if $n = n'$, $k \leqslant n$ and $k' \leqslant n$, then $\zeta$ still has to play on the subgame $N \Rightarrow \mathrm{List}_N(\underline{n}^\dagger)$, and its computation still satisfies this condition.
In this way, a play by $\zeta$ proceeds.
Hence, we indeed have $\zeta : \Pi(N, \mathrm{List}_N)$.
\end{example}

\subsection{Sigma on dependent predicate games}
\label{DependentPairSpace}
We next interpret \emph{Sigma-types}.
Recall that by the set-theoretic analogy the Sigma-type $\mathsf{\Sigma_{x:C}D(x) \ type}$ represents the set of all pairs $(c, d)$ such that $c \in C$ and $d \in D(c)$, called \emph{dependent pairs} on $C$ and $D$.
Hence, we model Sigma-types by:
\begin{definition}[Sigma]
\label{DefSigma}
The \emph{\bfseries sigma} of a p-game $\Gamma$ and a dependent p-game $A$ over $\Gamma$ is the p-game $\Sigma (\Gamma, A)$ defined by $|\Sigma (\Gamma, A)| \colonequals |\Gamma| \mathbin{\&} |A|$ and $\Sigma (\Gamma, A)(\langle \gamma, \alpha \rangle) \colonequals \Gamma(\gamma) \mathbin{\&} A^\star(\gamma^\dagger)(\alpha)$ for all $\langle \gamma, \alpha \rangle : |\Sigma(\Gamma, A)|$. 
\end{definition}

The basic idea of a sigma $\Sigma (\Gamma, A)$ is that strategies on $\Sigma (\Gamma, A)$ must be the pairings $\langle \gamma, \alpha \rangle : |\Gamma| \mathbin{\&} |A|$ that satisfy $\gamma : \Gamma$ and $\alpha : A(\gamma^\dagger)$ whenever $\gamma \in \mathbb{WPG}(\Gamma)$. 
The second condition matches the constraint on the pi $\Pi(\Gamma, A)$ as we shall see shortly.

When $A$ is a constant dependent p-game $\{ \Gamma' \}_{\oc \Gamma}$, the sigma $\Sigma (\Gamma, \{ \Gamma' \}_{\oc \Gamma})$ coincides with the product $\Gamma \mathbin{\&} \Gamma'$. 
Thus, sigma $\Sigma$ generalises product $\mathbin{\&}$ on p-games.

\begin{theorem}[Well-defined sigma]
\label{ThmWellDefinedSigma}
Given a (w.o., w.f.) dependent p-game $A$ over a (w.o., w.f.) p-game $\Gamma$, the sigma $\Sigma(\Gamma, A)$ is a (w.o., w.f.) p-game.
\end{theorem}
\begin{proof}
Straightforward and left to the reader.
\end{proof}

Let $\Gamma \in \mathbb{WPG}$ and $A \in \mathscr{D}^{\mathrm{w}}(\Gamma)$. 
Then, $\Sigma(\Gamma, A) \in \mathbb{WPG}$, and it attains the following nontrivial point.
First, the categorical view on semantics of MLTT (Definition~\ref{DefCwFs}) tells us that there must be a bijection between strategies $\psi \in \mathbb{WPG}(\Delta, \Sigma(\Gamma, A))$ and pairs $(\phi, \Check{\alpha})$ of strategies $\phi \in \mathbb{WPG}(\Delta, \Gamma)$ and $\Check{\alpha} \in \mathbb{WPG}(\Pi(\Delta, A\{\phi\}))$, where the dependent p-game $A\{ \phi \} \in \mathscr{D}^{\mathrm{w}}(\Delta)$ is defined by $|A\{ \phi \}| \colonequals |A|$ and $A\{\phi\}(\delta) \colonequals A(\phi^\dagger \circ \delta)$ ($\delta \in \mathbb{WPG}(\oc \Delta)$).
Next, the constraint on the codomain $A\{\phi\}$ of $\Pi(\Delta, A\{\phi\})$ is \emph{gradually} revealed along the gradual specification of a strategy $\delta$ on the domain $\oc \Delta$.
However, since $\Sigma(\Gamma, A)$ itself cannot refer to $\Delta$, unlike $\Pi(\Delta, A\{\phi\})$, the challenge is to define $\Sigma(\Gamma, A)$ in such a way that strategies $\psi \in \mathbb{WPG}(\Delta, \Sigma(\Gamma, A))$ achieve the bijection.\footnote{This problem is \emph{vacant} in Abramsky et al. \cite{abramsky2015games,vakar2018game} as they \emph{define} morphisms $\Delta \rightarrow \Sigma(\Gamma, A)$ to be the pairs $(\phi, \Check{\alpha})$ at the cost of the non-inductive nature of game semantics.}
Then, the sigma $\Sigma(\Gamma, A)$ in fact \emph{meets} this requirement (Theorem~\ref{ThmWellDefinedWPG}).

\begin{example}
\label{ExpSigma}
The sigmas $\Sigma(N, \mathrm{List}_N)$ and $\Sigma(N, N_b)$ respectively formalise the two examples given at the beginning of \S\ref{PredicateGames}.
\end{example}

The pairings $\langle \underline{k}, \underline{n_1} \otimes \underline{n_2} \otimes \dots \otimes \underline{n_k} \rangle$, where $k, n_1, n_2, \dots, n_k \in \mathbb{N}$, are all \emph{winning} on $\Sigma(N, \mathrm{List}_N)$ since the declaration of a strategy in a p-game can control odd-length positions.
Besides, given a morphism $\phi : \Sigma(N, \mathrm{List}_N) \rightarrow \Delta$ in $\mathbb{WPG}$, the composition $T \stackrel{\langle \underline{k}, \underline{n_1} \otimes \underline{n_2} \otimes \dots \otimes \underline{n_k} \rangle}{\rightarrow} \Sigma(N, \mathrm{List}_N) \stackrel{\phi}{\rightarrow} \Delta$ is well-defined and \emph{winning} (Lemma~\ref{LemWellDefinedConstructionsOnStrategiesBetweenPredicateGames}).

In contrast, the list $(\underline{k}, \underline{n_1} \otimes \underline{n_2} \otimes \dots \otimes \underline{n_k})$ is \emph{partial} on the list $(N, \mathrm{List}_N)$ in Abramsky et al. \cite{abramsky2015games,vakar2018game}, where we regard $N$ and the components of $\mathrm{List}_N$ as games in the evident way, since there is no way in their method to prevent O from playing \emph{arbitrarily} on the ambient game $|\mathrm{List}_N|$.
For instance, $(\underline{0}, \top)$ is partial on $(N, \mathrm{List}_N)$. 

Note that the underlying list of (families) games for $(\underline{k}, \underline{n_1} \otimes \underline{n_2} \otimes \dots \otimes \underline{n_k})$ cannot be $(N, \{\mathrm{List}_N(\underline{k}^\dagger)\})$ as the composition of their morphisms $(\underline{k}, \underline{n_1} \otimes \underline{n_2} \otimes \dots \otimes \underline{n_k}) : T \rightarrow (N, \{\mathrm{List}_N(\underline{k}^\dagger)\})$ and $\psi : (N, \mathrm{List}_N) \rightarrow \Theta$ is \emph{ill-defined}. 
Also, we cannot replace the domain of $\psi$ with $(N, \{\mathrm{List}_N(\underline{k}^\dagger)\})$ either since it cannot handle the other values of $k \in \mathbb{N}$. 
Hence, the codomain of $(\underline{k}, \underline{n_1} \otimes \underline{n_2} \otimes \dots \otimes \underline{n_k})$ must be $(N, \mathrm{List}_N)$.

For this problem, they instead take the lists $(\underline{k}, \tau)$ of winning strategies $\underline{k} : N$ and $\tau : |\mathrm{List}_N|$. 
However, $\tau$ is \emph{redundant} as it is total on $|\mathrm{List}_N|$; e.g., $(\underline{0}, \top)$ is replaced with any $(\underline{0}, \underline{n_1} \otimes \underline{n_2} \otimes \cdots)$ such that $n_i \in \mathbb{N}$ for all $i \in \mathbb{N}$ in their approach, where $\underline{n_1} \otimes \underline{n_2} \otimes \cdots$ is an \emph{infinite} iteration of tensor $\otimes$, but neither of them is canonical. 
Consequently, their approach cannot properly model the Sigma-type $\mathsf{\Sigma(N, List_N)}$.

Technically, Abramsky et al. employ the \textrm{O-sat} operation \cite[Remark~4.5 and Theorem~5.6]{vakar2018game}, which allows O to \emph{ignore the type dependency} of Pi- and Sigma-types, and impose winning on strategies against such unrestricted plays by O.
However, this method generates a significant \emph{gap} between their model and MLTT since terms $\mathsf{\vdash \langle a, b \rangle : \Sigma(A, B)}$ satisfy $\mathsf{\vdash a : A}$ and $\mathsf{\vdash b : B(a)}$, not $\mathsf{x : A \vdash b : B(x)}$.
How does their full completeness hold then?
The answer is that they focus on a very specific class of finite inductive types \cite[Figure~7]{vakar2018game} for which the problem just disappears.
However, as we have just seen, their method is unsuited for more standard types such as $\mathsf{List_N}$, which is generated by the elimination rule of N-type in the presence of universes.


In summary, the novel mathematical structure of p-games enables us to not only dispense with the list construction but also \emph{accurately} model Pi- and Sigma-types. 

Finally, recall that the interpretation of Sigma-types by Blot and Laird \cite{blot2018extensional} does not preserve the additive nature of product $\&$ on games (\S\ref{RelatedWorkAndOurContributions}).
In contrast, sigma $\Sigma$ preserves the additive nature.
The challenge we have overcome is that the additive nature requires that a play in $\Sigma(\Gamma, A)$ is \emph{either} a play in $\Gamma$ or $A$, but then we have to specify the constraint on plays in $A$ for $\Sigma(\Gamma, A)$ \emph{even without playing on $\Gamma$}.


Again, however, this construction of sigma $\Sigma$ is not general enough for the same reason as the case of pi $\Pi$ (\S\ref{DependentFunctionSpace}); we generalise it in \S\ref{Sum}. 

\subsection{A game-semantic category with families}
\label{GameSemanticCwF}
We are now ready to present our game-semantic CwF. 
Let us first recall the general definition of CwFs introduced by Dybjer \cite{dybjer1996internal}:

\begin{definition}[CwFs \cite{dybjer1996internal,hofmann1997syntax}]
\label{DefCwFs}
A \emph{\bfseries category with families (CwF)} is a tuple $\mathcal{C} = (\mathcal{C}, \mathrm{Ty}, \mathrm{Tm}, \_\{\_\}, T, \_.\_, \mathrm{p}, \mathrm{v}, \langle\_,\_\rangle_\_)$,
where
\begin{itemize}

\item $\mathcal{C}$ is a category with a terminal object $T \in \mathcal{C}$;

\item $\mathrm{Ty}$ assigns, to each object $\Gamma \in \mathcal{C}$, a set $\mathrm{Ty}(\Gamma)$ of \emph{\bfseries types} in the \emph{\bfseries context} $\Gamma$;

\item $\mathrm{Tm}$ assigns, to each pair $(\Gamma, A)$ of an object $\Gamma \in \mathcal{C}$ and a type $A \in \mathrm{Ty}(\Gamma)$, a set $\mathrm{Tm}(\Gamma, A)$ of \emph{\bfseries terms} of type $A$ in the context $\Gamma$;

\item To each $\phi : \Delta \to \Gamma$ in $\mathcal{C}$, $\_\{\_\}$ assigns a map $\_\{\phi\} : \mathrm{Ty}(\Gamma) \to \mathrm{Ty}(\Delta)$, called the \emph{\bfseries substitution on types}, and a family $(\_\{\phi\}_A)_{A \in \mathrm{Ty}(\Gamma)}$ of maps $\_\{\phi\}_A : \mathrm{Tm}(\Gamma, A) \to \mathrm{Tm}(\Delta, A\{\phi\})$, called the \emph{\bfseries substitutions on terms};


\item $\_ . \_$ assigns, to each pair $(\Gamma, A)$ of a context $\Gamma \in \mathcal{C}$ and a type $A \in \mathrm{Ty}(\Gamma)$, a context $\Gamma . A \in \mathcal{C}$, called the \emph{\bfseries comprehension} of $A$;

\item $\mathrm{p}$ (resp. $\mathrm{v}$) associates each pair $(\Gamma, A)$ of a context $\Gamma \in \mathcal{C}$ and a type $A \in \mathrm{Ty}(\Gamma)$ with a morphism $\mathrm{p}_A : \Gamma . A \to \Gamma$ in $\mathcal{C}$ (resp. a term $\mathrm{v}_A \in \mathrm{Tm}(\Gamma . A, A\{\mathrm{p}_A\})$), called the \emph{\bfseries first projection} on $A$ (resp. the \emph{\bfseries second projection} on $A$);

\item $\langle \_, \_ \rangle_\_$ assigns, to each triple $(\phi, A, \Check{\alpha})$ of a morphism $\phi : \Delta \to \Gamma$ in $\mathcal{C}$, a type $A \in \mathrm{Ty}(\Gamma)$ and a term $ \Check{\alpha} \in \mathrm{Tm}(\Delta, A\{\phi\})$, a morphism $\langle \phi,  \Check{\alpha} \rangle_A : \Delta \to \Gamma . A$
in $\mathcal{C}$, called the \emph{\bfseries extension} of $\phi$ by $ \Check{\alpha}$,

\end{itemize}
that satisfies, for any $\Theta \in \mathcal{C}$, $\varphi : \Theta \to \Delta$ and $\alpha \in \mathrm{Tm}(\Gamma, A)$, the equations
\begin{itemize}

\item \textsc{(Ty-Id)} $A \{ \mathrm{id}_\Gamma \} = A$;

\item \textsc{(Ty-Comp)} $A \{ \phi \circ \varphi \} = A \{ \phi \} \{ \varphi \}$;

\item \textsc{(Tm-Id)} $\alpha \{ \mathrm{id}_\Gamma \}_A = \alpha$;

\item \textsc{(Tm-Comp)} $\alpha \{ \phi \circ \varphi \}_A = \alpha \{ \phi \}_A \{ \varphi \}_{A\{\phi\}}$;

\item \textsc{(Cons-L)} $\mathrm{p}_A \circ \langle \phi,  \Check{\alpha} \rangle_A = \phi$;

\item \textsc{(Cons-R)} $\mathrm{v}_A \{ \langle \phi,  \Check{\alpha} \rangle_A \} =  \Check{\alpha}$;

\item \textsc{(Cons-Nat)} $\langle \phi, \Check{\alpha} \rangle_A \circ \varphi = \langle \phi \circ \varphi, \Check{\alpha} \{ \varphi \}_{A\{\phi\}} \rangle_A$;

\item \textsc{(Cons-Id)} $\langle \mathrm{p}_A, \mathrm{v}_A \rangle_A = \mathrm{id}_{\Gamma . A}$.

\end{itemize}
\end{definition}

Roughly, judgements of MLTT are interpreted in a CwF $\mathcal{C}$ by
\begin{small}
\begin{mathpar}
\mathsf{\vdash \Gamma \ ctx} \mapsto \llbracket \mathsf{\Gamma} \rrbracket \in \mathcal{C} \and
\mathsf{\Gamma \vdash A \ type} \mapsto \llbracket \mathsf{A} \rrbracket \in \mathrm{Ty}(\llbracket \mathsf{\Gamma} \rrbracket) \and
\mathsf{\Gamma \vdash a : A} \mapsto \llbracket \mathsf{a} \rrbracket \in \mathrm{Tm}(\llbracket \mathsf{\Gamma} \rrbracket, \llbracket \mathsf{A} \rrbracket) \and
\mathsf{\vdash \Gamma = \Delta \ ctx} \Rightarrow \llbracket \mathsf{\Gamma} \rrbracket = \llbracket \mathsf{\Delta} \rrbracket \and
\mathsf{\Gamma \vdash A = B \ type} \Rightarrow \llbracket \mathsf{A} \rrbracket = \llbracket \mathsf{B} \rrbracket \and
\mathsf{\Gamma \vdash a = a' : A} \Rightarrow \llbracket \mathsf{a} \rrbracket = \llbracket \mathsf{a'} \rrbracket,
\end{mathpar}
\end{small}where $\llbracket \_ \rrbracket$ denotes the \emph{semantic map} or \emph{interpretation}.
See \cite{hofmann1997syntax} for the details.

Let us now turn to introducing our game-semantic CwF:
\begin{definition}[A game-semantic CwF]
\label{DefCwFWPG}
We define a CwF $\mathbb{WPG}$ as follows:
\begin{itemize}

\item The category $\mathbb{WPG}$ is given in Definition~\ref{DefGameSemanticCategories}, and $T \in \mathbb{WPG}$ in Example~\ref{ExPGames};


\item $\mathrm{Ty}(\Gamma) \colonequals \mathscr{D}^{\mathrm{w}}(\Gamma)$ ($\Gamma \in \mathbb{WPG}$) and $\mathrm{Tm}(\Gamma, A) \colonequals \mathbb{WPG}(\Pi(\Gamma, A))$ ($A \in \mathscr{D}^{\mathrm{w}}(\Gamma)$); 

\item Given $\phi : \Delta \rightarrow \Gamma$ in $\mathbb{WPG}$, define $\_\{\phi\} : \mathrm{Ty}(\Gamma) \to \mathrm{Ty}(\Delta)$ by $|A\{ \phi \}| \colonequals |A|$ and $A \{ \phi \}(\delta_0^\dagger) \colonequals A(\phi^\dagger \bullet \delta_0)$ for all $A \in \mathrm{Ty}(\Gamma)$ and $\delta_0^\dagger \in \mathbb{WPG}(\oc \Delta)$, and define $\_\{\phi\}_A : \mathrm{Tm}(\Gamma, A) \to \mathrm{Tm}(\Delta, A\{\phi\})$ by $\alpha \{ \phi \}_A \colonequals \alpha \bullet \phi$ for all $\alpha \in \mathrm{Tm}(\Gamma, A)$; 


\item $\Gamma.A \colonequals \Sigma (\Gamma, A)$, $\mathrm{p}_A \colonequals \mathrm{der}_{|\Gamma|} : \Sigma (\Gamma, A) \rightarrow \Gamma$, $\mathrm{v}_A \colonequals \mathrm{der}_{|A|} : \Pi (\Sigma(\Gamma, A), A\{\mathrm{p}_A\})$ and $\langle \phi, \Check{\alpha} \rangle_A \colonequals  \langle \phi, \Check{\alpha} \rangle : \Delta \rightarrow \Sigma (\Gamma, A)$ ($\Check{\alpha} \in \mathrm{Tm}(\Delta, A\{ \phi \})$).


\end{itemize}

Given $\Gamma \in \mathbb{WPG}$ and $A \in \mathscr{D}^{\mathrm{w}}(\Gamma)$, we write $\mathbb{WPG}(\Gamma, A)$ for $\mathrm{Tm}(\Gamma, A)$.
We often omit subscripts on components of $\mathbb{WPG}$ when they are evident. 
\end{definition}

\begin{theorem}[A well-defined game-semantic CwF]
\label{ThmWellDefinedWPG}
The category $\mathbb{WPG}$ together with the structures defined in Definition~\ref{DefCwFWPG} gives rise to a CwF.
\end{theorem}
\begin{proof}
We focus on substitution of terms, second projections and extensions since the other structures of $\mathbb{WPG}$ are straightforward to check (e.g., the equational axioms on morphisms and terms in $\mathbb{WPG}$ simply follow from Theorem~\ref{ThmCCCsOfPredicateGames}). 
Let $\Gamma, \Delta \in \mathbb{WPG}$, $A \in \mathscr{D}^{\mathrm{w}}(\Gamma)$, $\phi \in \mathbb{WPG}(\Delta, \Gamma)$, $\alpha \in \mathbb{WPG}(\Gamma, A)$ and $\Check{\alpha} \in \mathbb{WPG}(\Delta, A\{ \phi \})$.  

On $\alpha \{ \phi \} = \alpha \bullet \phi \in \mathbb{WPG}(|\Delta|, |A|)$,\footnote{The following argument is based on the proof of Lemma~\ref{LemWellDefinedConstructionsOnStrategiesBetweenPredicateGames} on composition.} we see by induction on the lengths of positions that $(\overline{\alpha \{ \phi \}})_{|\Pi(\Delta, A\{ \phi \})|} \preccurlyeq \Pi(\Delta, A\{ \phi \})(\alpha \{ \phi \})$ follows from $\phi \in \mathbb{WPG}(\Delta, \Gamma)$ and $\alpha \in \mathbb{WPG}(\Gamma, A)$.
The same induction also shows that $\alpha\{\phi\} \cap \Pi(\Delta, A\{ \phi \})(\alpha \{ \phi \}) : \Pi(\Delta, A\{ \phi \})(\alpha \{ \phi \})$ is winning and w.b., proving $\alpha \{ \phi \} \in \mathbb{WPG}(\Delta, A\{ \phi \})$.

The same induction proves $\mathrm{v}_A = \mathrm{der}_{|A|} \in \mathbb{WPG}(\Sigma(\Gamma, A), A\{\mathrm{p}_A\})$, where the point is that the pi $\Pi(\Gamma, A)$ and the sigma $\Sigma(\Gamma, A)$ are defined by the same family $A^\star$.

Finally, on $\langle \phi, \Check{\alpha} \rangle \in \mathbb{WPG}(|\Delta|, |\Gamma| \mathbin{\&} |A|)$, the same induction shows that $\langle \phi, \Check{\alpha} \rangle \in \mathbb{WPG}(\Delta, \Sigma(\Gamma, A))$ follows from $\phi \in \mathbb{WPG}(\Delta, \Gamma)$ and $\Check{\alpha} \in \mathbb{WPG}(\Delta, A\{ \phi \})$, where the \emph{fixed} $\phi$ in $\langle \phi, \Check{\alpha} \rangle$ plays a crucial role even if a play in $\Pi(\Delta, \Sigma(\Gamma, A))$ starts in $A$.
\end{proof}

It is notable that terms in the CwF $\mathbb{WPG}$ are simply a certain class of strategies, and constructions on terms are those on strategies. 
This point is quite desirable since it means that the results/techniques on existing game semantics are also available for our game semantics of MLTT.
It also implies that our game semantics inherits the advantages of existing game semantics (\S\ref{IntroGameSemantics}) such as intensionality (\S\ref{Intensionality}).

\subsection{Game-semantic type formers}
\label{GameSemanticTypeFormers}
Nevertheless, CwFs only model the fragment of MLTT common to all types. 
Hence, in this section, we equip the CwF $\mathbb{WPG}$ with \emph{semantic type formers} \cite{hofmann1997syntax} that model One-, Zero-, N-, Pi-, Sigma- and Id-types.  

\subsubsection{Game semantics of Pi-types}
\label{Prod}
We begin with \emph{Pi-types}.
Recall first their semantic type former in an arbitrary CwF:
\begin{definition}[CwFs with Pi-types \cite{hofmann1997syntax}]
A CwF $\mathcal{C}$ \emph{\bfseries supports Pi-types} if
\begin{itemize}

\item \textsc{($\Pi$-Form)} Given $\Gamma \in \mathcal{C}$, $A \in \mathrm{Ty}(\Gamma)$ and $B \in \mathrm{Ty}(\Gamma . A)$, there is a type $\Pi (A, B) \in \mathrm{Ty}(\Gamma)$, where we also write $A \Rightarrow B$ for $\Pi (A, B)$ if $B\{ \langle \mathrm{id}_\Gamma, \alpha \rangle \} = B\{ \langle \mathrm{id}_\Gamma, \alpha' \rangle \} \in \mathrm{Ty}(\Gamma)$ for all $\alpha, \alpha' \in \mathrm{Tm}(\Gamma, A)$;

\item \textsc{($\Pi$-Intro)} Given $\beta \in \mathrm{Tm}(\Gamma . A, B)$, there is a term $\lambda_{A, B} (\beta) \in \mathrm{Tm}(\Gamma, \Pi (A, B))$;

\item \textsc{($\Pi$-Elim)} Given $\kappa \in \mathrm{Tm}(\Gamma, \Pi (A, B))$ and $\alpha \in \mathrm{Tm}(\Gamma, A)$, there is a term $\mathrm{App}_{A, B} (\kappa, \alpha) \in \mathrm{Tm}(\Gamma, B\{ \overline{\alpha} \})$, where $\overline{\alpha} \colonequals \langle \mathrm{id}_\Gamma, \alpha \rangle_A : \Gamma \to \Gamma . A$;

\item \textsc{($\Pi$-Comp)} $\mathrm{App}_{A, B} (\lambda_{A, B} (\beta) , \alpha) = \beta \{ \overline{\alpha} \}$;

\item \textsc{($\Pi$-Subst)} Given $\Delta \in \mathcal{C}$ and $\phi : \Delta \to \Gamma$ in $\mathcal{C}$, $\Pi (A, B) \{ \phi \} = \Pi (A\{\phi\}, B\{\phi_A^+\})$, where $\phi_A^+ \colonequals \langle \phi \circ \mathrm{p}_{A\{ \phi \}}, \mathrm{v}_{A\{ \phi \}} \rangle_A : \Delta . A\{ \phi \} \to \Gamma . A$;

\item \textsc{($\lambda$-Subst)} $\lambda_{A, B} (\beta) \{ \phi \} = \lambda_{A\{\phi\}, B\{\phi_A^+\}} (\beta \{ \phi_A^+ \}) \in \mathrm{Tm} (\Delta, \Pi (A\{\phi\}, B\{ \phi_A^+\}))$;

\item \textsc{(App-Subst)} $\mathrm{App}_{A, B} (\kappa, \alpha) \{ \phi \} = \mathrm{App}_{A\{\phi\}, B\{\phi_A^+\}} (\kappa \{ \phi \}, \alpha \{ \phi \}) \in \mathrm{Tm} (\Delta, B\{ \overline{\alpha} \} \{ \phi \})$. 
\end{itemize}

Furthermore, $\mathcal{C}$ \emph{\bfseries strictly supports Pi-types} if it additionally satisfies
\begin{itemize}
\item \textsc{($\lambda$-Uniq)} $\lambda_{A, B} \circ \mathrm{App}_{A\{\mathrm{p}_A\}, B\{(\mathrm{p}_A)_{A\{\mathrm{p}_A\}}^+\}}(\kappa\{ \mathrm{p}_A \}, \mathrm{v}_A) = \kappa$.

\end{itemize}
\end{definition}

Pi-types (with $\eta$-rule) are modelled in a CwF that (strictly) supports Pi-types \cite{hofmann1997syntax}.
Let us now proceed to show that the CwF $\mathbb{WPG}$ strictly supports Pi-types.

\begin{lemma}[Currying]
\label{LemPi-SigmaCorrespondence}
Given $\Gamma \in \mathbb{WPG}$, $A \in \mathscr{D}^{\mathrm{w}}(\Gamma)$ and $B \in \mathscr{D}^{\mathrm{w}}(\Sigma(\Gamma, A))$, there is a bijection $\lambda_{A, B} : \mathbb{WPG}(\Sigma(\Gamma, A), B) \stackrel{\sim}{\rightarrow} \mathbb{WPG}(\Gamma, \Pi(A, B))$, where $\Pi(A, B) \in \mathscr{D}^{\mathrm{w}}(\Gamma)$ is defined by $|\Pi(A, B)| \colonequals |A| \Rightarrow |B|$ and $\Pi(A, B)(\gamma_0^\dagger) \colonequals \Pi(A(\gamma_0^\dagger), B_{\gamma_0^\dagger})$ for all $\gamma_0^\dagger \in \mathbb{WPG}(\oc \Gamma)$, and $B_{\gamma_0^\dagger} \in \mathscr{D}^{\mathrm{w}}(A(\gamma_0^\dagger))$ by $|B_{\gamma_0^\dagger}| \colonequals |B|$ and $B_{\gamma_0^\dagger}(\alpha_0^\dagger) \colonequals B(\langle \gamma_0, \alpha_0 \rangle^\dagger)$ for all $\alpha_0^\dagger \in \mathbb{WPG}(\oc A(\gamma_0^\dagger))$.
\end{lemma}
\begin{proof}
Let $\phi \in \mathbb{WPG}(\Sigma(\Gamma, A), B)$; we first construct $\phi' \in \mathbb{WPG}(\Gamma, \Pi(A, B))$ from $\phi$ as follows.
By the isomorphism $|\oc \Sigma(\Gamma, A)| = \oc (|\Gamma| \mathbin{\&} |A|) \cong \oc |\Gamma| \otimes \oc |A|$, we adjust `tags' or \emph{curry} $\phi$ with respect to the adjunction between tensor $\otimes$ and linear implication $\multimap$ \cite{mccusker1998games}, obtaining another strategy $\phi' : |\Gamma| \Rightarrow (|A| \Rightarrow |B|)$.

Next, let us verify $\overline{\phi'}_{|\Pi(\Gamma, \Pi(A, B))|} \preccurlyeq \Pi(\Gamma, \Pi(A, B))(\phi')$ by induction on the length of positions, where again the only nontrivial case is the inductive step on even-length positions.
This inductive step goes through by $\phi \in \mathbb{WPG}(\Sigma(\Gamma, A), B)$ thanks to the correspondence between the constraints on $\Pi(\Sigma(\Gamma, A), B)(\phi)$ and $\Pi(\Gamma, \Pi(A, B))(\phi')$.

The same argument also shows that $\phi' \cap \Pi(\Gamma, \Pi(A, B))(\phi') : \Pi(\Gamma, \Pi(A, B))(\phi')$ is winning and w.b., and hence we have proven $\lambda_{A, B}(\phi) \colonequals \phi' \in \mathbb{WPG}(\Gamma, \Pi(A, B))$.

Finally, the inverse of this construction $\lambda_{A, B}$ is given in the same vein.
\end{proof}

\begin{example}
\label{ExPiSigma}
Define $N_b^+ \in \mathscr{D}^{\mathrm{w}}(\Sigma(N, N_b))$ by $|N_b^+| \colonequals N$ and $N_b^+(\langle \underline{k}, \underline{n} \rangle^\dagger) \colonequals N_b(\underline{k+n}^\dagger)$ for all $k, n \in \mathbb{N}$ such that $k \geqslant n$.
Then, Lemma~\ref{LemPi-SigmaCorrespondence} allows us to obtain $\lambda_{N_b, N_b^+}(\mathrm{der}_N) \in \mathbb{WPG}(N, \Pi(N_b, N_b^+))$ from $\mathrm{der}_N = \mathrm{v}_{N_b^+} \in \mathbb{WPG}(\Sigma(N, N_b), N_b^+)$.
\end{example}

\begin{theorem}[Game semantics of Pi-types]
\label{ThmGameSemanticsOfPiType}
$\mathbb{WPG}$ strictly supports Pi-types.
\end{theorem}
\begin{proof}
Let $\Gamma \in \mathbb{WPG}$, $A \in \mathscr{D}^{\mathrm{w}}(\Gamma)$, $B \in \mathscr{D}^{\mathrm{w}}(\Sigma(\Gamma, A))$ and $\beta \in \mathbb{WPG}(\Sigma(\Gamma, A), B)$.
\begin{itemize}

\item \textsc{($\Pi$-Form)} $\Pi(A, B) \in \mathscr{D}^{\mathrm{w}}(\Gamma)$ is defined in Lemma~\ref{LemPi-SigmaCorrespondence}.
 
\item \textsc{($\Pi$-Intro)} By Lemma~\ref{LemPi-SigmaCorrespondence}, we obtain $\lambda_{A, B} (\beta) \in \mathbb{WPG}(\Gamma, \Pi(A, B))$. 
We often omit the subscripts $(\_)_{A, B}$ on $\lambda_{A, B}$ and the inverse $\lambda_{A, B}^{-1}$.

\item \textsc{($\Pi$-Elim)} $\mathrm{App}_{A, B} (\kappa, \alpha) \colonequals \lambda_{A, B}^{-1}(\kappa) \{\overline{\alpha} \}$ for all $\kappa \in \mathbb{WPG}(\Gamma, \Pi(A, B))$ and $\alpha \in \mathbb{WPG}(\Gamma, A)$. 
We indeed have $\mathrm{App}_{A, B} (\kappa, \alpha) = \lambda_{A, B}^{-1}(\kappa) \{ \overline{\alpha} \} \in \mathbb{WPG}(\Gamma, B\{ \overline{\alpha} \})$ by Theorem~\ref{ThmWellDefinedWPG}. 
We often omit the subscripts $(\_)_{A, B}$ on $\mathrm{App}_{A, B}$.

\item \textsc{($\Pi$-Comp)} $\mathrm{App}_{A, B} (\lambda_{A, B} (\beta) , \alpha) = \lambda_{A, B}^{-1}(\lambda_{A, B}(\beta)) \{ \overline{\alpha} \} = \beta \{ \overline{\alpha} \}$.

\item \textsc{($\Pi$-Subst)} Let $\Delta \in \mathbb{WPG}$ and $\phi \in \mathbb{WPG}(\Delta, \Gamma)$.
For the first components, we have $| \Pi (A, B) \{ \phi \} | = |\Pi(A, B)| = |A| \Rightarrow |B| = |A\{ \phi \} | \Rightarrow |B\{ \phi_A^+ \})| = | \Pi(A\{ \phi \}, B\{ \phi_A^+ \}) |$.
For the second components, we have
\begin{align*}
\textstyle \Pi (A, B) \{ \phi \} &= (\Pi (A(\phi^\dagger \circ \delta_0^\dagger), B_{\phi^\dagger \circ \delta_0^\dagger}))_{\delta_0^\dagger \in \mathbb{WPG}(\oc \Delta)} \\
&= (\Pi (A\{ \phi \}(\delta_0^\dagger), B\{ \phi_A^+ \}_{\delta_0^\dagger}))_{\delta_0^\dagger \in \mathbb{WPG}(\oc \Delta)} \\
&= \Pi(A\{ \phi \}, B\{ \phi_A^+ \}),
\end{align*}
where the second equation holds since for all $\Hat{\alpha}_0^\dagger \in \mathbb{WPG}(\oc A(\phi^\dagger \circ \delta_0^\dagger))$ we have
\begin{align*}
B\{\phi_A^+\}_{\delta_0^\dagger} (\Hat{\alpha}_0^\dagger) &= B\{\phi_A^+\} (\langle \delta_0, \Hat{\alpha}_0 \rangle^\dagger) \\ 
&= B (\langle \phi \bullet \mathrm{p}_{A \{ \phi \}}, \mathrm{v}_{A\{ \phi \}} \rangle^\dagger \bullet \langle \delta_0, \Hat{\alpha}_0 \rangle) \\
&= B(\langle \phi \bullet \delta_0, \Hat{\alpha}_0 \rangle^\dagger) \\
&= B_{\phi^\dagger \circ \delta_0^\dagger} (\Hat{\alpha}_0^\dagger).
\end{align*}
We have shown $\| \Pi (A, B) \{ \phi \} \| = \| \Pi(A\{ \phi \}, B\{ \phi_A^+ \}) \|$ as well.

\item \textsc{($\lambda$-Subst)} By the definition of $\lambda$, we have $\lambda_{A, B} (\beta) \{ \phi \} = \lambda_{A\{\phi\}, B\{\phi_A^+\}} (\beta \{ \phi_A^+ \})$.

\item \textsc{(App-Subst)} We have
\begin{align*}
\mathrm{App}_{A, B} (\kappa, \alpha) \{ \phi \} &= \lambda_{A, B}^{-1} (\kappa) \{ \langle \mathrm{der}_{|\Gamma|}, \alpha \rangle \bullet \phi \} \\
&= \lambda_{A, B}^{-1} (\kappa) \{ \langle \phi, \alpha \{ \phi \} \rangle \} \\
&= \lambda_{A, B}^{-1} (\kappa) \{ \langle \phi \bullet \mathrm{p}_{A \{ \phi \}}, \mathrm{v}_{A\{\phi\}} \rangle \bullet \langle \mathrm{der}_{|\Delta|}, \alpha \{ \phi \} \rangle \} \\
&= \lambda_{A, B}^{-1} (\kappa) \{ \phi_A^+ \} \{ \overline{\alpha \{ \phi \}} \} \\
&= \lambda_{A\{ \phi \}, B\{ \phi_A^+ \}}^{-1} (\kappa \{ \phi \}) \{ \overline{\alpha \{ \phi \}} \} \quad \text{(by $\lambda$-Subst)} \\
&= \mathrm{App}_{A\{ \phi \}, B\{ \phi_A^+ \}} (\kappa \{ \phi \}, \alpha \{ \phi \}).
\end{align*}

\item \textsc{($\lambda$-Uniq)} We have
\begin{small}
\begin{align*} 
\lambda_{A, B}(\mathrm{App}_{A\{\mathrm{p}_A\}, B\{ \mathrm{p}_A^+ \}}(\kappa\{ \mathrm{p}_A \}, \mathrm{v}_A)) &= \lambda_{A, B}(\lambda^{-1}_{A\{\mathrm{p}_A\}, B\{ \mathrm{p}_A^+ \}}(\kappa\{ \mathrm{p}_A \}) \{ \overline{\mathrm{v}_A} \}) \\
&= \lambda_{A, B}(\lambda^{-1}_{A, B}(\kappa) \{ \mathrm{p}_A^+ \} \{ \overline{\mathrm{v}_A} \}) \quad \text{(by $\lambda$-Subst)} \\
&= \lambda_{A, B}(\lambda^{-1}_{A, B}(\kappa) \{ \langle \mathrm{p}_A \bullet \mathrm{p}_{A\{ \mathrm{p}_A \}}, \mathrm{v}_{A\{ \mathrm{p}_A \}} \rangle \bullet \langle \mathrm{der}_{|\Sigma(\Gamma, A)|}, \mathrm{v}_A \rangle \}) \\
&= \lambda_{A, B}(\lambda^{-1}_{A, B}(\kappa) \bullet \langle \mathrm{p}_A, \mathrm{v}_A \rangle) \\
&= \lambda_{A, B}(\lambda^{-1}_{A, B}(\kappa) \bullet \mathrm{der}_{|\Sigma(\Gamma, A)|}) \\
&= \lambda_{A, B}(\lambda^{-1}_{A, B}(\kappa)) \\
&= k,
\end{align*}
\end{small}
\end{itemize}
which completes the proof.
\end{proof}

\subsubsection{Game semantics of Sigma-types}
\label{Sum}
Next, we consider \emph{Sigma-types}. 
Again, let us first recall their general semantic type former in an arbitrary CwF:
\begin{definition}[CwFs with Sigma-types \cite{hofmann1997syntax}]
\label{DefCwFsWithSigmaType}
A CwF $\mathcal{C}$ \mbox{\emph{\bfseries supports Sigma-types} if}
\begin{itemize}

\item \textsc{($\Sigma$-Form)} Given $\Gamma \in \mathcal{C}$, $A \in \mathrm{Ty}(\Gamma)$ and $B \in \mathrm{Ty}(\Gamma . A)$, there is a type $\Sigma (A, B) \in \mathrm{Ty}(\Gamma)$, where we also write $A \times B$ for $\Sigma (A, B)$ if $B\{ \langle \mathrm{id}_\Gamma, \alpha \rangle \} = B\{ \langle \mathrm{id}_\Gamma, \alpha' \rangle \} \in \mathrm{Ty}(\Gamma)$ for all $\alpha, \alpha' \in \mathrm{Tm}(\Gamma, A)$;

\item \textsc{($\Sigma$-Intro)} There is a morphism $\mathrm{Pair}_{A, B} : \Gamma . A . B \to \Gamma . \Sigma (A, B)$ in $\mathcal{C}$;

\item \textsc{($\Sigma$-Elim)} Given $P \in \mathrm{Ty}(\Gamma . \Sigma (A, B))$ and $\rho \in \mathrm{Tm}(\Gamma . A . B, P \{ \mathrm{Pair}_{A, B} \})$, there is a term $\mathcal{R}^{\Sigma}_{A, B, P}(\rho) \in \mathrm{Tm}(\Gamma . \Sigma (A, B), P)$;

\item \textsc{($\Sigma$-Comp)} $\mathcal{R}^{\Sigma}_{A, B, P}(\rho) \{ \mathrm{Pair}_{A, B}\} = \rho$;

\item \textsc{($\Sigma$-Subst)} Given $\Delta \in \mathcal{C}$ and $\phi : \Delta \to \Gamma$ in $\mathcal{C}$, $\Sigma (A, B) \{ \phi \} = \Sigma (A\{\phi\}, B\{\phi_A^+\})$, where $\phi_A^+ \colonequals \langle \phi \circ \mathrm{p}_{A\{ \phi \}}, \mathrm{v}_{A\{\phi\}} \rangle_A : \Delta . A\{\phi\} \to \Gamma . A$;

\item \textsc{(Pair-Subst)} $\mathrm{p}_{\Sigma (A, B)} \circ \mathrm{Pair}_{A, B} = \mathrm{p}_A \circ \mathrm{p}_B$ and $\phi_{\Sigma(A, B)}^+ \circ \mathrm{Pair}_{A\{\phi\}, B\{\phi_A^+\}} = \mathrm{Pair}_{A, B} \circ \phi^{++}_{A, B}$, where $\phi^{++}_{A, B} \colonequals (\phi_A^+)^+_B : \Delta . A\{\phi\} . B\{\phi_A^+\} \to \Gamma . A . B$;

\item \textsc{($\mathcal{R}^{\Sigma}$-Subst)} $\mathcal{R}^{\Sigma}_{A, B, P}(p) \{\phi^+_{\Sigma(A, B)}\} = \mathcal{R}^{\Sigma}_{A\{f\}, B\{\phi_A^+\}, P\{\phi^+_{\Sigma(A, B)}\}} (p \{ \phi_{A, B}^{++} \})$.

\end{itemize}

In addition, $\mathcal{C}$ \emph{\bfseries strictly supports Sigma-types} if it also satisfies
\begin{itemize}
\item \textsc{($\mathcal{R}^{\Sigma}$-Uniq)} $\Check{\rho} = \mathcal{R}^{\Sigma}_{A, B, P}(\rho)$ if $\Check{\rho} \in \mathrm{Tm}(\Gamma . \Sigma(A, B), P)$ and $\Check{\rho} \{ \mathrm{Pair}_{A, B} \} = \rho$.
\end{itemize}
\end{definition}

Sigma-types (with $\eta$-rule) are modelled in a CwF that (strictly) supports Sigma-types \cite{hofmann1997syntax}.
Now, we show that our CwF $\mathbb{WPG}$ strictly supports Sigma-types:
\begin{theorem}[Game semantics of Sigma-types]
\label{ThmGameSemanticsOfSigmaType}
$\mathbb{WPG}$ strictly supports Sigma-types.
\end{theorem}
\begin{proof}
Let $\Gamma \in \mathbb{WPG}$, $A \in \mathscr{D}^{\mathrm{w}}(\Gamma)$, $B \in \mathscr{D}^{\mathrm{w}}(\Sigma(\Gamma, A))$ and $P \in \mathscr{D}^{\mathrm{w}}(\Sigma(\Gamma, \Sigma(A,B)))$.
\begin{itemize}

\item \textsc{($\Sigma$-Form)} Similarly to pi $\Pi$, $\Sigma (A, B) \colonequals (|A| \mathbin{\&} |B|, (\Sigma(A(\gamma_0^\dagger), B_{\gamma_0^\dagger}))_{\gamma_0^\dagger \in \mathbb{WPG}(\oc \Gamma)})$.

\item \textsc{($\Sigma$-Intro)} By the evident bijection $\Sigma(\Sigma(\Gamma,A),B) \cong \Sigma(\Gamma, \Sigma(A,B))$, define $\mathrm{Pair}_{A, B} : \Sigma(\Sigma(\Gamma,A),B) \rightarrow \Sigma(\Gamma, \Sigma(A,B))$ in $\mathbb{WPG}$ to be the dereliction up to `tags,' i.e., $\mathrm{Pair}_{A, B} \colonequals \langle \mathrm{p}_A \bullet \mathrm{p}_B, \langle \mathrm{v}_A \{ \mathrm{p}_B \}, \mathrm{v}_B \rangle \rangle$. 
Note that the inverse $\mathrm{Pair}^{-1}_{A, B}$ is $\langle \langle \mathrm{p}_{\Sigma(A, B)}, \varpi^{A, B}_1 \rangle, \varpi^{A, B}_2 \rangle$, where $\varpi^{A, B}_1 : \Sigma(\Gamma, \Sigma(A, B)) \rightarrow A \{ \mathrm{p}_{\Sigma(A, B)} \}$ and $\varpi^{A, B}_2 : \Sigma(\Gamma, \Sigma(A, B)) \rightarrow B \{ \langle \mathrm{p}_{\Sigma(A, B)}, \varpi^{A, B}_1 \rangle \}$ are the derelictions up to `tags.'

\item \textsc{($\Sigma$-Elim)} Given $\rho \in \mathbb{WPG}(\Sigma(\Sigma(\Gamma, A), B), P\{\mathrm{Pair}_{A,B}\})$, define $\mathcal{R}^{\Sigma}_{A, B, P}(\rho) \colonequals \rho \{ \mathrm{Pair}_{A, B}^{-1} \} \in \mathbb{WPG}(\Sigma(\Gamma, \Sigma (A, B)), P\{\mathrm{Pair}_{A,B}\}\{\mathrm{Pair}_{A,B}^{-1}\}) = \mathbb{WPG}(\Sigma(\Gamma, \Sigma (A, B)), P)$.

\item \textsc{($\Sigma$-Comp)} We have
\begin{align*}
\mathcal{R}^{\Sigma}_{A, B, P}(\rho) \{ \mathrm{Pair}_{A, B}\} &= \rho \{ \mathrm{Pair}_{A, B}^{-1} \} \{ \mathrm{Pair}_{A, B} \} \\
&= \rho \{ \mathrm{Pair}_{A, B}^{-1} \bullet \mathrm{Pair}_{A, B} \} \\
&= \rho \{ \mathrm{id}_{\Sigma(\Sigma(\Gamma, A), B)} \} \\
&= \rho.
\end{align*}

\item \textsc{($\Sigma$-Subst)} Similar to the case of pi $\Pi$ (Theorem~\ref{ThmGameSemanticsOfPiType}).

\item \textsc{(Pair-Subst)} $\mathrm{p}_{\Sigma (A, B)} \bullet \mathrm{Pair}_{A, B} = \mathrm{p}_{\Sigma (A, B)} \bullet \langle \mathrm{p}_A \bullet \mathrm{p}_B, \langle \mathrm{v}_A \{ \mathrm{p}_B \}, \mathrm{v}_B \rangle \rangle \notag = \mathrm{p}_A \bullet \mathrm{p}_B$ and
\begin{scriptsize}
\begin{align*}
\phi^+_{\Sigma(A, B)} \bullet \mathrm{Pair}_{A\{\phi\}, B\{\phi_A^+\}} &= \langle \phi \bullet \mathrm{p}_{\Sigma(A,B)\{\phi\}}, \mathrm{v}_{\Sigma(A,B)\{\phi\}} \rangle \bullet \mathrm{Pair}_{A\{\phi\}, B\{\phi_A^+\}} \\
&= \langle \phi \bullet \mathrm{p}_{\Sigma(A\{\phi\},B\{\phi_A^+\})} \bullet \mathrm{Pair}_{A\{\phi\}, B\{\phi_A^+\}}, \mathrm{v}_{\Sigma(A,B)\{\phi\}} \{ \mathrm{Pair}_{A\{ \phi \}, B\{ \phi_A^+ \}} \}\rangle \\
&= \langle \phi \bullet \mathrm{p}_{A\{ \phi \}} \bullet \mathrm{p}_{B\{ \phi_A^+ \}}, \mathrm{v}_{\Sigma(A\{ \phi \},B\{ \phi_A^+ \})} \{ \mathrm{Pair}_{A\{ \phi \}, B\{ \phi_A^+ \}} \} \rangle \quad \text{(by the above equation)} \\
&= \langle \phi \bullet \mathrm{p}_{A\{ \phi \}} \bullet \mathrm{p}_{B\{ \phi_A^+ \}}, \langle \mathrm{v}_{A\{\phi\}} \{ \mathrm{p}_{B\{ \phi_A^+ \}} \}, \mathrm{v}_{B\{ \phi_A^+ \}} \rangle \rangle \\
&= \langle \mathrm{p}_A \bullet \mathrm{p}_B, \langle \mathrm{v}_A \{ \mathrm{p}_B \}, \mathrm{v}_B \rangle \rangle \bullet \langle \langle \phi \bullet \mathrm{p}_{A\{ \phi \}} \bullet \mathrm{p}_{B\{ \phi_A^+ \}}, \mathrm{v}_{A\{ \phi \}}\{ \mathrm{p}_{B\{ \phi_A^+ \}} \} \rangle, \mathrm{v}_{B\{ \phi_A^+ \}} \rangle \\
&= \langle \mathrm{p}_A \bullet \mathrm{p}_B, \langle \mathrm{v}_A \{ \mathrm{p}_B \}, \mathrm{v}_B \rangle \rangle \bullet \langle \langle \phi \bullet \mathrm{p}_{A\{ \phi \}}, \mathrm{v}_{A\{ \phi \}} \rangle \bullet \mathrm{p}_{B\{ \phi_A^+ \}}, \mathrm{v}_{B\{ \phi_A^+ \}} \rangle \\
&= \mathrm{Pair}_{A, B} \bullet \langle \phi_A^+ \bullet \mathrm{p}_{B\{ \phi_A^+ \}}, \mathrm{v}_{B\{ \phi_A^+ \}} \rangle \\
&= \mathrm{Pair}_{A, B} \bullet \phi_{A, B}^{++}.
\end{align*}
\end{scriptsize}

\item \textsc{($\mathcal{R}^{\Sigma}$-Subst)} We have
\begin{scriptsize}
\begin{align*}
\mathcal{R}^{\Sigma}_{A, B, P}(\rho) \{ \phi^+_{\Sigma(A, B)} \} &= \rho \{ \mathrm{Pair}_{A, B}^{-1} \} \{ \langle \phi \bullet \mathrm{p}_{\Sigma(A,B)\{ \phi \}}, \mathrm{v}_{\Sigma(A,B)\{ \phi \}} \rangle \} \\
&= \rho \{ \langle \langle \mathrm{p}_{\Sigma(A, B)}, \varpi^{A, B}_1 \rangle, \varpi^{A, B}_2 \rangle \bullet \langle \phi \bullet \mathrm{p}_{\Sigma(A,B)\{ \phi \}}, \mathrm{v}_{\Sigma(A,B)\{ \phi \}} \rangle \} \\
&= \rho \{ \langle \langle \phi \bullet \mathrm{p}_{\Sigma(A, B)\{ \phi \}}, \varpi^{A\{ \phi \}, B\{ \phi_A^+ \}}_1 \rangle, \varpi^{A\{ \phi \}, B\{ \phi_A^+ \}}_2 \rangle \} \quad \text{(by the definition of $\varpi_1$ and $\varpi_2$)} \\
&= \rho \{ \langle \langle \phi \bullet \mathrm{p}_{A\{ \phi \}}, \mathrm{v}_{A\{\phi\}} \rangle \bullet \mathrm{p}_{B\{ \phi_A^+ \}}, \mathrm{v}_{B\{ \phi_A^+ \}} \rangle \} \{ \langle \langle \mathrm{p}_{\Sigma(A, B)\{ \phi \}}, \varpi^{A\{ \phi \}, B\{ \phi_A^+ \}}_1 \rangle, \varpi^{A\{ \phi \}, B\{ \phi_A^+ \}}_2 \rangle \} \\
&= \rho \{ \langle \phi_A^+ \bullet \mathrm{p}_{B\{ \phi_A^+ \}}, \mathrm{v}_{B\{ \phi_A^+ \}} \rangle \} \{ \mathrm{Pair}_{A\{ \phi \}, B\{ \phi_A^+ \}}^{-1} \} \\
&= \mathcal{R}^{\Sigma}_{A\{ \phi \}, B\{ \phi_A^+ \}, P\{ \phi_{\Sigma(A, B)}^+ \}} (\rho \{ \phi_{A, B}^{++} \}).
\end{align*}
\end{scriptsize}

\item \textsc{($\mathcal{R}^{\Sigma}$-Uniq)} Given $\Check{\rho} \in \mathbb{WPG}(\Sigma(\Gamma, \Sigma(A, B)), P)$ with $\Check{\rho} \{ \mathrm{Pair}_{A, B} \} = \rho$, we have
\begin{equation*}
\Check{\rho} = \Check{\rho} \{ \mathrm{id}_{\Sigma(\Gamma, \Sigma(A, B))} \} = \Check{\rho} \{ \mathrm{Pair}_{A, B} \} \{ \mathrm{Pair}_{A, B}^{-1} \} = \rho \{ \mathrm{Pair}_{A, B}^{-1} \} = \mathcal{R}^{\Sigma}_{A, B, P}(\rho),
\end{equation*}
\end{itemize}
which completes the proof.
\end{proof}

\subsubsection{Game semantics of N-type}
\label{GameTheoreticNaturalNumbers}
We next present our game semantics of \emph{N-type}.
Again, we first recall the general semantic type former for N-type in an arbitrary CwF:
\begin{definition}[CwFs with N-type \cite{hofmann1997syntax}]
A CwF $\mathcal{C}$ \emph{\bfseries supports N-type} if
\begin{itemize}

\item \textsc{($N$-Form)} Given $\Gamma \in \mathcal{C}$, there is a type $N^{[\Gamma]} \in \mathrm{Ty}(\Gamma)$. 
We abbreviate it as~$N$.

\item \textsc{($N$-Intro)} There are a term $\underline{0}_{\Gamma} \in \mathrm{Tm}(\Gamma, N)$ and a morphism $\mathrm{succ}_{\Gamma} : \Gamma . N \to \Gamma . N$ in $\mathcal{C}$ that satisfy for any morphisms $\phi : \Delta \to \Gamma$ and $\psi : \Delta . N \to \Gamma$ in $\mathcal{C}$
\begin{mathpar}
\underline{0}_{\Gamma} \{ \phi \} = \underline{0}_\Delta \in \mathrm{Tm}(\Delta, N) \and
\mathrm{p}_N \circ \mathrm{succ}_{\Gamma} = \mathrm{p}_N : \Gamma . N \to \Gamma \and
\mathrm{succ}_{\Gamma} \circ \langle \psi, \mathrm{v}_{N} \rangle_{N} = \langle \psi, \mathrm{v}_{N}\{\mathrm{succ}_{\Delta}\} \rangle_{N} : \Delta . N \to \Gamma . N,
\end{mathpar}
where the last equation makes sense since $\mathrm{succ}_{\Gamma} \circ \langle \psi, \mathrm{v}_{N} \rangle_{N}, \langle \psi, \mathrm{v}_{N}\{\mathrm{succ}_{\Delta}\} \rangle_{N} : \Delta . N \to \Gamma . N$ by $N$-Subst given below.
We henceforth skip the same remark.

\begin{notation}
Define $\mathrm{zero}_\Gamma \colonequals \langle \mathrm{id}_\Gamma, \underline{0}_\Gamma \rangle_N : \Gamma \to \Gamma . N$ for each $\Gamma \in \mathcal{C}$; it satisfies $\mathrm{zero}_\Gamma \circ \phi = \langle \phi, \underline{0}_\Delta \rangle_N = \langle \phi, \mathrm{v}_N \{ \mathrm{zero}_\Delta \} \rangle_N : \Delta \to \Gamma . N$ for any $\phi : \Delta \to \Gamma$ in $\mathcal{C}$.
We often omit the subscript $(\_)_\Gamma$ on $\underline{0}$, $\mathrm{zero}$ and $\mathrm{succ}$.
Define $\underline{n}_\Gamma \in \mathrm{Tm}(\Gamma, N)$ for each $n \in \mathbb{N}$ by: $\underline{0}_\Gamma$ is already given, and $\underline{n+1}_\Gamma \colonequals \mathrm{v}_N \{ \mathrm{succ}_\Gamma \circ \langle \mathrm{id}_\Gamma, \underline{n}_\Gamma \rangle_N \}$.
\end{notation}

\item \textsc{($N$-Elim)} Given a type $P \in \mathrm{Ty}(\Gamma . N)$, and terms $c_{\mathrm{z}} \in \mathrm{Tm}(\Gamma, P\{\mathrm{zero}\})$ and $c_{\mathrm{s}} \in \mathrm{Tm}(\Gamma . N . P, P\{ \mathrm{succ} \circ \mathrm{p}_P\})$, there is a term $\mathcal{R}^{N}_{P}(c_{\mathrm{z}}, c_{\mathrm{s}}) \in \mathrm{Tm}(\Gamma . N, P)$;

\item \textsc{($N$-Comp)} We have
\begin{align*}
\mathcal{R}^{N}_{P}(c_{\mathrm{z}}, c_{\mathrm{s}}) \{ \mathrm{zero} \} &= c_{\mathrm{z}} \in \mathrm{Tm}(\Gamma, P\{\mathrm{zero}\}); \\
\mathcal{R}^{N}_{P}(c_{\mathrm{z}}, c_{\mathrm{s}}) \{ \mathrm{succ} \} &= c_{\mathrm{s}} \{ \langle \mathrm{id}_{\Gamma . N}, \mathcal{R}^{N}_{P}(c_{\mathrm{z}}, c_{\mathrm{s}}) \rangle_{P} \} \in \mathrm{Tm}(\Gamma . N, P\{\mathrm{succ}\});
\end{align*}

\item \textsc{($N$-Subst)} $N^{[\Gamma]}\{\phi\} = N^{[\Delta]} \in \mathrm{Ty}(\Delta)$;

\item \textsc{($\mathcal{R}^{N}$-Subst)} $\mathcal{R}^{N}_{P}(c_{\mathrm{z}}, c_{\mathrm{s}}) \{\phi_N^+\} = \mathcal{R}^{N}_{P\{\phi_N^+\}}(c_{\mathrm{z}}\{\phi\}, c_{\mathrm{s}}\{\phi_{N, P}^{++}\}) \in \mathrm{Tm}(\Delta . N, P\{\phi_N^+\})$, where $\phi_N^+ \colonequals \langle \phi \circ \mathrm{p}_N, \mathrm{v}_{N} \rangle_{N} : \Delta . N \to \Gamma . N$ \mbox{and $\phi_{N, P}^{++} \colonequals (\phi_N^+)^+_P : \Delta . N . P\{\phi_N^+\} \to \Gamma . N . P$}.

\end{itemize}
\end{definition}

Let us now present our game semantics of N-type, which is based on the standard game semantics of PCF \cite{abramsky1999game}:
\begin{theorem}[Game semantics of N-type]
\label{ThmGameSemanticsOfNType}
$\mathbb{WPG}$ supports N-type.
\end{theorem}
\begin{proof}
Let $\Gamma, \Delta \in \mathbb{WPG}$, $\phi \in \mathbb{WPG}(\Delta, \Gamma)$ and $\psi \in \mathbb{WPG}(\Sigma(\Delta, \{ N \}_{\oc \Delta}), \Gamma)$.

\begin{itemize}
\item \textsc{($N$-Form)} $N^{[\Gamma]}$ is the constant one $\{ N \}_{\oc \Gamma} \in \mathscr{D}^{\mathrm{w}}(\Gamma)$ at $N$ (Example~\ref{ExPGames}).

\item \textsc{($N$-Intro)} 
$\underline{0}_\Gamma \in \mathbb{WPG}(\Gamma, \{ N \})$ is $\underline{0} : N$ (Examples~\ref{ExamplesOfGames} and \ref{ExPGames}) up to `tags,' and $\mathrm{succ}_\Gamma \colonequals \langle \mathrm{p}, \mathrm{sc}_\Gamma \rangle \in \mathbb{WPG}(\Sigma(\Gamma, \{ N \}), \Sigma(\Gamma, \{ N \}))$, where we define $\mathrm{sc}_\Gamma \colonequals \mathrm{Pref}(\{ \, q_{[1]}q_{[0]}n_{[0]}(n+1)_{[1]} \mid n \in \mathbb{N} \, \})^{\mathrm{Even}} \in \mathbb{WPG}(\Sigma(\Gamma, \{ N_{[0]} \}), \{ N_{[1]} \})$.

Clearly, we have $\underline{0}_\Gamma \bullet \phi = \underline{0}_\Delta$ and $\mathrm{sc}_\Gamma \bullet \langle \psi, \mathrm{v}_{\{N\}_\Delta} \rangle = \mathrm{sc}_\Delta = \mathrm{v}_{\{N\}_\Delta}\{ \mathrm{succ}_\Delta \}$, and therefore the required three equations hold.

\item \textsc{($N$-Elim)} Given $P \in \mathscr{D}^{\mathrm{w}}(\Sigma(\Gamma, \{ N \}))$, $c_{\mathrm{z}} \in \mathbb{WPG}(\Gamma, P\{ \mathrm{zero} \})$ and $c_{\mathrm{s}} \in \mathbb{WPG}(\Sigma(\Sigma(\Gamma, \{ N \}), P), P\{ \mathrm{succ} \bullet \mathrm{p}_P\})$, there are two terms 
\begin{align*}
&\widetilde{c}_{\mathrm{z}} \in \mathbb{WPG}(\Sigma(\Pi(\Sigma(\Gamma, \{N\}), P), \{ \Sigma(\Gamma, \{ N \}) \}), P\{ \mathrm{zero} \bullet \mathrm{p} \bullet \mathrm{v} \}); \\
&\widetilde{c}_{\mathrm{s}} \in \mathbb{WPG}(\Sigma(\Pi(\Sigma(\Gamma, \{N\}), P), \{ \Sigma(\Gamma, \{ N \}) \}), P\{ \mathrm{succ} \bullet \mathrm{v} \}),
\end{align*}
where note that $\mathrm{v} \in \mathbb{WPG}(\Sigma(\Pi(\Sigma(\Gamma, \{N\}), P), \{ \Sigma(\Gamma, \{ N \}) \}), \{ \Sigma(\Gamma, \{ N \}) \} \{ \mathrm{p} \})$ is just $\pi_2 \in \mathbb{WPG}(\Pi((\Gamma \mathbin{\&} N), P) \mathbin{\&} (\Gamma \mathbin{\&} N ), \Gamma \mathbin{\&} N)$, defined respectively by
\begin{footnotesize}
\begin{align*}
\widetilde{c}_{\mathrm{z}} &\colonequals \big(|\Pi(\Sigma(\Gamma, \{ N \}), P)| \mathbin{\&} |\Sigma(\Gamma, \{ N \})| \stackrel{\mathrm{v}}{\longrightarrow} |\Sigma(\Gamma, \{ N \})| \stackrel{\mathrm{p}}{\longrightarrow} |\Gamma| \stackrel{c_{\mathrm{z}}}{\longrightarrow} |P\{ \mathrm{zero} \}|\big); \\
\widetilde{c}_{\mathrm{s}} &\colonequals \big(|\Pi(\Sigma(\Gamma, \{N\}), P)| \mathbin{\&} |\Sigma(\Gamma, \{N\})| \stackrel{\langle \mathrm{v}, \mathrm{ev}_P \rangle}{\longrightarrow} |\Sigma(\Gamma, \{N\})| \mathbin{\&} |P| \stackrel{c_{\mathrm{s}}}{\longrightarrow} \textstyle |P\{ \mathrm{succ} \bullet \mathrm{p} \}|\big),
\end{align*}
\end{footnotesize}where $\mathrm{ev}_P \in \mathbb{WPG}(\Sigma(\Pi(\Sigma(\Gamma, \{N\}), P), \{ \Sigma(\Gamma, \{N\}) \}), P\{\mathrm{v}\})$ is the \emph{evaluation} \cite{abramsky1997semantics}, i.e., $\mathrm{ev}_P \colonequals \lambda^{-1}(\mathrm{der}_{\Pi(\Sigma(\Gamma, \{ N \}), P)})$. 

Next, we define $\mathrm{pred}_\Gamma \colonequals \langle \mathrm{p}, \mathrm{pd}_\Gamma \rangle \in \mathbb{WPG}(\Sigma(\Gamma, \{ N \}), \Sigma(\Gamma, \{ N \}))$, where $\mathrm{pd}_\Gamma \colonequals \mathrm{Pref}(\{ q_{[1]}q_{[0]}0_{[0]}0_{[1]} \} \cup \{ \, q_{[1]}q_{[0]}(n+1)_{[0]}n_{[1]} \mid n \in \mathbb{N} \, \})^{\mathrm{Even}} \in \mathbb{WPG}(\Sigma(\Gamma, \{ N_{[0]} \}), \{ N_{[1]} \})$, and often omit the subscript $(\_)_\Gamma$ on $\mathrm{pred}$.
Besides, let $\mathrm{cond}_P \in \mathbb{WPG}(\Sigma(\Sigma(\Sigma(\Gamma, \{ N \}), P\{ \mathrm{zero} \bullet \mathrm{p} \}), P\{ \mathrm{succ} \bullet \mathrm{pred} \bullet \mathrm{p} \}), P\{ \mathrm{p} \bullet \mathrm{p} \})$ be the game semantics of \emph{conditionals} \cite{abramsky2000full,hyland2000full,abramsky1999game}: Given any initial move, it asks the number $n$ on $N$, and plays as the dereliction between $P\{ \mathrm{zero} \bullet \mathrm{p} \}$ and $P\{ \mathrm{p} \bullet \mathrm{p} \}$ if $n = 0$, and between $P\{ \mathrm{succ} \bullet \mathrm{pred} \bullet \mathrm{p} \}$ and $P\{ \mathrm{p} \bullet \mathrm{p} \}$ otherwise. 

We then define $\mathcal{F}^{N}_{P}(c_{\mathrm{z}}, c_{\mathrm{s}}) \in \mathbb{WPG}(\Pi(\Sigma(\Gamma, \{N\}), P), \Pi(\Sigma(\Gamma, \{N\}), P))$ by
\begin{equation}
\label{RecursiveCall}
\mathcal{F}^{N}_{P}(c_{\mathrm{z}}, c_{\mathrm{s}}) \colonequals \lambda_{\{ \Sigma(\Gamma, \{ N \}) \}, \{ P\{ \mathrm{v} \} \}} (\mathrm{cond}_P \{ \langle \langle \mathrm{v}, \widetilde{c}_{\mathrm{z}} \rangle, \widetilde{c_{\mathrm{s}}}\{ \langle \mathrm{p}, \mathrm{pred} \bullet \mathrm{v} \rangle \} \rangle \}),
\end{equation}
which is well-defined thanks to the equation $(\mathrm{succ} \bullet \mathrm{v}) \bullet \langle \mathrm{p}, \mathrm{pred} \bullet \mathrm{v} \rangle = \mathrm{succ} \bullet \mathrm{pred} \bullet \mathrm{v} = (\mathrm{succ} \bullet \mathrm{pred} \bullet \mathrm{p}) \bullet \langle \mathrm{v}, \widetilde{c}_{\mathrm{z}} \rangle$.

Finally, let $\mathcal{R}^{N}_{P}(c_{\mathrm{z}}, c_{\mathrm{s}}) \in \mathbb{WPG}(\Sigma(\Gamma, \{N\}), P)$ be the least upper bound of the chain $(\mathcal{R}^{N}_{P}(c_{\mathrm{z}}, c_{\mathrm{s}})_n : \Pi(\Sigma(\Gamma, \{N\}), P))_{n \in \mathbb{N}}$ given by
\begin{mathpar}
\mathcal{R}^{N}_{P}(c_{\mathrm{z}}, c_{\mathrm{s}})_0 \colonequals \{ \boldsymbol{\epsilon} \} \and
\mathcal{R}^{N}_{P}(c_{\mathrm{z}}, c_{\mathrm{s}})_{n+1} \colonequals \mathcal{F}^{N}_{P}(c_{\mathrm{z}}, c_{\mathrm{s}}) \bullet \mathcal{R}^{N}_{P}(c_{\mathrm{z}}, c_{\mathrm{s}})_n.
\end{mathpar}
The difference from the game semantics of \emph{fixed-point combinator} \cite{abramsky2000full,hyland2000full,abramsky1999game} is that the strategy $\mathrm{pred}$ occurring in $\mathcal{F}^{N}_{P}(c_{\mathrm{z}}, c_{\mathrm{s}})$ (\ref{RecursiveCall}) decreases the parameter number on $N$ at every recursive call by $\mathcal{F}^{N}_{P}(c_{\mathrm{z}}, c_{\mathrm{s}})$, so that $\mathcal{R}^{N}_{P}(c_{\mathrm{z}}, c_{\mathrm{s}})$ is total.

\item \textsc{($N$-Comp)} By the definition of $\mathcal{R}^{N}_{P}(c_{\mathrm{z}}, c_{\mathrm{s}})$, we clearly have
\begin{align*}
\mathcal{R}^{N}_{P}(c_{\mathrm{z}}, c_{\mathrm{s}}) \{ \mathrm{zero} \} &= c_{\mathrm{z}} \in \mathbb{WPG}(\Gamma, P \{ \mathrm{zero} \}); \\
\mathcal{R}^{N}_{P}(c_{\mathrm{z}}, c_{\mathrm{s}}) \{ \mathrm{succ} \} &= c_{\mathrm{s}} \{ \langle \mathrm{der}_{\Sigma(\Gamma, \{N\})}, \mathcal{R}^{N}_{P}(c_{\mathrm{z}}, c_{\mathrm{s}}) \rangle \} \in \mathbb{WPG}(\Sigma(\Gamma, \{N\}), P\{ \mathrm{succ} \}).
\end{align*}

\item \textsc{($N$-Subst)} It is clear that $\{N\}_{\oc \Gamma}\{ \phi \} = \{N\}_{\oc \Delta}$ holds.

\item \textsc{($\mathcal{R}^{N}$-Subst)}
Finally, by the definition of $\mathcal{R}^{N}_{P}(c_{\mathrm{z}}, c_{\mathrm{s}})$, we clearly have
\begin{equation*}
\mathcal{R}^{N}_{P}(c_{\mathrm{z}}, c_{\mathrm{s}}) \{ \phi_N^+ \} = \mathcal{R}^{N}_{P\{ \phi_N^+ \}}(c_{\mathrm{z}}\{ \phi \}, c_{\mathrm{s}}\{ \phi_{N, P}^{++} \})
\end{equation*}
\end{itemize}
(or alternatively show $\mathcal{R}^{N}_{P}(c_{\mathrm{z}}, c_{\mathrm{s}})_n \{ \phi_N^+ \} = \mathcal{R}^{N}_{P\{ \phi_N^+ \}}(c_{\mathrm{z}}\{ \phi \}, c_{\mathrm{s}}\{ \phi_{N, P}^{++} \})_n$ for all $n \in \mathbb{N}$ by induction on $n$ so that the above equation holds), which completes the proof.
\end{proof}

\subsubsection{Game semantics of One- and Zero-types}
Next, we interpret \emph{One}- and \emph{Zero-types}.
Since it is trivial to interpret these types, we only sketch the proof. 
See \cite{hofmann1997syntax} for the semantic type formers for these types.
 
\begin{theorem}[Game semantics of One- and Zero-types]
\label{ThmGameSemanticsOfOneAndZeroTypes}
$\mathbb{WPG}$ supports the semantic type formers for One- (in the strict sense) and Zero-types.
\end{theorem}
\begin{proof}[Proof (sketch)]
We model One- and Zero-types by the constant dependent p-games at the terminal p-game $T$ and the empty p-game $\boldsymbol{0}$, respectively (Example~\ref{ExPGames}).

The only point on One-type is that there is only the trivial one $\{ \boldsymbol{\epsilon} \} \in \mathbb{WPG}(\Gamma, \{ T \})$ for any $\Gamma \in \mathbb{WPG}$, so that we can interpret all the rules of One-type.

Finally, given $A \in \mathscr{D}^{\mathrm{w}}(\Sigma(\Gamma, \{ \boldsymbol{0} \}))$ and $\zeta \in \mathbb{WPG}(\Gamma, \{ \boldsymbol{0} \})$, the only point on Zero-type is that for the elimination rule we may obtain $\mathcal{R}^{\boldsymbol{0}}_A(\zeta) \in \mathbb{WPG}(\Gamma, A\{ \langle \mathrm{id}_\Gamma, \zeta \rangle \})$ from $\zeta$ by case distinction: Let $\mathcal{R}^{\boldsymbol{0}}_A(\zeta)$ be the one obtained from $\zeta$ by replacing the unique first move $q$ of $\boldsymbol{0}$ in $\zeta$ with those of $|A|$ if any, and $\mathcal{R}^{\boldsymbol{0}}_A(\zeta) \colonequals \{ \boldsymbol{\epsilon} \}$ otherwise.
\end{proof}

\subsubsection{Game semantics of Id-types}
Let us proceed to present our game semantics of \emph{Id-types}.
Again, we first review the general semantic type former for Id-types in an arbitrary CwF:
\begin{definition}[CwFs with Id-types \cite{hofmann1997syntax}]
A CwF $\mathcal{C}$ \emph{\bfseries supports Id-types} if
\begin{itemize}

\item \textsc{(Id-Form)} Given $\Gamma \in \mathcal{C}$ and $A \in \mathrm{Ty}(\Gamma)$, there is a type $\mathrm{Id}_A \in \mathrm{Ty}(\Gamma . A . A^+)$, where $A^+ \colonequals A\{\mathrm{p}_A \} \in \mathrm{Ty}(\Gamma . A)$;

\item \textsc{(Id-Intro)} There is a morphism $\mathrm{Refl}_A : \Gamma . A \rightarrow \Gamma . A . A^+ . \mathrm{Id}_A$ in $\mathcal{C}$ that satisfies $\mathrm{p}_{\mathrm{Id}_A} \circ \mathrm{Refl}_A = \overline{\mathrm{v}_A} : \Gamma . A \rightarrow \Gamma . A . A^+$, where $\overline{\mathrm{v}_A} \colonequals \langle \mathrm{id}_{\Gamma . A}, \mathrm{v}_{A} \rangle$;

\item \textsc{(Id-Elim)} Given $B \in \mathrm{Ty}(\Gamma . A . A^+ . \mathrm{Id}_A)$ and $\beta \in \mathrm{Tm}(\Gamma . A, B\{\mathrm{Refl}_A\})$, there is a term $\mathcal{R}^{\mathrm{Id}}_{A,B}(\beta) \in \mathrm{Tm}(\Gamma . A . A^+ . \mathrm{Id}_A, B)$;

\item \textsc{(Id-Comp)} $\mathcal{R}^{\mathrm{Id}}_{A,B}(\beta)\{\mathrm{Refl}_A\} = \beta$;

\item \textsc{(Id-Subst)} $\mathrm{Id}_A\{\phi_{A, A^+}^{++}\} = \mathrm{Id}_{A\{\phi\}} \in \mathrm{Ty}(\Delta . A\{\phi\} . A\{\phi\}^+)$ for all $\Delta \in \mathcal{C}$ and $\phi : \Delta \to \Gamma$ in $\mathcal{C}$, where $\phi_A^+ \colonequals \langle \phi \circ \mathrm{p}_{A\{ \phi \}}, \mathrm{v}_{A\{\phi\}} \rangle_A : \Delta . A\{\phi\} \to \Gamma . A$ and $\phi_{A, A^+}^{++} \colonequals (\phi_A^+)^+_{A^+} : \Delta . A\{\phi\} . A^+\{\phi_A^+\} \to \Gamma . A . A^+$;

\item \textsc{(Refl-Subst)} $\mathrm{Refl}_A \circ \phi_A^+ = \phi_{A, A^+, \mathrm{Id}_A}^{+++} \circ \mathrm{Refl}_{A\{\phi\}} : \Delta . A\{\phi\} \to \Gamma . A . A^+ . \mathrm{Id}_A$, where $\phi_{A, A^+, \mathrm{Id}_A}^{+++} \colonequals (\phi_{A, A^+}^{++})^+_{\mathrm{Id}_A} : \Delta . A\{\phi\} . A^+\{\phi^+\} . \mathrm{Id}_{A\{\phi\}} \to \Gamma . A . A^+ . \mathrm{Id}_A$;

\item \textsc{($\mathcal{R}^{\mathrm{Id}}$-Subst)} $\mathcal{R}^{\mathrm{Id}}_{A,B}(\beta)\{\phi_{A, A^+, \mathrm{Id}_A}^{+++}\} = \mathcal{R}^{\mathrm{Id}}_{A\{\phi\}, B\{\phi_{A, A^+, \mathrm{Id}_A}^{+++}\}}(\beta\{\phi_A^+\})$.

\end{itemize}
\end{definition}

Then, we present our game semantics of Id-types, which is essentially the same as the interpretation of Id-types by Abramsky et al. \cite{abramsky2015games,vakar2018game}:
\begin{theorem}[Game semantics of Id-types]
\label{LemGameSemanticIdTypes}
$\mathbb{WPG}$ supports Id-types.
\end{theorem}
\begin{proof}
Let $\Gamma \in \mathbb{WPG}$, $A \in \mathscr{D}^{\mathrm{w}}(\Gamma)$ and $B \in \mathscr{D}^{\mathrm{w}}(\Sigma(\Sigma(\Sigma(\Gamma, A), A^+), \mathrm{Id}_A))$.
\begin{itemize}

\item \textsc{(Id-Form)} Let $\boldsymbol{1}$ be the flat game $\mathrm{flat}(\{ \mathbin{\surd} \})$ (Example~\ref{ExamplesOfGames}), where $\{ \mathbin{\surd} \}$ is an arbitrarily fixed singleton set.
Define $\mathrm{Id}_A \in \mathscr{D}^{\mathrm{w}}(\Sigma(\Sigma(\Gamma, A), A^+))$ by $|\mathrm{Id}_A| \colonequals \boldsymbol{1}$ and $\mathrm{Id}_A(\langle \langle \gamma_0, \alpha_0 \rangle, \alpha_0' \rangle^\dagger) \colonequals \begin{cases} (\boldsymbol{1}, \kappa_{\boldsymbol{1}}) &\text{if $\alpha_0 = \alpha_0'$;} \\ (\boldsymbol{1}, \kappa_{\boldsymbol{0}}) &\text{otherwise,} \end{cases}$ for all $\langle \langle \gamma_0, \alpha_0 \rangle, \alpha_0' \rangle^\dagger \in \mathbb{WPG}(\oc \Sigma(\Sigma(\Gamma, A), A^+))$, where $\kappa_{X}$ is the constant family at $X$ ($X = \boldsymbol{1}, \boldsymbol{0}$);

\item \textsc{(Id-Intro)} Define $\mathrm{Refl}_A \colonequals \langle \overline{\mathrm{v}_A}, \mathrm{refl}_A \rangle \in \mathbb{WPG}(\Sigma(\Gamma, A), \Sigma(\Sigma(\Sigma(\Gamma, A), A^+), \mathrm{Id}_A))$, where $\mathrm{refl}_A \in \mathbb{WPG}(\Sigma(\Gamma, A), \mathrm{Id}_A\{ \overline{\mathrm{v}_A} \})$ is $\underline{\surd} : \boldsymbol{1}$ (Example~\ref{ExamplesOfGames}) up to `tags.'
Note that $\mathrm{refl}_A$ is well-defined since its codomain is always the game $\boldsymbol{1}$.

\item \textsc{(Id-Elim)} Given $\beta \in \mathbb{WPG}(\Sigma(\Gamma, A), B\{\mathrm{Refl}_A\})$, let us construct the strategy $\mathcal{R}^{\mathrm{Id}}_{A,B}(\beta) \in \mathbb{WPG}(\Sigma(\Sigma(\Sigma(\Gamma, A_{[0]}), A_{[1]}^+), \mathrm{Id}_A), B)$ such that given any initial move in $B$ it makes the move $q$ in $\mathrm{Id}_A$, and if O plays there by $\surd$, then it plays as $\beta$ between $\Sigma(\Gamma, A_{[0]})$ and $B$ in the rest of the play.
Since there is the evident bijection $\mathbb{WPG}(\Sigma(\Gamma, A)) \cong \mathbb{WPG}(\Sigma(\Sigma(\Sigma(\Gamma, A), A^+), \mathrm{Id}_A))$ thanks to the definition of $\mathrm{Id}_A$, $\mathcal{R}^{\mathrm{Id}}_{A,B}(\beta)$ can play as $\beta$ after the initial two moves.

\end{itemize}

Having gone through the detailed proofs of Theorems~\ref{ThmGameSemanticsOfPiType} and \ref{ThmGameSemanticsOfSigmaType}, it is a routine to verify Id-Comp, Id-Subst and Refl-Subst. 
Also, essentially the same interpretation of Id-types is given by Abramsky et al. \cite{abramsky2015games,vakar2018game}.
Hence, we leave it to the reader. 
\end{proof}

\subsection{Intensionality}
\label{Intensionality}
We next show the \emph{intensionality} of our game semantics.
In this section, we focus on the components of $\mathbb{WPG}$ \emph{definable} by MLTT with One-, Zero-, N-, Pi-, Sigma- and Id-types (see \cite{hofmann1997syntax} for the interpretation) and formulate type-theoretic principles in $\mathbb{WPG}$.
In the following, let us also fix an arbitrary CwF $\mathcal{C}$ that supports One-, Zero-, N-, Pi-, Sigma- and Id-types, and assume $\Gamma \in \mathcal{C}$, $A \in \mathrm{Ty}(\Gamma)$ and $B \in \mathrm{Ty}(\Gamma . A)$.

\paragraph{Equality reflection.}
\emph{Equality reflection} \cite{dybjer2016intuitionistic} states that propositionally equal terms are judgmentally equal: $\alpha = \alpha' \in \mathrm{Tm}(\Gamma, A)$ if $\mathrm{Tm}(\Gamma, \mathrm{Id}_A\{\langle \langle \mathrm{id}, \alpha \rangle, \alpha' \rangle\})$ is inhabited.

Equality reflection \emph{fails} in our game semantics: $N \mathbin{\&} \boldsymbol{0} \stackrel{\pi_1}{\rightarrow} N$ and $N \mathbin{\&} \boldsymbol{0} \stackrel{\mathrm{succ}}{\rightarrow} N \mathbin{\&} \boldsymbol{0} \stackrel{\pi_1}{\rightarrow} N$ are different definable terms in $\mathbb{WPG}(\Sigma(N, \{\boldsymbol{0}\}), \{ N \})$, but the elimination rule on Zero-type gives a definable one $\rho \in \mathbb{WPG}(\Sigma(N, \{\boldsymbol{0}\}), \mathrm{Id}_N\{ \langle \langle \mathrm{id}, \pi_1 \bullet \mathrm{succ} \rangle, \pi_1 \rangle \})$.

Recall that $\rho$ is the strategy that makes, after the initial move $q$ on the codomain is made by O, the move $q$ in $\boldsymbol{0}$ on the domain (see the proof of Theorem~\ref{ThmGameSemanticsOfOneAndZeroTypes}).


\paragraph{Function extensionality.} 
\emph{Function extensionality} \cite{hottbook} postulates that if terms of a Pi-type are propositionally equal as a function, then they are propositionally equal, i.e., the inhabitance of the type $\Pi(A, \mathrm{Id}_B\{ \langle \langle \mathrm{id}, \mathrm{App}(\psi\{ \mathrm{p} \}, \mathrm{v}) \rangle, \mathrm{App}(\psi'\{ \mathrm{p} \}, \mathrm{v}) \rangle \}) \Rightarrow \mathrm{Id}_{\Pi(A, B)}\{ \langle \langle \mathrm{id}, \psi \rangle, \psi' \rangle \} \in \mathrm{Ty}(\Gamma)$ for all terms $\psi, \psi' \in \mathrm{Tm}(\Gamma, \Pi(A, B))$.

Our game semantics \emph{refutes} function extensionality since terms $\psi \in \mathbb{WPG}(\Gamma, A)$ are in general not completely specified by their extensions $\gamma : \oc \Gamma \mapsto \phi \circ \gamma$. 
To see this point closely, consider, e.g., the terms $\underline{0}, \underline{0}' \in \mathbb{WPG}(T, \Pi(\{ N_{[0]} \}, \{ N_{[1]} \}))$ defined by $\underline{0} \colonequals \mathrm{Pref}(\{ \, q_{[1]}q_{[0]}n_{[0]}0_{[1]} \mid n \in \mathbb{N} \, \})^{\mathrm{Even}}$ and $\underline{0}' \colonequals \mathrm{Pref}(\{ q_{[1]}0_{[1]} \})^{\mathrm{Even}}$.
They are definable in MLTT by the introduction and the elimination rules of N-type.
Then, for the reason mentioned above, the function extensionality on these terms fails. 

\paragraph{Criteria of intensionality.}
There are Streicher's \emph{Criteria of Intensionality} \cite{streicher1993investigations}:
\begin{enumerate}

\item $\mathrm{v}\{ \mathrm{p} \circ \mathrm{p} \} \neq \mathrm{v}\{ \mathrm{p} \} \in \mathrm{Tm}(\Gamma . A . A^+ . \mathrm{Id}_A, A\{ \mathrm{p} \circ \mathrm{p} \circ \mathrm{p} \})$;

\item $B\{ \langle \mathrm{p} \circ \mathrm{p} \circ \mathrm{p}, \mathrm{v}\{ \mathrm{p} \circ \mathrm{p} \rangle \} \} \neq B\{ \langle \mathrm{p} \circ \mathrm{p} \circ \mathrm{p}, \mathrm{v}\{ \mathrm{p} \} \rangle \} \in \mathrm{Tm}(\Gamma . A . A^+ . \mathrm{Id}_A)$;

\item If $\mathrm{Tm}(T, \mathrm{Id}_S\{ \langle \langle \mathrm{id}_T, \sigma \rangle, \sigma' \rangle \})$ is inhabited, then $\sigma = \sigma' \in \mathrm{Tm}(T, S)$,
\end{enumerate} 
for all $\Gamma \in \mathcal{C}$, $A \in \mathrm{Ty}(\Gamma)$, $B \in \mathrm{Ty}(\Gamma . A)$, $S \in \mathrm{Ty}(T)$ and $\sigma, \sigma' \in \mathrm{Tm}(T, S)$.

Our game semantics \emph{satisfies} the first criterion.
For example, let $\Gamma \colonequals T$, $A \colonequals \{ N \}_{\oc T}$ and $B \colonequals \{ N \}_{\oc \Sigma(T, \{ N \})}$.
Then, $\mathrm{v}\{ \mathrm{p} \circ \mathrm{p} \}$ is $T \mathbin{\&} N \mathbin{\&} N \mathbin{\&} \mathrm{Id}_{\{ N \}} \stackrel{\pi_1}{\rightarrow} T \mathbin{\&} N \mathbin{\&} N  \stackrel{\pi_1}{\rightarrow} T \mathbin{\&} N \stackrel{\pi_2}{\rightarrow} N$, and $\mathrm{v}\{ \mathrm{p} \}$ is $T \mathbin{\&} N \mathbin{\&} N \mathbin{\&} \mathrm{Id}_{\{ N \}} \stackrel{\pi_1}{\rightarrow} T \mathbin{\&} N \mathbin{\&} N \stackrel{\pi_2}{\rightarrow} N$. 
They are distinct.

In contrast, our game semantics does \emph{not} satisfy the second criterion as dependent p-games are indexed by \emph{winning} strategies: Any $\tau \in \mathbb{WPG}(\oc \Sigma(\Sigma(\Sigma(\Gamma, A), A^+), \mathrm{Id}_A))$ must be of the form $\tau = \langle \langle \langle \gamma_0, \alpha_0 \rangle, \alpha_0' \rangle, \underline{\surd} \rangle^\dagger$ such that $\alpha_0 = \alpha_0'$.
In other words, our game semantics of dependent types is \emph{extensional} as opposed to that of terms.

Finally, our model \emph{satisfies} the third criterion: Any $\rho \in \mathbb{WPG}(T, \mathrm{Id}_S\{ \langle \langle \mathrm{id}_T, \sigma \rangle, \sigma' \rangle \})$ must be the strategy $\underline{\surd}$ (up to `tags') since it is winning, and hence $\sigma = \sigma'$.


\subsection{Independence of Markov's principle}
\label{Independence}
\emph{Markov's principle (MP)} \cite{markov1962constructive} is a subtle principle in constructive mathematics and computability theory.
In fact, it depends on the school of constructive mathematics whether MP should be regarded as \emph{constructive}.
Roughly, MP postulates that if it is impossible that there is no natural number $n \in \mathbb{N}$ such that $f(n) = 0$ for a given function $f : \mathbb{N} \rightarrow \mathbb{N}$, then there is a natural number $n' \in \mathbb{N}$ such that $f(n') = 0$.

We can formulate MP in MLTT as the type
\begin{equation}
\label{MP}
\mathsf{f : N \Rightarrow N \vdash \neg \neg \Sigma_{x:N}Id_N(f(x), \underline{0}) \Rightarrow \Sigma_{y:N}Id_N(f(y), \underline{0}) \ type},
\end{equation}
where $\mathsf{\neg A}$ abbreviates $\mathsf{A \Rightarrow \boldsymbol{0}}$ for any type $\mathsf{A}$, and $\mathsf{0}$, $\mathsf{N}$, $\mathsf{\Rightarrow}$, $\mathsf{\Pi}$, $\mathsf{\Sigma}$ and $\mathsf{Id_N}$ are Zero-, N-, implication, Pi-, Sigma- and Id-types, respectively.

Mannaa and Coquand \cite{mannaa2017independence} have shown that MP is \emph{independent} from MLTT; i.e., there is no term of the type (\ref{MP}) in MLTT.
Their proof of this result is \emph{syntactic}.

However, even representative computational models of MLTT such as the \emph{effective topos} \cite{hyland1982effective} validate MP.
Besides, it is easy to see that the model by Blot and Laird \cite{blot2018extensional} validates MP too (though this is hardly surprising as their model admits classical reasonings \cite[\S 8]{blot2018extensional}).
In other words, there is a \emph{gap} between MLTT and these models.

In this context, we illustrate the \emph{precision} of our game semantics of MLTT by:
\begin{corollary}[Game semantics refutes Markov's principle]
\label{CorIndependence}
There is no game-semantic term on the interpretation of (\ref{MP}) in $\mathbb{WPG}$.
Hence, the game semantics of MLTT in $\mathbb{WPG}$ (\S\ref{GameSemanticCwF}--\ref{GameSemanticTypeFormers}) implies that MP is independent from MLTT.
\end{corollary}
\begin{proof}[Proof (sketch)]
Assume for a contradiction that there is a term on the interpretation of (\ref{MP}) in $\mathbb{WPG}$, which is the p-game
\begin{tiny}
\begin{equation}
\label{GameSemanticsOfMP}
\Pi\bigg(N_{[0]} \Rightarrow N_{[1]}, \Big(\big(\Sigma(N_{[2]}, \mathrm{Id}_N\{ \langle \mathrm{App}(\pi_1, \pi_2), \underline{0} \rangle \}_{[3]}) \Rightarrow \boldsymbol{0}_{[4]}\big) \Rightarrow \boldsymbol{0}_{[5]}\Big) \Rightarrow \Sigma(N_{[6]}, \mathrm{Id}_N\{ \langle \mathrm{App}(\pi_1, \pi_2), \underline{0} \rangle \}_{[7]})\{ \mathrm{p} \} \bigg),
\end{equation}
\end{tiny}where we omit the terminal p-game $T$ and the bracket $\{ \_ \}$ for constant dependent p-games; e.g., we write $N$ for $T . \{ N \}$.
We call $\Sigma(N_{[6]}, \mathrm{Id}_N\{ \langle \mathrm{App}(\pi_1, \pi_2), \underline{0} \rangle \}_{[7]})$ the \emph{codomain}, and $N_{[0]} \Rightarrow N_{[1]}$ and $\big(\Sigma(N_{[2]}, \mathrm{Id}_N\{ \langle \mathrm{App}(\pi_1, \pi_2), \underline{0} \rangle \}_{[3]}) \Rightarrow \boldsymbol{0}_{[4]}\big) \Rightarrow \boldsymbol{0}_{[5]}$ the outer and the inner \emph{domains} of (\ref{GameSemanticsOfMP}).
We write $\langle \phi, \psi \rangle$ for the assumed term.

Assume first that there are \emph{total} input strategies $\varphi^\dagger : \oc (N_{[0]} \Rightarrow N_{[1]})$ and $\langle \underline{n}, \underline{\surd} \rangle^\dagger : \oc \Sigma(N_{[2]}, \mathrm{Id}_N\{ \langle \mathrm{App}(\pi_1, \pi_2), \underline{0} \rangle \}_{[3]})$ on which $\phi$ eventually makes a P-move $n'_\varphi$ in $N_{[6]}$ when O begins a play in $N_{[6]}$.
If $\varphi \bullet \underline{n'_\varphi} \neq \underline{0}$ for \emph{some} of these $\varphi$ and $n$, and O plays by them, then $\mathrm{Id}_N\{ \langle \mathrm{App}(\pi_1, \pi_2), \underline{0} \rangle \}_{[7]}$ can be the empty p-game $\boldsymbol{0}$, and $\psi$ plays on the domains \emph{forever}, contradicting its \emph{noetherianity}.
If $\varphi \bullet \underline{n'_\varphi} = \underline{0}$ for \emph{all} $\varphi$ and $n$, then $\phi$ is strict since otherwise $n'_\varphi$ is the same even when O slightly changes $\varphi$ so that $\varphi \bullet \underline{n'_\varphi} \neq \underline{0}$, a contradiction; hence, $\phi$ is strict and answers the question in $N_{[6]}$ yet with no answer to the question in $\boldsymbol{0}_{[4]}$, contradicting the \emph{well-bracketing} of $\phi$.

Thus, given any total inputs $\varphi^\dagger$ and $\langle \underline{n}, \underline{\surd} \rangle^\dagger$, $\phi$ does not make a P-move in $N_{[6]}$.
Similarly to the case of $\psi$, however, this contradicts the \emph{noetherianity} of $\phi$.
\end{proof}

\subsection{Game semantics of subtyping on dependent types}
\label{GameSemanticsOfSubtyping}
Finally, we show that p-games enable us to interpret \emph{subtyping on dependent types}. 
The subtyping relation between types is something like the subset relation between sets, and it is mainly to ensure that if $\alpha$ is a term of a type $A$, and $A$ is a subtype of a type $B$, then $\mathrm{coe}(\alpha)$ is a term of $B$ that computes in the same way as $\alpha$, where $\mathrm{coe}$ is an operation on terms, called \emph{coercion}.
Therefore, subtyping has a pragmatic importance since it ensures \emph{reusability} of terms in different types. 

Let us first formulate a subtyping relation between dependent types in CwFs:
\begin{definition}[CwFs with subtyping]
\label{DefCwFwWithSubtyping}
A CwF $\mathcal{C}$ that supports One-, Pi- and Sigma-types \emph{\bfseries supports subtyping} (on One-, Pi- and Sigma-types) if it is equipped with a partial order $\trianglelefteqslant_\Gamma$ on $\mathrm{Ty}(\Gamma)$ for each $\Gamma \in \mathcal{C}$ that satisfies
\begin{enumerate}

\item $\big(A \trianglelefteqslant_\Gamma A' \wedge \alpha \in \mathrm{Tm}(\Gamma, A)\big) \Rightarrow \alpha \in \mathrm{Tm}(\Gamma, A')$;


\item $\big(A \trianglelefteqslant_\Gamma A' \wedge \phi \in \mathcal{C}(\Delta, \Gamma)\big) \Rightarrow A\{ \phi \} \trianglelefteqslant_\Delta A'\{ \phi \}$;

\item $\forall \Gamma \in \mathcal{C}, A \in \mathrm{Ty}(\Gamma) . \, A \trianglelefteqslant_\Gamma \boldsymbol{1}$;

\item $\big(A \trianglelefteqslant_\Gamma A' \wedge B \trianglelefteqslant_{\Gamma . A'} B'\big) \Rightarrow \big(\Pi(A', B) \trianglelefteqslant_\Gamma \Pi(A, B') \wedge \Sigma(A, B) \trianglelefteqslant_\Gamma \Sigma(A', B')\big)$.

\end{enumerate}
\end{definition}

The first axiom requires the reusability of terms. 
In addition, for the \emph{compositional} nature of denotational semantics, the other axioms require that constructions on types preserve $\trianglelefteqslant_\Gamma$.
Note that the last axiom tacitly assumes $B, B' \in \mathrm{Ty}(\Gamma . A)$, and the order between $A$ and $A'$ is \emph{flipped} in Pi-types as fewer morphisms can take more inputs \cite[p.~419]{amadio1998domains}. 
We define game semantics of subtyping based on Definition~\ref{DefLivenessOrdering}:
\begin{definition}[Predicate liveness ordering]
\label{DefPredicateLivenessOrdering}
The \emph{\bfseries predicate (p-) liveness ordering} is a partial order $\trianglelefteqslant$ between p-games $\Gamma$ and $\Delta$ defined by $\Gamma \trianglelefteqslant \Delta \ratio \Leftrightarrow |\Gamma| = |\Delta| \wedge \forall \gamma : |\Gamma| . \, \Gamma(\gamma) \preccurlyeq \Delta(\Gamma)$, and lifted to the one $\trianglelefteqslant_\Gamma$ between linearly dependent p-games $L$ and $R$ over $\Gamma$ by $L \trianglelefteqslant_\Gamma R \ratio \Leftrightarrow |L| = |R| \wedge \forall \gamma : \Gamma . \, L(\gamma) \trianglelefteqslant R(\gamma)$.
\end{definition}

It is easy to see that the p-liveness orderings $\trianglelefteqslant$ and $\trianglelefteqslant_\Gamma$ are partial orders since so is the liveness ordering $\preccurlyeq$ \cite[Theorem~9]{chroboczek2000game}.
Besides, the p-liveness ordering $\trianglelefteqslant_\Gamma$ clearly satisfies the second and the third axioms of Definition~\ref{DefCwFwWithSubtyping}, where $\boldsymbol{1} = \{ T \}$.

Moreover, if $\gamma : \Gamma$ and $\Gamma \trianglelefteqslant \Delta$, then $\gamma : \Delta$ since $\gamma : |\Gamma| = |\Delta|$ and $\overline{\gamma}_{|\Delta|} = \overline{\gamma}_{|\Gamma|} \preccurlyeq \Gamma(\gamma) \preccurlyeq \Delta(\gamma)$.
This observation shows that $\trianglelefteqslant_\Gamma$ also satisfies the first axiom:
\begin{lemma}[Preservation of linear typing under predicate liveness ordering]
\label{LemPreservationOfTyping}
Let $L$ and $R$ be linearly dependent p-games over a p-game $\Gamma$ such that $L \trianglelefteqslant_\Gamma R$.
\begin{enumerate}

\item If $\psi : \Pi_\ell(\Gamma, L)$, then $\psi : \Pi_\ell(\Gamma, R)$;

\item If $\psi : \Pi_\ell(\Gamma, L)$ is winning (resp. w.b.), then so is $\psi : \Pi_\ell(\Gamma, R)$.


\end{enumerate}
\end{lemma}

Further, $\trianglelefteqslant_\Gamma$ satisfies the last axiom as well by:
\begin{lemma}[Preservation of predicate liveness ordering]
\label{LemPreservationOfPredicateLivenessOrdering}
If $\Delta \trianglelefteqslant \Delta'$ and $\Gamma \trianglelefteqslant \Gamma'$, then $\Delta \otimes \Gamma \trianglelefteqslant \Delta' \otimes \Gamma'$, $\Delta \mathbin{\&} \Gamma \trianglelefteqslant \Delta' \mathbin{\&} \Gamma'$, $\oc \Delta \trianglelefteqslant \oc \Delta'$ and $\Delta' \multimap \Gamma \trianglelefteqslant \Delta \multimap \Gamma'$.
Moreover, if $A \trianglelefteqslant_\Gamma A'$ and $B \trianglelefteqslant_{\Gamma . A'} B'$, then $\Pi(A', B) \trianglelefteqslant_\Gamma \Pi(A, B')$ and $\Sigma(A, B) \trianglelefteqslant_\Gamma \Sigma(A', B')$.
(N.b., strictly, the pair $(|B|, \|B\| \upharpoonright \Gamma . A)$, where $\|B\| \upharpoonright \Gamma . A$ is the restriction of $\| B \|$ to $\mathrm{st}(\Gamma . A) \subseteq \mathrm{st}(\Gamma . A')$, not $B$ itself, is an element of $\mathscr{D}(\Gamma . A)$, and similarly for $B'$.)
\end{lemma}
\begin{proof}
For the first part, we focus on linear implication $\multimap$ since the cases of the other constructions are simpler.
First, observe $|\Delta' \multimap \Gamma| = |\Delta'| \multimap |\Gamma| = |\Delta| \multimap |\Gamma'| = |\Delta \multimap \Gamma'|$.
Next, let $\phi : |\Delta' \multimap \Gamma|$; it remains to show $(\Delta' \multimap \Gamma)(\phi) \preccurlyeq (\Delta \multimap \Gamma')(\phi)$, but it follows from $\Delta \trianglelefteqslant \Delta'$ and $\Gamma \trianglelefteqslant \Gamma'$ by induction on the lengths of positions, where the symmetry (between O and P) of the liveness ordering $\preccurlyeq$ is crucial. 

Finally, we show the second part by essentially the same way as the first part.
\end{proof}

Hence, we have finally shown:
\begin{corollary}[Game semantics of subtyping]
$\mathbb{WPG}$ supports subtyping.
\end{corollary}

\section*{Acknowledgements}
The author acknowledges the financial support from the Funai Foundation, and he is grateful to fruitful discussions with Samson Abramsky and Thierry Coquand.


\bibliographystyle{bmc-mathphys} 
\bibliography{CategoricalLogic,GamesAndStrategies,RecursionTheory,PCF,TypeTheoriesAndProgrammingLanguages,HoTT,LinearLogic}      

\end{document}